\newcommand{\pbaddress}{biran@math.tau.ac.il}
\newcommand{\ocaddress}{cornea@dms.umontreal.ca}

\documentclass[12pt]{amsart}

\usepackage[colorlinks=true, urlcolor=blue,bookmarks=true,bookmarksopen=true,
citecolor=blue
]{hyperref}

\usepackage{color}

\usepackage{graphicx}

\usepackage{epsfig}
\usepackage{amscd}
\usepackage[mathscr]{eucal}
\usepackage{amssymb}
\usepackage{amsxtra}
\usepackage{amsmath}
\usepackage[all]{xy}
\usepackage{enumerate}
\usepackage{mathrsfs}

\theoremstyle{plain}

\newtheorem{thm}{Theorem}[subsection]
\newtheorem{thmspec}[thm]{Theorem$^*$}
\newtheorem{cor}[thm]{Corollary}
\newtheorem{corspec}[thm]{Corollary$^*$}

\newtheorem{prop}[thm]{Proposition}

\theoremstyle{definition}

\newtheorem{assumption}[thm]{Assumption}

\theoremstyle{remark}
\newtheorem{rem}[thm]{Remark}

\newtheorem{ex}[thm]{Example}

\theoremstyle{plain}

\newcommand{\Qed}{\hfill \qedsymbol \medskip}

\newcommand{\Id}{{{\mathchoice {\rm 1\mskip-4mu l} {\rm 1\mskip-4mu l}
      {\rm 1\mskip-4.5mu l} {\rm 1\mskip-5mu l}}}}

\newcommand{\R}{\mathbb{R}}
\newcommand{\Z}{\mathbb{Z}}

\newcommand{\C}{\mathbb{C}}

\newcommand{\La}{\Lambda}

\newcommand{\Crit}{\rm{Crit\/}}

\newcommand{\mubar}{{\bar{\mu}}}

\begin{document}

\title{Lagrangian Quantum Homology}
\date{\today}

\thanks{The first author was partially supported by the ISRAEL SCIENCE
FOUNDATION (grant No. 1227/06 *); the second author was supported by
an NSERC Discovery grant and a FQRNT Group Research grant}

\author{Paul Biran and Octav Cornea} \address{Paul Biran, School of
  Mathematical Sciences, Tel-Aviv University, Ra mat-Aviv, Tel-Aviv
  69978, Israel} \email{\pbaddress} \address{Octav Cornea, Department
  of Mathematics and Statistics University of Montreal C.P. 6128 Succ.
  Centre-Ville Montreal, QC H3C 3J7, Canada} \email{\ocaddress}

\dedicatory{\large Dedicated to Yasha Eliashberg on the occasion of
  his 60'th birthday}

\bibliographystyle{plain}

%----------------------------------------------------------------------
%
% Abstract
%
%\begin{abstract}
% ...
%\end{abstract}

\maketitle

%----------------------------------------------------------------------
%
% Beginning of text
%
\tableofcontents

\section{Introduction}

The present paper is mainly a survey of our work
\cite{Bi-Co:qrel-long} and \cite{Bi-Co:rigidity} but it also contains
the announcement of some new results. Its main purpose is to present
an accessible introduction to a technique allowing efficient
calculations in Lagrangian Floer theory.

This technique is based on counting elements in $0$-dimensional moduli
spaces formed by configurations consisting of pseudo-holomorphic disks
joined together by Morse trajectories.  In some form, such
configurations have first appeared in the work of Oh in
\cite{Oh:relative} and have been used in a more general setting in
\cite{Cor-La:Cluster-1}.  There are two basic reasons why such
configurations are natural in this context.

First, if one tries to develop quantum homology and additional
operations in the Lagrangian setting one needs to introduce a
mechanism which compensates for the bubbling of disks as this is a
co-dimension one phenomenon.  The second reason is that the lens
through which the topology of manifolds is understood algebraically is
algebraic topology and this, via classical Morse theory, can be seen
as the combinatorics of Morse trajectories. It is thus completely
natural to approach symplectic topology and the topology of
Lagrangians via the combinatorics of, so-called, pearly trajectories -
schematically, these are just Morse trajectories with a finite number
of points replaced by $J$-holomorphic curves.

As will be discussed below, in the case of monotone Lagrangians with
minimal Maslov number at least $2$, this idea can be fully implemented
while dealing with the technical transversality issues in a relatively
elementary way. The end result is a machinery which is effective in
computations and which leads to several applications.

The paper is structured as follows. The second section reviews the
construction of the quantum homology $QH(L)$ of a monotone Lagrangian
$L\subset (M^{2n},\omega)$ as an algebra over the quantum homology
$QH(M)$ of the ambient manifold.  The main ideas necessary to prove
the properties of $QH(L)$ are described in \S\ref{s:main-ideas-proof}.
In \S\ref{s:further} some additional useful structures are presented.
We emphasize that in our (monotone) setting, as is well-known since
the work of Oh \cite{Oh:HF1}, the Floer homology $HF(L,L)$ is well
defined. Moreover, with appropriate coefficients, $QH(L)$ is
isomorphic to $HF(L,L)$ and some of the structures that we define in
``pearly'' terms for $QH(L)$ are identified by this isomorphism to
structures that are already known for $HF(L,L)$. The key point however
is that, in applications, the ``pearly'' description of these
operations is, by far, the most efficient one.  This will become
apparent by going over the examples of applications which are
presented in the last four sections of the paper.

\subsection*{Acknowledgments.} Some of the results of this paper have
been announced at the Yasha Fest 2007 at Stanford University. We would
like to thank the organizers for the opportunity to present our work
there. We thank Yasha Eliashberg for many years of inspiration, both
mathematical as well as non-mathematical.

\section{The algebraic structures} \label{s:alg-struct} 

\subsection{Setting} \label{sb:setting} All our symplectic manifolds
will be implicitly assumed to be connected and tame (see~\cite{ALP}).
The main examples of such manifolds are closed symplectic manifolds,
manifolds which are symplectically convex at infinity as well as
products of such. We denote by $\mathcal{J}$ the space of
$\omega$-compatible almost complex structures on $M$ for which $(M,
g_{\omega,J})$ is geometrically bounded, where $g_{\omega,J}$ is the
associated Riemannian metric.

Lagrangian submanifolds $L \subset (M,\omega)$ will be  assumed
to be connected and closed.  We
denote by $H^D_2(M,L) \subset H_2(M,L)$ the image of the Hurewicz
homomorphisms $\pi_2(M,L) \longrightarrow H_2(M,L)$. We will be 
interested in {\em monotone} Lagrangians.  This means that the two
homomorphisms:
$$\omega:H^D_{2}(M,L) \longrightarrow \Z, \quad \mu:H^D_{2}(M,L)
\longrightarrow \R$$ given respectively by integration of $\omega$, $A
\mapsto \int_A \omega$, and by the Maslov index satisfy:
$$\omega(A) > 0 \quad \textnormal{iff} \quad \mu(A)>0,
\quad \forall\; A \in H^D_2(M,L).$$ It is easy to see that this is
equivalent to the existence of a constant $\tau > 0$ such that
\begin{equation}\label{eq:monotonicity}
   \omega(A)=\tau \mu(A),\ \forall \ A \in
   H^D_{2}(M,L)~.~
\end{equation}
We refer to $\tau$ as the {\em monotonicity constant} of $L \subset
(M,\omega)$. Define the {\em minimal Maslov number} of $L$ to be the
integer $$N_L = \min \{ \mu(A)>0 \mid A \in H^D_2(M,L) \}.$$
Throughout this paper we assume that $L$ is monotone with $N_L \geq
2$. Since the Maslov numbers come in multiples of $N_L$ we will use
sometimes the following notation:
\begin{equation}
   \label{eq:mubar}
   \mubar = \tfrac{1}{N_L}\mu :
   H^D_2(M,L) \longrightarrow \mathbb{Z}.
\end{equation}

Let $L \subset (M, \omega)$ be a monotone Lagrangian submanifold. Let
$\Lambda = \mathbb{Z}_2[t^{-1}, t]$ be the ring of Laurent polynomials
in $t$. We grade this ring so that $\deg t = -N_L$.  Denote by
$HF(L,L)$ the Floer homology of $L$ with itself, defined over
$\Lambda$. This is essentially the same homology as introduced by
Oh~\cite{Oh:HF1, Oh:HF3} only that since we work over $\Lambda$ our
$HF_*(L,L)$ has a relative $\mathbb{Z}$-grading (not a $\mathbb{Z} /
N_L$-grading as in~\cite{Oh:HF1}) and is $N_L$-periodic in the sense
that $HF_i(L,L) = HF_{i+N_L}(L,L) \cdot t$, $\forall i \in
\mathbb{Z}$. See~\cite{Bi-Co:qrel-long} for more details.

\subsection{Conventions from Morse theory} \label{sb:conv-morse} Let
$f$ be a Morse function on a manifold and $\rho$ a Riemannian metric.
In case the manifold is not compact we will implicitly assume $f$ to
be proper, bounded below and with finitely many critical points.
Denote by $\textnormal{Crit}(f)$ the set of critical
points of $f$. For $x \in \textnormal{Crit}(f)$ we write $|x|$ for
the Morse index of $x$.

We write $\nabla f$ for the gradient vector field of $f$ with respect
to $\rho$ when the metric $\rho$ is clear from the context. We will
mostly work with the {\em negative} gradient flow of $f$, namely the
flow of $-\nabla f$. We denote this flow by $\Phi_t$, $-\infty \leq t
\leq \infty$ (or $-\infty \leq t \leq \infty$ when the manifold is
closed). In particular, all Morse homological constructions will be
carried out using the {\em negative} gradient flow of the Morse
function. For $x \in \textnormal{Crit}(f)$ we denote by $W_x^u(f)$,
$W_x^s(f)$ the unstable and stable submanifolds of the flow $\Phi_t$.

\subsection{The pearl complex} \label{sb:prl-complex}

Let $L \subset (M, \omega)$ be a monotone Lagrangian. Fix a triple
$(f, \rho, J)$ where $f: L \longrightarrow \mathbb{R}$ is a Morse
function, $\rho$ is a Riemannian metric on $L$ and $J \in
\mathcal{J}$. Define a complex generated by the critical points of
$f$:
$$\mathcal{C}(f, \rho, J) = \mathbb{Z}_2 \langle \textnormal{Crit}(f)
\rangle \otimes \Lambda.$$ We grade $\mathcal{C}(f, \rho, J)$ using
the Morse indices of $f$ and the grading of $\Lambda$ mentioned above.
In order to define a differential we need to introduce some moduli
spaces.

Given two points $x, y \in L$ and
a class $0 \neq A \in H^D_2(M,L)$ consider the space of all sequences
$(u_1, \ldots, u_l)$ of every possible length $l \geq 1$, where:
\begin{enumerate}
  \item $u_i : (D, \partial D) \longrightarrow (M, L)$ is a {\em
     non-constant} $J$-holomorphic disk. Here and in what follows $D$
   stands for the closed unit disk in $\mathbb{C}$.
  \item There exists $-\infty \leq t' < 0$ such that $\Phi_{t'}(
   u_1(-1) )= x$. \label{i:prl-t'}
  \item For every $1 \leq i \leq l-1$ there exists $0< t_i < \infty$
   such that $\Phi_{t_i}(u_i(1)) = u_{i+1}(-1)$.
  \item There exists $0< t'' \leq \infty$ such that $\Phi_{t''}
   (u_l(1)) = y$. \label{i:prl-t''}
  \item $[u_1] + \cdots + [u_l] = A$.
\end{enumerate}
We view two elements in this space $(u_1, \ldots, u_l)$ and $(u'_1,
\ldots, u'_{l'})$ as equivalent if $l=l'$ and for every $1 \leq i \leq
l$ there exists $\sigma_i \in \textnormal{Aut}(D)$ with
$\sigma_i(-1)=-1$, $\sigma_i(1)=1$ and such that $u'_i = u_i \circ
\sigma_i$. The space obtained from moding out by this equivalence
relation is denoted by $\mathcal{P}_{\textnormal{prl}}(x,y;A;f, \rho,
J)$.  Elements of this space will be called {\em pearly trajectories
  connecting $x$ to $y$}. A typical pearly trajectory is depicted in
the left part of Figure~\ref{f:pearls}.
\begin{figure}[htbp]
      \psfig{file=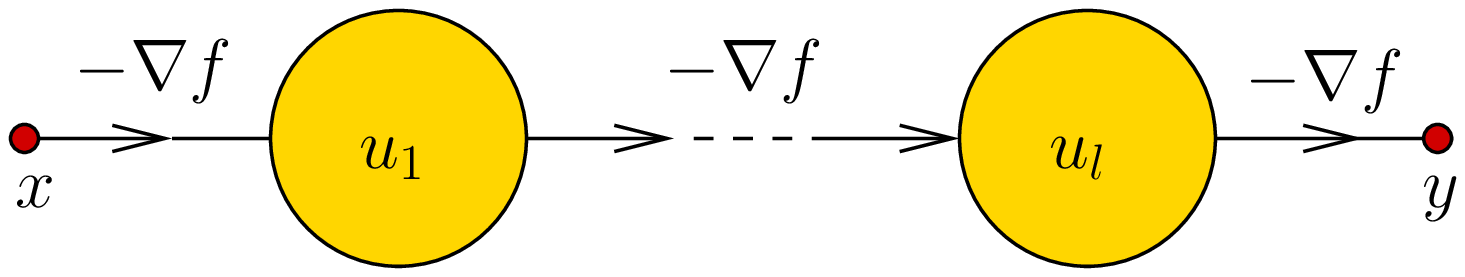, width=0.65 \linewidth} \quad
      \psfig{file=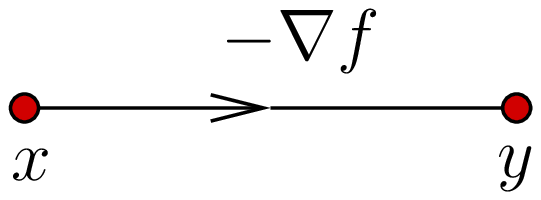, width=0.25 \linewidth}
      \caption{Pearly trajectories connecting $x$ to $y$. On the left
        $A \neq 0$, on the right $A=0$.}
      \label{f:pearls}
\end{figure}

Most of the times we will be interested in the case when both $x$ and
$y$ are critical points of $f$. Of course, in that case
conditions~(\ref{i:prl-t'}),~(\ref{i:prl-t''}) above say that $u_1(-1) \in
W_x^u(f)$ and $u_l(1) \in W_s^y(f)$ (in particular $t' = -\infty$,
$t''= \infty$). We extend the definition of the space of pearly
trajectories to the case $A=0$ by setting
$\mathcal{P}_{\textnormal{prl}}(x,y;0;f, \rho, J)$ to be the space of
unparametrized trajectories of the negative gradient flow $\Phi_t$
connecting $x$ to $y$. See the right part of Figure~\ref{f:pearls}.

When $x,y \in \textnormal{Crit}(f)$ the virtual dimension of
$\mathcal{P}_{\textnormal{prl}}(x,y;A;f, \rho, J)$ is:
\begin{equation} \label{eq:d-prl}
   \delta_{\textnormal{prl}}(x,y;A) = |x| - |y| + \mu(A)-1.
\end{equation}
Suppose that $(f, \rho)$ is Morse-Smale and
$\delta_{\textnormal{prl}}(x,y;A) = 0$. It turns out that for a
generic choice of $J$ the space
$\mathcal{P}_{\textnormal{prl}}(x,y;A;f,\rho,J)$ consists of a finite
number of points (see \S\ref{sb:trans} below). We denote by
$\#_{\mathbb{Z}_2}\mathcal{P}(x,y;A;f,\rho,J)$ this number modulo $2$.
To define a differential $d: \mathcal{C}_*(f, \rho, J) \longrightarrow
\mathcal{C}_{*-1}(f, \rho, J)$, fix a generic $J \in \mathcal{J}$. For
 $x \in \textnormal{Crit}(f)$ define:
\begin{equation} \label{eq:dif-prl} d(x) = \sum_{y,A} \bigl(
   \#_{\mathbb{Z}_2} \mathcal{P}_{\textnormal{prl}}(x,y;A;f,\rho,J)
   \bigr) y t^{\bar{\mu}(A)},
\end{equation}
where the sum is taken over all pairs $y \in \textnormal{Crit}(f)$, $A
\in H^D_2(M,L)$ with $\delta_{\textnormal{prl}}(x,y;A) = 0$. Finally,
extend $d$ to $\mathcal{C}(f, \rho, J)$ by linearity over $\Lambda$.

\begin{thm} \label{t:prl} The map $d$ defined above is a differential,
   namely $d \circ d = 0$. The homology of the complex
   $(\mathcal{C}_*(f, \rho, J), d)$, denoted $QH_*(L)$, is independent
   of the choice of the generic triple $(f, \rho, J)$. More
   specifically, for every two generic triples $\mathscr{D} = (f,
   \rho, J)$, $\mathscr{D}' = (f', \rho', J')$ there exists a chain
   map $\psi_{\mathscr{D}', \mathscr{D}}: \mathcal{C}_*(\mathscr{D})
   \longrightarrow \mathcal{C}_*(\mathscr{D}')$ which descends to a
   canonical isomorphism in homology $\Psi_{\mathscr{D}',\mathscr{D}}:
   H_*(\mathcal{C}(\mathscr{D})) \longrightarrow
   H_*(\mathcal{C}(\mathscr{D}'))$. This systems of isomorphisms is
   compatible with composition: $\Psi_{\mathscr{D}'', \mathscr{D}'}
   \circ \Psi_{\mathscr{D}',\mathscr{D}} =
   \Psi_{\mathscr{D}'',\mathscr{D}}$, $\Psi_{\mathscr{D},\mathscr{D}}
   = \Id$.

   Furthermore, there is an isomorphism $\Theta: HF_*(L,L)
   \longrightarrow QH_*(L)$ which is canonical up to a shift in
   grading.
\end{thm}
We will refer to $QH(L)$ as the {\em quantum homology of $L$}. When we
want to emphasize the specific choice of parameters $(f, \rho, J)$ we
will write $QH_*(L; f, \rho, J)$ for the homology of $\mathcal{C}_*(f,
\rho, J)$. We call the canonical isomorphisms
$\Psi_{\mathscr{D},\mathscr{D'}}$ {\em identification maps}.

\subsection{The Lagrangian quantum product} \label{sb:qprod} Fix three
Morse functions $f, f', f'': L \longrightarrow \mathbb{R}$, a
Riemannian metric $\rho$ on $L$ and a generic $J \in \mathcal{J}$.  We
will now define an operation
\begin{equation} \label{eq:q-prod} \circ: \mathcal{C}(f, \rho, J)
   \otimes_{\Lambda} \mathcal{C}(f', \rho, J) \longrightarrow
   \mathcal{C}(f'', \rho, J), \quad x \otimes y \longmapsto x \circ y.
\end{equation}
This operation will have degree $-n$, where $n = \dim L$, i.e. $\circ
: \mathcal{C}_i(f, \rho, J) \otimes \,\mathcal{C}_j(f', \rho, J)
\longrightarrow \mathcal{C}_{i+j-n}(f'', \rho, J)$ for every $i, j \in
\mathbb{Z}$. For this end we have to introduce some other moduli spaces.
Let $x \in \textnormal{Crit}(f)$, $y \in
\textnormal{Crit}(f')$, $z \in \textnormal{Crit}(f'')$ and  $A \in
H_2^D(M,L)$. Consider the space of all tuples $(\mathbf{u},
\mathbf{u}', \mathbf{u}'', v)$ where:
\begin{enumerate}
  \item $v: (D, \partial D) \longrightarrow (M,L)$ is a
   $J$-holomorphic disk (which is allowed to be constant).
  \item If we denote $\tilde{x} = v(e^{-2 \pi i /3})$, $\tilde{y} =
   v(e^{2 \pi i /3})$, $\tilde{z} = v(1)$ then: $$\mathbf{u} \in
   \mathcal{P}_{\textnormal{prl}}(x,\tilde{x};B;J, \rho, f), \quad
   \mathbf{u}' \in \mathcal{P}_{\textnormal{prl}}(y,\tilde{y};B';J,
   \rho, f'), \quad \mathbf{u}'' \in
   \mathcal{P}_{\textnormal{prl}}(\tilde{z}, z;B'';J, \rho, f''),$$
   for some $B, B', B'' \in H_2^D(M,L)$.
  \item $B + B' + B'' + [v] = A$.
\end{enumerate}
Elements of this space will be denoted by
$\mathcal{P}_{\textnormal{prod}}(x,y,z;A;f,f',f'', \rho, J)$.  A
typical element is depicted in figure~\ref{f:qprod}.

\begin{figure}[htbp]
   \psfig{file=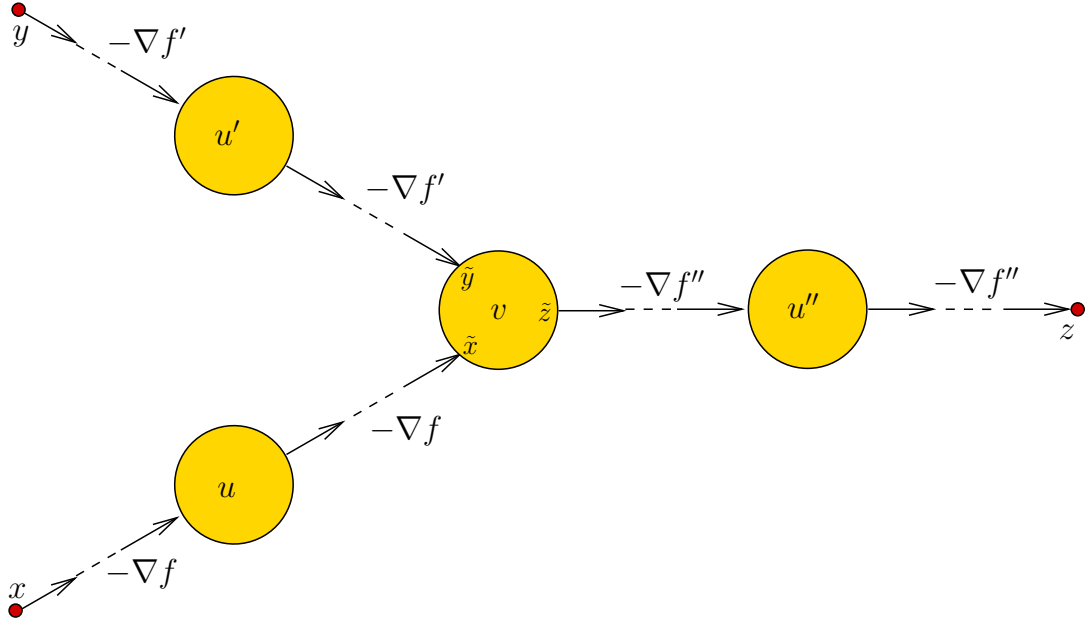, width=0.90 \linewidth}
   \caption{An element of
     $\mathcal{P}_{\textnormal{prod}}(x,y,z;A;f,f',f'', \rho, J)$.}
   \label{f:qprod}
\end{figure}

The virtual dimension of
$\mathcal{P}_{\textnormal{prod}}(x,y,z;A;f,f',f'', \rho, J)$ is
\begin{equation} \label{eq:d-prod}
   \delta_{\textnormal{prod}}(x,y,z;A) = |x| + |y| - |z| - n + \mu(A).
\end{equation}
Suppose that $(f, f', f'', \rho)$ are in general position and
$\delta_{\textnormal{prod}}(x,y,z;A) = 0$. It turns out that for
generic $J \in \mathcal{J}$ the space
$\mathcal{P}_{\textnormal{prod}}(x,y;A;f,\rho,J)$ consists of a finite
number of points (see \S\ref{sb:trans}). The operation $\circ$ is now
defined as follows.  For $x \in \textnormal{Crit}(f)$, $y \in
\textnormal{Crit}(f')$ put:
\begin{equation} \label{eq:qprod} x \circ y = \sum_{z,A} \bigl(
   \#_{\mathbb{Z}_2} \mathcal{P}_{\textnormal{prod}}(x,y,z;A;f, f',
   f'', \rho, J) \bigr) z t^{\bar{\mu}(A)},
\end{equation}
where the sum is taken over all $z \in \textnormal{Crit}(f'')$ and $A
\in H_2^D(M,L)$ with $\delta_{\textnormal{prod}}(x,y,z;A) = 0$.
Again, we extend the definition of $\circ$ by linearity over $\Lambda$.

\begin{thm} \label{t:qprod}
   \begin{enumerate}[i.]
     \item The map $\circ$ is a chain map, hence descends to an
      operation in homology. The operation in homology is canonical in
      the sense that it is compatible with the system of
      identification maps $\Psi_{-,-}$ mentioned in
      Theorem~\ref{t:prl}. Thus we obtain a canonical operation, still
      denoted $\circ$:
      $$\circ: QH_i(L) \otimes QH_j(L) \to QH_{i+j-n}(L), \quad
      \forall \, i, j \in \mathbb{Z}.$$
     \item The operation $\circ$ endows $QH(L)$ with the structure of
      an associative ring with unity. This ring is, in general, not
      commutative (not even in the graded sense).
     \item The unity of $QH(L)$ has degree $n$. In fact, if $f: L \to
      \mathbb{R}$ is a Morse function with exactly one (local) maximum
      $x \in L$ then $x \in \mathcal{C}_n(f, \rho, J)$ is a cycle
      whose homology class $[x] \in QH_n(L)$ does not depend on $(f,
      \rho, J)$ and which represents the unity. By abuse of notation,
      and by analogy to Morse theory, we denote the unity by $[L] \in
      QH_n(L)$.
     \item The product $\circ$ corresponds under the identification
      $\Theta: HF_*(L,L) \to QH_*(L)$ to the Donaldson product defined
      by counting holomorphic triangles.
   \end{enumerate}
\end{thm}

\subsection{The quantum module structure} \label{sb:qmod} Here we
define an external operation which makes $QH(L)$  a module over
the quantum homology of the ambient manifold.

We start with a few preliminaries on quantum homology. First recall
that if $L \subset (M, \omega)$ is monotone then the ambient
symplectic manifold $(M, \omega)$ is spherically monotone, namely
there exists a constant $\nu>0$ such that $\omega(A) = \nu c_1(A)$ for
every $A \in \pi_2(M)$, where $c_1 \in H^2(M)$ is the first Chern
class of the tangent bundle of $M$. In fact the monotonicity constant
$\nu$ is related to $\tau$ (see~\eqref{eq:monotonicity}) by $\nu = 2
\tau$. We denote by $C_M$ the minimal Chern number of $M$: $$C_M =
\min \{ c_1(A)>0 \mid A \in \pi_2(M) \}.$$

Let $\Gamma = \mathbb{Z}_2[s, s^{-1}]$. Define a grading on $\Gamma$
by setting $\deg s = -2C_M$.  A special convention is valid if
$c_{1}|_{\pi_{2}(M)}=0$.  In this case, we put $C_{M}=\infty$ and
$\Gamma=\Z_{2}$.

Denote by $QH(M) = H(M;\mathbb{Z}_2)
\otimes \Gamma$ the quantum homology of $M$ endowed with the quantum
intersection product $* : QH_l(M) \otimes QH_k(M) \longrightarrow
QH_{l+k-2n}(M)$, where $2n = \dim M$. Recall that this is an
associative and commutative product (we work over $\mathbb{Z}_2$). The
unity is the fundamental class $[M] \in QH_{2n}(M)$. We refer
to~\cite{McD-Sa:Jhol-2} for the foundations of quantum homology
theory.

We will actually need to work with the following small extension of
$QH(M)$. Consider the ring embedding $\Gamma \hookrightarrow \Lambda$
induced by $s \mapsto t^{2C_M/N_L}$.  Using this embedding we can
regard $\Lambda$ as a module over $\Gamma$.  Define
$$QH(M;\Lambda) = QH(M) \otimes_{\Gamma} \Lambda.$$
We endow $QH(M;\Lambda)$ with the same quantum intersection product
$*$.

\begin{ex}\label{ex:cpn} Consider $M = {\mathbb{C}}P^n$ endowed with
   its standard K\"{a}hler symplectic form. This manifold is monotone
   with $C_M = n+1$. Denote by $h = [{\mathbb{C}}P^{n-1}] \in
   H_{2n-2}({\mathbb{C}}P^n)$ the homology class of a hyperplane and
   by $h^{\cap j} = [{\mathbb{C}}P^{n-j}] \in
   H_{2n-2j}({\mathbb{C}}P^n)$ the class of a codimension $j$ complex
   linear subspace. The quantum product in $QH({\mathbb{C}}P^n)$ is
   given by:
   \begin{equation} \label{eq:qh-cpn}
      h^{*j} =
      \begin{cases}
         h^{\cap j}, & 0 \leq j \leq n \\
         [{\mathbb{C}}P^n] s, & j=n+1
      \end{cases}
   \end{equation}
   On the other hand, if we work for example with Lagrangians $L$ with
   $N_L = n+1$ (e.g. $L = \mathbb{R}P^n \subset {\mathbb{C}}P^n$) then
   in $\Lambda$ we have $\deg t = -(n+1)$ and the embedding $\Gamma
   \hookrightarrow \Lambda$ is given by $s \mapsto t^2$. Thus the last
   identity in~\eqref{eq:qh-cpn} becomes in $QH({\mathbb{C}}P^n;
   \Lambda)$: $h^{*(n+1)} = [{\mathbb{C}}P^n] t^2$.
\end{ex}

We proceed with the definition of the module action of $QH(M;
\Lambda)$ on $QH(L)$. Let $f: L \longrightarrow \mathbb{R}$, $h: M
\longrightarrow \mathbb{R}$ be Morse functions and $\rho_L$, $\rho_M$
Riemannian metrics on $L$ and $M$. We write $\Phi^f_t$ and $\Phi^h_t$
for the negative gradient flows of $f$ and $h$ with respect to the
corresponding Riemannian metrics. Denote by $C(h, \rho_M; \Lambda) =
\mathbb{Z}_2 \langle \textnormal{Crit}(h) \rangle \otimes \Lambda$ the
Morse complex with coefficients in $\Lambda$. Clearly there is an
isomorphism of $\Lambda$-modules: $H_*(C(h, \rho; \Lambda)) \cong
QH_*(M; \Lambda)$.

Let $x, y \in \textnormal{Crit}(f)$, $a \in \textnormal{Crit}(h)$, $A
\in H_2^D(M,L)$ (we allow $A$ to be $0$ here). Consider the space of
all sequences $(u_1, \ldots, u_l; k)$, of every possible length $l
\geq 1$, where:
\begin{enumerate}
  \item $1 \leq k \leq l$.
  \item $u_i: (D, \partial D) \longrightarrow (M,L)$ is a
   $J$-holomorphic disk for every $1 \leq i \leq l$, which is assumed
   to be non-constant except possibly when $i=k$.
  \item $u_1(-1)\in W_x^u(f)$.
  \item For every $1 \leq i \leq l-1$ there exists $0< t_i < \infty$
   such that $\Phi^f_{t_i}(u_i(1)) = u_{i+1}(-1)$.
  \item $u_l(1) \in W_s^y(f)$.
  \item $u_k(0) \in W_a^u(h)$.
  \item $[u_1] + \cdots + [u_l] = A$.
\end{enumerate}
We view two elements in this space $(u_1, \ldots, u_l;k)$ and $(u'_1,
\ldots, u'_{l'};k')$ as equivalent if $l=l'$, $k=k'$, and for every $i
\neq k$ there exists $\sigma_i \in \textnormal{Aut}(D)$ with
$\sigma_i(-1)=-1$, $\sigma_i(1)=1$ and such that $u'_i = u_i \circ
\sigma_i$. The space obtained by moding out by this equivalence
relation is denoted by $\mathcal{P}_{\textnormal{mod}}(a,x,y;A;h,
\rho_M, f, \rho_L, J)$. A typical element of this space is depicted in
Figure~\ref{f:qmod}.
\begin{figure}[htbp]
   \psfig{file=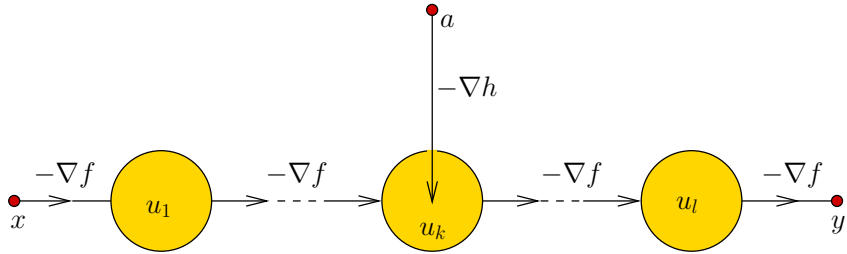, width=0.70 \linewidth}
   \caption{An element of $\mathcal{P}_{\textnormal{mod}}(a,x,y;A;h,
     \rho_M, f, \rho_L, J)$.}
      \label{f:qmod}
\end{figure}

The virtual dimension of $\mathcal{P}_{\textnormal{mod}}(a,x,y;A;h,
\rho_M, f, \rho_L, J)$ is:
\begin{equation} \label{eq:d-mod} \delta_{\textnormal{mod}}(a,x,y; A)
   = |a| + |x| - |y| + \mu(A) - 2n.
\end{equation}
As before, if $(h, \rho_M, f, \rho_L)$ are in general position and $J
\in \mathcal{J}$ is generic then whenever
$\delta_{\textnormal{mod}}(a,x,y;A)=0$ the space
$\mathcal{P}_{\textnormal{mod}}(a,x,y;A;h, \rho_M, f, \rho_L, J)$
consists of a finite number of points.

We now define a map $\circledast: C(h, \rho_M; \Lambda)
\otimes_{\Lambda} \mathcal{C}(f,\rho_L,J) \longrightarrow
\mathcal{C}(f, \rho_L,J)$, $a \otimes x \longmapsto a \circledast x$.
For $a \in \textnormal{Crit}(h)$, $x \in \textnormal{Crit}(f)$ put:
\begin{equation} \label{eq:qmod} a \circledast x = \sum_{y,A}
   \#_{\mathbb{Z}_2} \bigl(
   \mathcal{P}_{\textnormal{mod}}(a,x,y;A;h,\rho_M, f, \rho_L, J)
   \bigr) y t^{\bar{\mu}(A)},
\end{equation}
where the sum is taken over all pairs $y \in \textnormal{Crit}(f)$, $A
\in H_2^D(M,L)$ with $\delta_{\textnormal{mod}}(a,x,y;A)=0$.  Finally,
extend $\circledast$ by linearity over $\Lambda$. Note that the
operation $\circledast$ has degree $-2n$, i.e. $\circledast:
C_k(h,\rho_M;\Lambda) \otimes_{\Lambda} \mathcal{C}_j(f, \rho_L, J)
\longrightarrow \mathcal{C}_{k+j-2n}(f, \rho_L, J)$.

\begin{thm} \label{t:qmod}
   \begin{enumerate}[i.]
     \item The map $\circledast$ is a chain map, hence descends to a
      an operation in homology. This operation in homology is
      compatible with the identification maps $\Psi_{-,-}$ mentioned
      in Theorem~\ref{t:prl} as well as with the Morse homological
      identifications for the homology $QH(M;\Lambda)$. Thus we obtain
      a canonical operation, still denoted $\circledast$:
      $$\circledast: QH_k(M;\Lambda) \otimes QH_j(L) \longrightarrow
      QH_{k+j-2n}(L), \quad \forall\, k,j \in \mathbb{Z}.$$
     \item The operation $\circledast$ makes $QH(L)$ into module over
      the ring $QH(M;\Lambda)$ when the latter is endowed with its
      quantum product $*$.  This means, in particular, that the
      following identities hold (in homology):
      $$a \circledast (b \circledast x) = (a*b) \circledast x, \quad
      [M] \circledast x = x,$$ for every homology classes $a, b \in
      QH(M;\Lambda)$, $x \in QH(L)$.
     \item Furthermore, the ring $QH(L)$ endowed with the product
      $\circ$ (see Theorem~\ref{t:qprod}), becomes a two-sided algebra
      over $QH(M)$. This means that we have the following additional
      identities (in homology):
      $$a \circledast (x \circ y) = (a \circledast x) \circ y =
      x \circ (a \circledast y),$$ for every homology classes $a \in
      QH(M;\Lambda)$, $x, y \in QH(L)$.
   \end{enumerate}
\end{thm}

\begin{rem} The quantum homology ring $QH(L)$ is actually a symplectic
   invariant of $L$ in the sense that if $\phi:M\to M$ is a
   symplectomorphism and $L'=\phi(L)$, then $QH(L)\cong QH(L')$.  In
   case, $\phi\in \textnormal{Symp}_{H}$, then this isomorphism is
   also an isomorphism of algebras (here, $\textnormal{Symp}_{H}$ is
   the group of symplectomorphisms of $M$ which induce the identity in
   $H_{\ast}(M;\Z_{2})$).
\end{rem}

\subsection{The quantum inclusion map} \label{sb:qinc} We now define a
quantum version of the classical map $H_*(L) \longrightarrow H_*(M)$
induced by the inclusion.

As in~\S\ref{sb:qmod} above, fix Morse functions $h: M
\longrightarrow \mathbb{R}$, $f: L \longrightarrow \mathbb{R}$,
Riemannian metrics $\rho_M$, $\rho_L$ on $M$ and $L$ and an almost
complex structure $J \in \mathcal{J}$. We use the same notation
$\Phi_t^h$, $\Phi_t^f$ for the negative gradient flows, as
in~\S\ref{sb:qmod}.

For $x \in \textnormal{Crit}(f)$, $a \in \textnormal{Crit}(h)$ and $A
\in H_2^D(M,L)$ consider the space of all sequences $(u_1, \ldots,
u_l)$ of every possible length $l \geq 1$ such that:
\begin{enumerate}
  \item $u_i: (D, \partial D) \longrightarrow (M, L)$ is a
   $J$-holomorphic disk for every $1 \leq i \leq l$. All the disks
   $u_i$, $1 \leq i \leq l-1$ are assumed to be non-constant, but
   $u_l$ is allowed to be constant.
  \item $u_1(-1) \in W_x^u(f)$.
  \item For every $1 \leq i \leq l-1$ there exists $0< t_i < \infty$
   such that $\Phi^f_{t_i}(u_i(1)) = u_{i+1}(-1)$.
  \item $u_l(0) \in W_a^s(h)$.
  \item $[u_1] + \cdots + [u_l] = A$.
\end{enumerate}
As before, we view two elements in this space $(u_1, \ldots, u_l)$ and
$(u'_1, \ldots, u'_{l'})$ as equivalent if $l=l'$ and for every $1
\leq i \leq l-1$ there exists $\sigma_i \in \textnormal{Aut}(D)$ with
$\sigma_i(-1)=-1$, $\sigma_i(1)=1$ and such that $u'_i = u_i \circ
\sigma_i$. The space obtained by moding out by this equivalence
relation is denoted by $\mathcal{P}_{\textnormal{inc}}(x,a;A;h,
\rho_M, f, \rho_L, J)$. A typical element of this space is depicted in
Figure~\ref{f:qinc}.
\begin{figure}[htbp]
   \psfig{file=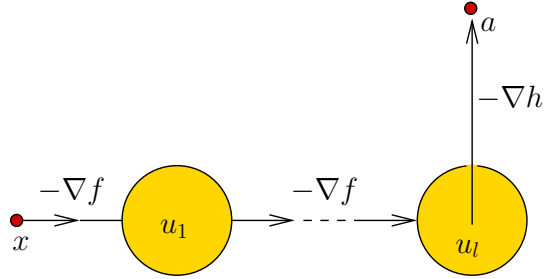, width=0.45 \linewidth}
   \caption{An element of $\mathcal{P}_{\textnormal{inc}}(x,a;A;h,
     \rho_M, f, \rho_L, J)$.}
   \label{f:qinc}
\end{figure}

The virtual dimension of this space is:
\begin{equation} \label{eq:d-inc} \delta_{\textnormal{inc}}(x,a;A) =
   |x| - |a| + \mu(A).
\end{equation}
As before, if $(h, \rho_M, f, \rho_L)$ are in general position and $J
\in \mathcal{J}$ is generic then whenever
$\delta_{\textnormal{inc}}(x,a;A)=0$ the space
$\mathcal{P}_{\textnormal{mod}}(x,a;A;h, \rho_M, f, \rho_L, J)$
consists of a finite number of points.

We now define a map $\widetilde{i}_L: \mathcal{C}_*(f, \rho_L,J)
\longrightarrow C_*(h,\rho_M;\La)$ of degree $0$ using the formula:
\begin{equation} \label{eq:inc} \widetilde{i}_L (x) = \sum_{a,A}
   \bigl( \#_{\mathbb{Z}_2} \mathcal{P}_{\textnormal{inc}}(x,a;A;h,
   \rho_M, f, \rho_L, J) \bigr) a t^{\bar{\mu}(A)}, \quad \forall \, x
   \in \textnormal{Crit}(f),
\end{equation}
where the sum is taken over all pairs $a$, $A$ with
$\delta_{\textnormal{inc}}(x,a;A)=0$. We extend $\widetilde{i}_L$ to
$\mathcal{C}(f, \rho_L, J)$ by linearity over $\Lambda$.

\begin{thm} \label{t:qinc}
   The map $\widetilde{i}_L$ is a chain map, hence descends to
   homology. The induced map in homology is compatible with the
   identifications maps $\Psi_{-,-}$ mentioned in Theorem~\ref{t:prl}
   as well as with the Morse homological identifications for the
   homology $QH(M;\Lambda)$. Thus we obtain a canonical map
   $$i_L : QH_*(L) \longrightarrow QH_*(M;\Lambda).$$
   Moreover, when viewing $QH(L)$ as a module over $QH(M;\Lambda)$
   (see Theorem~\ref{t:qmod}), $i_L$ is a map of
   $QH_*(M;\Lambda)$-modules. In other other words, for every $a \in
   QH(M;\Lambda)$, $x \in QH(L)$ we have $i_L(a \circledast x) =
   a*i_L(x)$.
\end{thm}

\subsection{Relation to the classical operations} \label{sb:classical}
All the operations described in~\S\ref{sb:prl-complex} -~\ref{sb:qinc}
have classical Morse-theoretic counterparts.  For example, the pearly
differential $d$ can be written as a sum of operators $d = \partial_0
+ \partial_1 t + \cdots + \partial_{\nu}t^{\nu}$, where $\partial_i :
C_*(f, \rho) \to C_{*-1+i N_L}(f, \rho)$ is defined as:
$$\partial_i(x) = \sum_{\substack{y, A, \\
    \mu(A) = i N_L, \\ |y| = |x|-1 + iN_L}} \#_{\mathbb{Z}_2}
\mathcal{P}_{\textnormal{prl}}(x,y;A;f,\rho,J) \bigr) y.$$ While the
operators $\partial_i$, $i \geq 1$ are in general not differentials
the operator $\partial_0 : \mathcal{C}_* \to \mathcal{C}_{*-1}$ is
precisely the Morse homology differential. To see this note that the
only space $\mathcal{P}_{\textnormal{prl}}(x,y;A;f, \rho, J)$ that
contributes to $\partial_0$ is when $A=0$. This follows from
monotonicity since there are no pseudo-holomorphic disks with Maslov
index $0$ that are not constant. Thus $\partial_0(x)$ involves only
the spaces $\mathcal{P}_{\textnormal{prl}}(x,y;0;f, \rho,J)$ which, by
definition, are the spaces of negative gradient trajectories of $f$
connecting $x$ to $y$.

Similarly, the operation $\circ: \mathcal{C}(f, \rho, J)
\otimes_{\Lambda} \mathcal{C}(f', \rho, J) \longrightarrow
\mathcal{C}(f'', \rho, J)$ defined in~\S\ref{sb:qprod} is related to
the classical intersection product in Morse homology in the following
way. Write $\circ$ as a sum:
$$x \circ y = x\circ_0 y + x\circ_1 y t + \cdots +
x \circ_{\kappa} y t^{\kappa},$$ where $\circ_i : C_p(f,
\rho) \otimes C_q(f', \rho) \to
C_{p+q-n-iN_L}(f'', \rho)$ stands for the coefficient
in front of $t^i$ in formula~\eqref{eq:qprod}. The operator
$\circ_0: \mathcal{C}_p \otimes \mathcal{C}_q \to \mathcal{C}_{p+q-n}$
coincides with the Morse-theoretic intersection product. Indeed, by
monotonicity $\circ_0$ involves only the spaces
$\mathcal{P}_{\textnormal{prod}}(x,y,z;0; f, f', f'', \rho, J)$.
Moreover, in this case every element $(\mathbf{u}, \mathbf{u}',
\mathbf{u}'', v)$ must have $v =$ const and all the other pearly
trajectories $\mathbf{u}, \mathbf{u}', \mathbf{u}''$ contain no disks.
Thus, the points of $\mathcal{P}_{\textnormal{prod}}(x,y,z;0; f, f',
f'', \rho, J)$ are in 1--1 correspondence with points of the triple
intersection $W_x^u(f) \cap W_y^u(f') \cap W_z^s(f'')$. This is precisely
the Morse-theoretic definition of the intersection product on the
chain level.

The quantum module structure of~\S\ref{sb:qmod} is related to the
external intersection product, intersecting cycles in $M$ with cycles
in $L$. Indeed, if we take $A=0$ in the definition of
$\mathcal{P}_{\textnormal{mod}}(a,x,y;A;h, \rho_M, f, \rho_L, J)$ we
see that every element $(u_1, \ldots, u_l;k)$ in this space must have
$l=1$ and the disk $u_1$ must be constant. These elements are in 1--1
correspondence with the points of the triple intersection $W_x^u(f)
\cap W_y^s(f) \cap W_a^u(h)$.

Finally, the quantum inclusion from~\S\ref{sb:qinc} is related in a
similar way to the classical inclusion map sending cycles in $L$ to
cycles in $M$.

\medskip The relation to the classical operation bears some analogy to
the situation in the theory of quantum homology (of the ambient
symplectic manifold). However a bit of caution is necessary here: this
analogy holds on the chain level but not in homology.  In fact, there
is no way to recover the singular homology $H_*(L)$ from the quantum
homology $QH_*(L)$. Similarly, while the (ambient) quantum product on
$QH(M)$ can be seen as a deformation of the classical intersection
product this is not the case for $QH(L)$. For example, there are
situations in which $QH(L)$ vanishes (e.g. when $L$ is displaceable).
The reason is that the pearly differential $d$ is already deformed
with respect to the Morse differential $\partial_0$ hence the relation
between $QH_*(L)$ and $H_*(L)$ is more complicated.  In fact,
$QH_*(L)$ and $H_*(L)$ are related via a spectral sequence whose
second page can be constructed from $H_*(L)$.  This spectral sequence
was introduced by Oh~\cite{Oh:spectral}. See
also~\cite{Bi:Nonintersections, Bu:products} for an alternative
description and applications of this point of view. In
\S\ref{sbsb:duality} we will briefly review this construction.

In~\S\ref{sb:specialization} we  will discuss further the relation between
the quantum operations  and  the  classical ones on  the   homological
level.

\subsection{Previous works and related references}
\label{sb:prev-work}
Parts of the constructions above appear already in the literature and
have been verified up to various degrees of rigor. The complex
$\mathcal{C}(f,\rho,J)$ was first introduced by Oh~\cite{Oh:relative}
(see also Fukaya~\cite{Fu:Morse-homotopy}) and is a particular case of
the cluster complex as described in
Cornea-Lalonde~\cite{Cor-La:Cluster-1}. The module structure is
probably known to experts -- at least in the Floer homology setting --
but has not been explicitly described yet in the literature.  The
quantum product which is a variant of the Donaldson product might not
be widely known in the form presented above.  The quantum inclusion
map $i_{L}$ is the analogue of a map first studied by Albers
in~\cite{Alb:extrinisic} in the absence of bubbling. The comparison
map $\Theta$ from Theorem~\ref{t:prl} is an extension of the
Piunikin-Salamon-Schwarz construction~\cite{PSS}, it extends also the
partial map constructed by Albers in~\cite{Alb:PSS} and a more general
such map was described in~\cite{Cor-La:Cluster-1} in the ``cluster"
context.  We also remark that this comparison map identifies all the
algebraic structures described above with the corresponding ones
defined in terms of the Floer complex.

% LocalWords:  Albers Piunikin Schwarz

\section{Main ideas for the proofs of the Theorems from \S\ref{s:alg-struct}}
\label{s:main-ideas-proof}

Most of the proofs of Theorems~\ref{t:prl}-~\ref{t:qinc} follow
standard arguments from Morse and Floer theories, the main building
blocks being: {\em transversality, compactness and gluing}. The scheme
is roughly as follows. One considers the same moduli spaces introduced
above but with virtual dimension $1$. A transversality argument shows
that for a generic choice of parameters these spaces are smooth
$1$-dimensional manifolds. These manifolds are in general not compact.
Compactness and gluing are then used to give a precise description of
the compactification of these $1$-dimensional manifolds. It then turns
out that these compactifications still have a structure of
$1$-dimensional manifolds with boundary. The boundary points can
usually be described in terms of elements of the same types of moduli
spaces, but now having virtual dimension $0$. As the number of
boundary points of a compact $1$-dimensional manifold must be $0$ mod
$2$ we obtain form this procedure an identity involving the number of
points in various $0$-dimensional moduli spaces. These identities, it
turns out, are equivalent to the statements saying that $d$ is a
differential, and that the quantum operations $\circ$, $\circledast$,
$i_L$ are chain maps. The other properties stated in the Theorems
above can be proved by a similar scheme by introducing appropriate
moduli spaces, $0$-dimensional as well as $1$-dimensional.

Below we will outline in some detail the proof of the simplest
statement: the fact that the map $d$ is a differential, as stated in
Theorem~\ref{t:prl}. Still, we will skip many technical points, and
only mention the main ideas in each step. We refer the reader
to~\cite{Bi-Co:qrel-long, Bi-Co:rigidity} for the precise details.

While compactness and gluing are rather standard by now, our approach
to transversality is somewhat less mainstream. It will be explained in
the next subsection. Throughout the rest of this section we continue to
assume implicitly that $L \subset (M, \omega)$ is monotone.

\subsection{Transversality for pearly moduli spaces} \label{sb:trans}
Formally we need (at least) four types of transversality results: one
for each of the spaces $\mathcal{P}_{\textnormal{prl}}$,
$\mathcal{P}_{\textnormal{prod}}$, $\mathcal{P}_{\textnormal{mod}}$,
$\mathcal{P}_{\textnormal{inc}}$. The statements in all four
cases are quite similar. They all assert that when the Morse
functions, metric and almost complex structures are chosen generically
then whenever the virtual dimension $\delta(\cdots)$ is $\leq 1$, the
corresponding moduli space $\mathcal{P}(\cdots)$ is a smooth manifold
whose dimension equals the virtual dimension. Moreover, when
$\delta(\cdots)=0$ the corresponding space is a compact
$0$-dimensional manifold hence consists of a finite number of points.

In order not to make lengthy repetitions of similar statements we will
use the following unifying notation. We will denote by $\mathcal{S}$
the type of the moduli space under considerations, namely
$\mathcal{S}$ can be one of ``prl'', ``prod'', ``mod'' or ``inc''. We
denote by $\mathcal{F}$ the choice of the Morse data and by $I$ a tuple
consisting of critical points and homology class $A \in H_2^D(M,L)$.
More specifically:
\begin{enumerate}
  \item When $\mathcal{S} = \textnormal{prl}$, $\mathcal{F} = (f,
   \rho)$, $I = (x,y;A)$, where $f$ is a Morse function on $L$, $\rho$
   is a Riemannian metric on $L$ and $x, y \in \textnormal{Crit}(f)$.
  \item When $\mathcal{S} = \textnormal{prod}$, $\mathcal{F}= (f, f',
   f'', \rho)$, $I = (x, y, z; A)$, where $f, f', f''$ are Morse
   functions on $L$, $\rho$ is a Riemannian metric on $L$ and $x \in
   \textnormal{Crit}(f)$, $y \in \textnormal{Crit}(f')$, $z \in
   \textnormal{Crit}(f'')$.
  \item When $\mathcal{S} = \textnormal{mod}$, $\mathcal{F} = (h,
   \rho_M, f, \rho_L)$, $I = (a, x, y; A)$, where $h$, $\rho_M$, resp.
   $f$, $\rho_L$, are a Morse function and a Riemannian metric on $M$,
   resp. $L$, and $a \in \textnormal{Crit}(h)$, $x,y \in
   \textnormal{Crit}(f)$. \label{i:symb-mod}
  \item When $\mathcal{S} = \textnormal{inc}$, $\mathcal{F}= (h,
   \rho_M, f, \rho_L)$, $I = (x, a; A)$, where the components of
   $\mathcal{F}$ as well as $x$, $a$ are as in point~\ref{i:symb-mod}
   above.
\end{enumerate}
We denote by $\delta_{\mathcal{S}}(I)$ the virtual dimension of the
space $\mathcal{P}_{\mathcal{S}}(I,\mathcal{F},J)$ as defined by
formulae
~\eqref{eq:d-prl},~\eqref{eq:d-prod},~\eqref{eq:d-mod},~\eqref{eq:d-inc}
in \S\ref{s:alg-struct}.

We will have to impose some genericity assumptions on the Morse data
$\mathcal{F}$. We will call $\mathcal{F}$ generic if the following
holds:
\begin{assumption}[Genericity] \label{a:generic} When $\mathcal{S} =
   \textnormal{prl}$ assume that $\mathcal{F}= (f, \rho)$ is
   Morse-Smale. When $\mathcal{S} = \textnormal{prod}$ assume that
   $\mathcal{F} = (f, f', f'', \rho)$ has the property that for every
   critical point $p \in \textnormal{Crit}(f)$, $p' \in
   \textnormal{Crit}(f')$, $p'' \in \textnormal{Crit}(f'')$ the triple
   intersection $W_{p}^u(f) \cap W_{p'}^u(f') \cap W_{p''}^s(f'')$ is
   transverse. Finally, when $\mathcal{S} = \textnormal{mod}$ or inc
   assume that the following holds: each of the pairs $(f,\rho_L)$ and
   $(h,\rho_M)$ is Morse-Smale and, if $M$ is compact, $h$ has a
   single maximum. Furthermore:
   \begin{enumerate}[a.]
     \item In case $M$ is not compact we assume that $h$ is proper,
      bounded below and has finitely many critical points.
     \item None of the critical points of $h$ lies on $L$.
     \item For every $a \in \textnormal{Crit}(h)$ the unstable
      submanifold $W_a^u(h)$ as well as the stable submanifold
      $W_a^s(h)$ are both transverse to $L$.
     \item For every $a \in \textnormal{Crit}(h)$, $x,y \in
      \textnormal{Crit}(f)$, $W_a^u(h)$ is transverse to $W_x^u(f)$
      and to $W_y^s(f)$.
   \end{enumerate}
\end{assumption}

Standard Morse theory arguments show that if $\mathcal{F}$ is generic
in the usual sense, then it satisfies Assumption \ref{a:generic}.
Here is the transversality result needed to construct the structures
in \S\ref{sb:prl-complex}-~\ref{sb:qinc} and to show that they induce
the respective operations in homology.
\begin{prop} \label{p:transversality} Let $\mathcal{S}$ and
   $\mathcal{F}$ be as above. Assume that $\mathcal{F}$ satisfies the
   genericity assumption~\ref{a:generic} and that, if $N_{L}=2,\
   \delta_{\mathcal{S}}(I)=1$, then
   $\mathcal{S}\not=\textnormal{mod}$.  Then there exists a second
   category subset $\mathcal{J}_{\textnormal{reg}} \subset
   \mathcal{J}$ such that for every $J \in
   \mathcal{J}_{\textnormal{reg}}$ the following holds. For every
   tuple $I$ as above with $\delta_{\mathcal{S}}(I) \leq 1$ the space
   $\mathcal{P}_{\mathcal{S}}(I,\mathcal{F},J)$ is either empty or a
   smooth manifold of dimension $\delta_{\mathcal{S}}(I)$.  Moreover,
   when $\delta_{\mathcal{S}}(I)=0$ this $0$-dimensional manifold is
   compact, hence consists of a finite number of points.
\end{prop}

This transversality statement is emblematic for the types of arguments
involved. However, it is not sufficient to also prove the relations -
associativity etc - contained in the statements of \ref{sb:qprod} and
\ref{sb:qmod} as well as to deal with the exceptional case
$\mathcal{S}=\textnormal{mod}, N_{L}=2, \delta_{\mathcal{S}}(I)=1$.
New moduli spaces are needed for this purpose and Hamiltonian
perturbations are required to show the fact that $QH(L)$ is an algebra
over $QH(M;\La)$ (see \S\ref{sbsb:prf-rest} for a more complete
discussion of this).

\subsubsection{How to prove transversality} \label{sbsb:prf-trs} In
order to insure that moduli spaces involving pseudo-holomorphic curves
are smooth manifolds, and that certain evaluation maps are transverse
to some submanifolds, one has to restrict to curves $u : \Sigma \to M$
that are simple (or, at least, somewhere injective). Indeed, it is
well known (see~\cite{McD-Sa:Jhol-2}) that for generic $J$ the space
of simple $J$-holomorphic curves (in a given class) is a smooth
manifold whose dimension equals the virtual dimension. Moreover, for
simple curves one can arrange all appropriate evaluation maps to be
transverse to any given submanifold in their target.

Appearance of non-simple curves is relatively easy to deal with (at
least in the monotone case) when the domains of the curves $\Sigma$
are closed Riemann surfaces since a curve $u$ that is not simple
factors as $u' \circ \phi$ where $u': \Sigma' \to M$ is a simple curve
and $\phi: \Sigma \to \Sigma'$ is a branched covering
(see~\cite{McD-Sa:Jhol-2}). One then replaces $u$ by $u'$ for which
transversality holds.

The situation becomes more involved when the domain of the curves has
boundary,  as in our case, when $\Sigma$ is a disk.
It is well known that in this case a pseudo-holomorphic curve $u:(D,
\partial D) \to (M,L)$ might not be simple yet not multiply covered in
the sense of the factorization $u = u' \circ \phi$ just mentioned. In
fact, it may happen that the number of points in the preimage
$u^{-1}(p)$, $p \in \textnormal{image\,}u$ is not constant, even away
from the set of zeros of $du$. The reason for that is roughly speaking
that points in the interior $z \in \textnormal{Int\,} D$ might be
mapped by $u$ to $L$.

The main tool which enables to deal with this difficulty has been
obtained by Lazzarini and, independently, by Kwon and Oh. The key point
is the following.  Roughly speaking, when a $J$-holomorphic disk $u:
(D, \partial D) \to (M,L)$ is not simple it is possible to decompose
its domain $D$ into subdomains $\mathfrak{D}_i$ such that the
restriction of $u$ to the closure of each of them,
$u|_{\overline{\mathfrak{D}}_i}$, factors through a simple
$J$-holomorphic disk $v_i: (D, \partial D) \to (M,L)$ via a branched
covering $\overline{\mathfrak{D}}_i \to D$ of some degree $m_i$.
Moreover, the total homology class is preserved: $[u] = \sum_i m_i
[v_i] \in H_2^D(M,L)$. We refer the reader to
Lazzarini~\cite{Laz:discs, Laz:decomp} and to~\cite{Kw-Oh:discs} for
the precise details.

Coming back to our situation, we know that for generic $J$ the
subspace $\mathcal{P}^*_{\mathcal{S}}(I,\mathcal{F},J) \subset
\mathcal{P}_{\mathcal{S}}(I,\mathcal{F},J)$ formed by elements
containing only simple disks are smooth manifolds of the expected
dimension. It is therefore enough to show that for generic $J$,
whenever the virtual dimension $\delta_{\mathcal{S}}(I)$ is $\leq 1$
we actually have: $\mathcal{P}^*_{\mathcal{S}}(I,\mathcal{F},J) =
\mathcal{P}_{\mathcal{S}}(I,\mathcal{F},J)$, i.e.  all the disks $u$
participating in elements of the moduli space
$\mathcal{P}_{\mathcal{S}}(I,\mathcal{F},J)$ are simple. This is
typically proved using the decomposition technique as follows. Assume
for simplicity that $\mathcal{S} = \textnormal{prl}$, $I=(x,y;A)$,
$\mathcal{F}=(f, \rho)$. Suppose by contradiction that one of the
disks $u$ participating in a pearly trajectory $\mathbf{w}=(w_1,
\ldots, w_l) \in \mathcal{P}_{\textnormal{prl}}(x,y;A;f, \rho, J)$ is
not simple.  We first decompose $u$ - in the sense above - into simple
disks, $v_1, \ldots, v_m$.  Then, it is possible to find among the
$v_i$'s a chain of disks, say $v_{i_0}, v_{i_{1}},\ldots$, such that
if we replace in $\mathbf{w}$ the disk $u$ by this chain we still get
a pearly trajectory $\mathbf{w}' \in
\mathcal{P}_{\textnormal{prl}}(x,y;A';f, \rho, J)$ connecting $x$ to
$y$.  Without loss of generality assume that all the disks in
$\mathbf{w}'$ are now simple (otherwise we repeat the same procedure).
By monotonicity, it follows that the total Maslov index decreases by
at least $2$, i.e.
$$\mu(A') \leq \mu(A) - N_L \leq \mu(A) - 2.$$
It follows that the virtual dimension of
$\mathcal{P}_{\textnormal{prl}}(x,y;A';f, \rho, J)$ becomes negative:
$$\delta_{\textnormal{prl}}(x,y;A') \leq
\delta_{\textnormal{prl}}(x,y;A) -2 \leq 1 -2 < 0.$$ By transversality
(this time for simple disks) it follows that
$\mathcal{P}_{\textnormal{prl}}(x,y;A';f, \rho,J) = \emptyset$, a
contradiction. Thus all elements $\mathbf{w} \in
\mathcal{P}_{\textnormal{prl}}(x,y;A; f, \rho, J)$ consist of simple
disks.

Similar arguments work also for the other types of moduli spaces
$\mathcal{P}_{\textnormal{prod}}$, $\mathcal{P}_{\textnormal{mod}}$
and $\mathcal{P}_{\textnormal{inc}}$. The main difference with respect
to $\mathcal{P}_{\textnormal{prl}}$ is that now some of the
$J$-holomorphic disks involved in these spaces have more marked
points. For example, in the case of $\mathcal{P}_{\textnormal{mod}}$
one of the disks has $-1$, $0$, $1$ as marked points. When applying
the preceding argument to such a disk $u$ we do not have good control
on how the corresponding marked points are distributed among the
$v_i$'s. Nevertheless by a combinatorially more involved argument it
is still possible to apply the previous procedure in order to show
that $\mathcal{P}^*_{\mathcal{S}}(I,\mathcal{F},J) =
\mathcal{P}_{\mathcal{S}}(I,\mathcal{F},J)$ for $\mathcal{S}=$
``prod'', ``mod'' and ``inc'', whenever $\delta_{\mathcal{S}}(I) \leq
1$, hence obtain transversality. The only exceptional case is
$S=\textnormal{mod}$, $N_{L}=2$, $\delta_{S}(I)=1$ which needs to be
treated by other methods: roughly, the reason is that
$\mathcal{P}_{\textnormal{mod}}$ consists of configurations containing
an interior marked point and the reduction to simple disks might
increase the degree of liberty of this point so that, as a
consequence, the dimension of the respective moduli spaces might not
drop by $2$ but just by $1$.

There is yet another source of complications. 
Transversality for moduli spaces whose elements
involve a single disk at a time can be obtained by restricting to
simple disks, but, when considering sequences of pseudo-holomorphic disks
$\mathbf{w} = (u_1, \ldots, u_l) \in
\mathcal{P}_{\mathcal{S}}(I,\mathcal{F},J)$ and various evaluation
maps, one has to add the assumption that the disks $u_1, \ldots, u_l$
are {\em absolutely distinct}. This means that for every $1 \leq i
\leq l$ we have $u_i(D) \not\subset \cup_{j \neq i} u_j(D)$. It turns
out that when the virtual dimension $\delta_{\mathcal{S}}$ is $\leq
1$, for generic $J$ all elements of the moduli spaces
$\mathcal{P}_{\mathcal{S}}$ indeed consist of absolutely distinct
disks. This is also proved using the monotonicity assumption by
similar arguments as above. If the disks in a sequences $\mathbf{w}
\in \mathcal{P}_{\mathcal{S}}(I, \mathcal{F}, J)$ are not absolutely
distinct then after a suitable omission of some of them we still get an
element $\mathbf{w}' \in
\mathcal{P}^{*}_{\mathcal{S}}(I',\mathcal{F},J)$ in which the disks
are absolutely distinct. The point is that by monotonicity, the
virtual dimension of $\mathcal{P}^{*}_{\mathcal{S}}(I',\mathcal{F},J)$
now becomes negative hence by transversality the latter space is
empty. A contradiction.

\medskip We refer the reader to~\cite{Bi-Co:qrel-long, Bi-Co:rigidity}
for more information and precise details on transversality in the
context of pearly moduli spaces.

\subsection{Compactness and gluing} \label{sb:comp-glue} These are
standard ingredients in Morse and Floer theory. In essence,
compactness and gluing give a precise description of the boundary of
the moduli spaces $\mathcal{P}_{\mathcal{S}}(I,\mathcal{F},J)$.

For simplicity we elaborate on the case $S=$ prl. A similar discussion
applies for the  spaces $\mathcal{P}_{\textnormal{prod}}$,
$\mathcal{P}_{\textnormal{mod}}$, $\mathcal{P}_{\textnormal{inc}}$. 
Here is a description of the boundary of the space
$\mathcal{P}_{\textnormal{prl}}(x,y;A; f, \rho, J)$. Below we
abbreviate $\mathcal{F}= (f, \rho)$. Let $\mathbf{w}_k = (u_{1,k},
\ldots, u_{l,k})$ be a sequence in
$\mathcal{P}_{\textnormal{prl}}(x,y;A;\mathcal{F},J)$
that does not have a converging subsequence in that space. Then, after
passing to a subsequence, still denoted $\mathbf{w}_k$, we have the
following possibilities:
\begin{enumerate}[(C-1)]
  \item One of the gradient trajectories of $f$ breaks at a new
   critical point $z \in \textnormal{Crit}(f)$, i.e. $\mathbf{w}_k$
   splits, as $k \to \infty$, into $\mathbf{w}', \mathbf{w}''$ where
   $\mathbf{w}' \in \mathcal{P}_{\textnormal{prl}}(x, z; B;
   \mathcal{F}, J)$, $\mathbf{w}'' \in
   \mathcal{P}_{\textnormal{prl}}(z, y; C, \mathcal{F}, J)$  and
   $B+C = A$.
   \label{i:comp-1}
  \item One of the gradient trajectories of $f$ connecting adjacent
   disks, say $u_{i,k}$ to $u_{i+1,k}$, shrinks to a point, i.e.
   $\mathbf{w}_k$ converges to $(\mathbf{w}', \mathbf{w}'')$, where
   $\mathbf{w}' = (u'_1, \ldots, u'_{l'}) \in
   \mathcal{P}_{\textnormal{prl}}(x, p; B; \mathcal{F}, J)$,
   $\mathbf{w}'' = (u''_1, \ldots, u''_{l''}) \in
   \mathcal{P}_{\textnormal{prl}}(p, y; C, \mathcal{F}, J)$, $l' \geq
   1$, $l'' \geq 1$, $l'+l'' = l$ and $p = u'_{l'}(1) = u''_1(-1)$ is
   (in general) not a critical point of $f$.  See the lefthand side of
   Figure~\ref{f:pearls-compactness}. Denote by
   $\mathcal{P}_{\textnormal{prl,C-\ref{i:comp-2}}}(x,y;(B,C);\mathcal{F},
   J)$ the space of such pairs $(\mathbf{w}', \mathbf{w}'')$ (after
   moding out by the obvious symmetries coming from reparametrizations
   of the disks). A simple computation shows that the virtual
   dimension of this space is:
   $$\delta_{\textnormal{prl,C-\ref{i:comp-2}}}(x,y;B,C) =
   |x|-|y|+\mu(B)+\mu(C)-2.$$
   \label{i:comp-2}
  \item Bubbling of a $J$-holomorphic disk occurs, i.e. there exists
   $1 \leq i \leq l$ such that the sequence $u_{i,k}$ converges to a
   reducible $J$-holomorphic curve consisting of two $J$-holomorphic
   disks $u_{i, \infty}$ and $u'_{i, \infty}$ attached to each other
   at a point on the boundary $\partial D$. Note that, apriori there
   are two possibilities for this attaching point. It may be either
   $\pm 1 \in \partial D$ (i.e. coincide with one of the marked points
   for elements of $\mathcal{P}_{\textnormal{prl}}$), or it may be
   another point $\tau \in \partial D \setminus \{-1, 1\}$.  The
   latter case is called {\em side bubbling}. See the righthand side
   of Figure~\ref{f:pearls-compactness}. In that case we can remove
   $u'_{i, \infty}$ from the limit and obtain a new pearly trajectory
   $\mathbf{w}$ connecting $x$ to $y$ whose total homology class is $A
   - [u'_{i,\infty}]$. Note that $\mathbf{w}$ belongs to a space whose
   virtual dimension is
   $$\delta_{\textnormal{prl}}(x,y;A-[u'_{i,\infty}]) =
   \delta_{\textnormal{prl}}(x,y;A) - \mu([u'_{i, \infty}]) \leq
   \delta_{\textnormal{prl}}(x,y;A) - 2.$$

   In the former case (i.e. bubbling occurs at $\tau = \pm 1 \in
   \partial D$) the limit can be described as a pair $(\mathbf{w}',
   \mathbf{w}'')$, whose total length is $l+1$, with the same
   description as elements of
   $\mathcal{P}_{\textnormal{prl,C-\ref{i:comp-2}}}(x,y;(B,C);\mathcal{F},
   J)$.  We denote by
   $\mathcal{P}_{\textnormal{prl,C-\ref{i:comp-3}'}}(x,y;(B,C);\mathcal{F},
   J)$ the space of such elements $(\mathbf{w}', \mathbf{w}'')$, where
   $B$, $C$ stand for the homology class of the sum of disks in
   $\mathbf{w}'$ and $\mathbf{w}''$ respectively. Although formally
   the space
   $\mathcal{P}_{\textnormal{prl,C-\ref{i:comp-3}'}}(x,y;(B,C);\mathcal{F},
   J)$ is the same as
   $\mathcal{P}_{\textnormal{prl,C-\ref{i:comp-2}}}(x,y;(B,C);\mathcal{F},
   J)$ we denote these two spaces differently, since the analytic
   reason for $\mathbf{w}_k$ converging to a point in each of them is
   different.
   \label{i:comp-3}
  \item Bubbling of a $J$-holomorphic sphere occurs in one of the
   disks $u_{i,k}$ of $\mathbf{w}_k$, either at an interior point or
   at a point on the boundary. If we denote by $C$ the class of the
   bubbled sphere, then, after removing this sphere from the limit, we
   obtain a pearly trajectory $\mathbf{w}$ connecting $x$ to $y$ and
   of total class $A-C$. Thus $\mathbf{w}$ belongs to a space whose
   virtual dimension is $$\delta_{\textnormal{prl}}(x,y;A-C) =
   \delta_{\textnormal{prl}}(x,y;A) - \mu(C) \leq
   \delta_{\textnormal{prl}}(x,y;A) - 2.$$
   \label{i:comp-4}
  \item A combination of~(C-\ref{i:comp-1})--(C-\ref{i:comp-4}) above,
   where each of these possibilities can occur repeatedly.
   \label{i:comp-5}
\end{enumerate}

\begin{figure}[htbp]
      \psfig{file=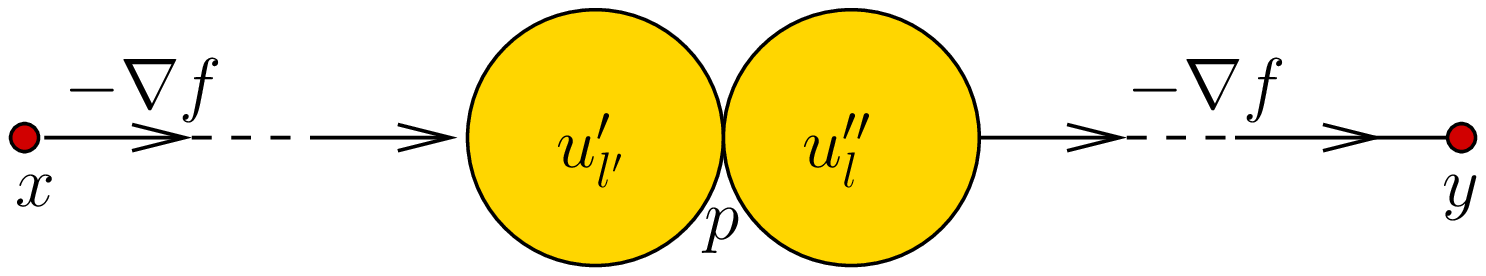, width=0.48 \linewidth} \quad
      \psfig{file=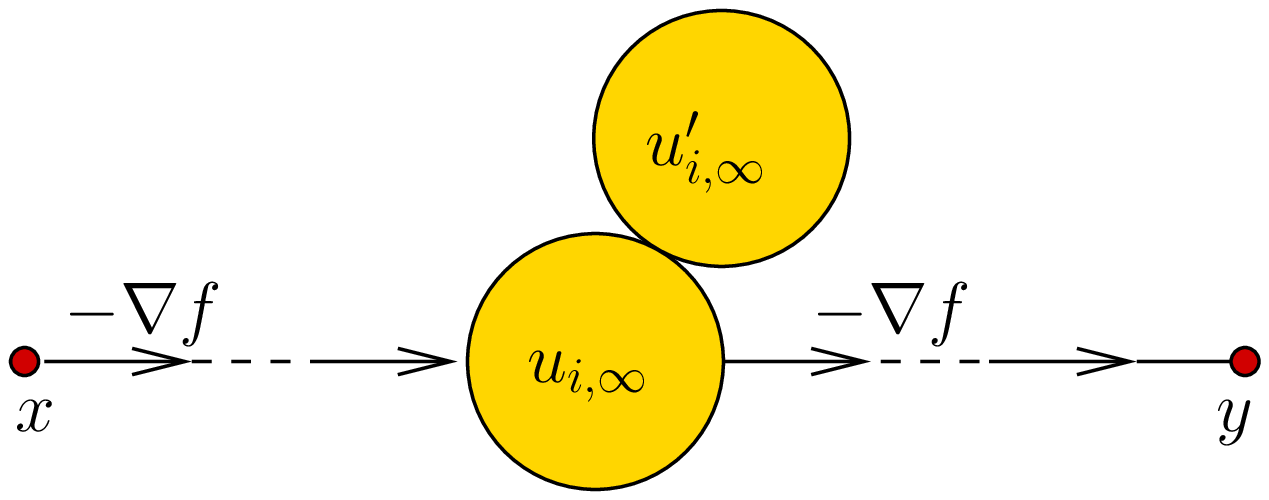, width=0.48 \linewidth}
      \caption{On the left: a gradient trajectory has shrunk to a
        point. On the right: side bubbling.}
      \label{f:pearls-compactness}
\end{figure}

Continuing with $\mathcal{S} =$ prl, assume now that
$\delta_{\textnormal{prl}}(x, y; A)=0$. We claim that for generic $J$
the space $\mathcal{P}_{\textnormal{prl}}(x,y;A;\mathcal{F},J)$ is
compact. To prove this we notice that in each of the cases
(C-\ref{i:comp-1})--(C-\ref{i:comp-5}) above the virtual dimension is
smaller than $\delta_{\textnormal{prl}}(x,y;A)$ by at least $1$ hence
negative. A transversality argument, similar to the ones in
\S\ref{sb:trans}, but this time for the spaces
$\mathcal{P}_{\textnormal{prl,C-\ref{i:comp-2}}}$,
$\mathcal{P}_{\textnormal{prl,C-\ref{i:comp-3}'}}$, shows that when
their virtual dimension is $\leq 1$ then for generic $J$ these spaces
are smooth manifolds of the expected dimension. In our case, since the
virtual dimension is negative, this means that they are just empty
sets. As none of the possibilities
(C-\ref{i:comp-1})--(C-\ref{i:comp-5}) can occur, the space
$\mathcal{P}_{\textnormal{prl}}(x, y; A; \mathcal{F}, J)$ is compact
(hence a finite set), as claimed by Proposition~\ref{p:transversality}
and used in formula~\eqref{eq:d-prl}.

Assume now that $\delta_{\textnormal{prl}}(x,y;A)=1$. We claim that
the $1$-dimensional manifold
$\mathcal{P}_{\textnormal{prl}}(x,y;A;\mathcal{F},J)$ can be
compactified into a manifold with boundary
$\overline{\mathcal{P}}_{\textnormal{prl}}$, whose boundary is:
\begin{equation} \label{eq:bndry-prl}
   \begin{aligned}
      \partial \overline{\mathcal{P}}_{\textnormal{prl}} =
      & \bigcup_{\substack{z \in \textnormal{Crit}(f), B+C = A \\
          \delta_{\textnormal{prl}}(x,z;B)=0 \\
          \delta_{\textnormal{prl}}(z,y;C)=0}}
      \mathcal{P}_{\textnormal{prl}}(x,z;B;\mathcal{F},J) \times
      \mathcal{P}_{\textnormal{prl}}(z,y;C;\mathcal{F},J) \; \; \coprod \\
      & \bigcup_{B+C=A}
      \mathcal{P}_{\textnormal{prl,C-\ref{i:comp-2}}}(x,y;B,C;\mathcal{F},J)
      \; \; \coprod \; \; \bigcup_{B+C=A}
      \mathcal{P}_{\textnormal{prl,C-\ref{i:comp-3}'}}(x,y;B,C;\mathcal{F},J).
   \end{aligned}
\end{equation}
To see this, first note that elements corresponding to possibility
(C-\ref{i:comp-4}) cannot occur in the boundary of
$\mathcal{P}_{\textnormal{prl}}(x,y;A;\mathcal{F},J)$ when
$\delta_{\textnormal{prl}}(x,y;A)=1$. The reason is that these
elements correspond to spaces of pearly trajectories whose virtual
dimension in (C-\ref{i:comp-4}) is $\leq
\delta_{\textnormal{prl}}(x,y;A) - 2 < 0$. Thus by transversality for
pearls these spaces are empty. For a similar reason side bubbling
cannot occur either. Thus, we are left with possibilities
(C-\ref{i:comp-1}), (C-\ref{i:comp-2}) and (C-\ref{i:comp-3}').  This
shows that $\partial \overline{\mathcal{P}}_{\textnormal{prl}} \subset
(\textnormal{RHS of } \eqref{eq:bndry-prl})$.

It remains to show that the opposite inclusion $\partial
\overline{\mathcal{P}}_{\textnormal{prl}} \supset (\textnormal{RHS of
} \eqref{eq:bndry-prl})$ holds too. This is a consequence of the
gluing procedure. The precise details of gluing are beyond the scope
of this paper and we skip the details. The fact that
$\overline{\mathcal{P}}_{\textnormal{prl}}$ as described
in~\eqref{eq:bndry-prl} has the structure of a $1$-dimensional
manifold with boundary, i.e. that each element on the RHS
of~\eqref{eq:bndry-prl} corresponds to a {\em unique} end of the
(possibly non-compact) manifold $\mathcal{P}_{\textnormal{prl}}$ is a
consequence of the so called surjectivity of the gluing map.

Here is some reference on the gluing procedure. Gluing of closed
pseudo-holomorphic curves is presented in a very detailed way
in~\cite{McD-Sa:Jhol-2}. Gluing of pseudo-holomorphic disks is
developed in~\cite{FO3} and further elaborated
in~\cite{Bi-Co:qrel-long} where is also treated the surjectivity of
the gluing map.

\subsection{Putting everything together and the main scheme of proof}
\label{sb:main-scheme}
We are now ready to prove that $d \circ d = 0$ as claimed by
Theorem~\ref{t:prl}. This final step is standard in Morse-Floer theory
and goes as follows. Fix $x \in \textnormal{Crit}(f)$. For every $y
\in \textnormal{Crit}(f)$ denote by $\langle d \circ d (x), y \rangle$
the coefficient (in $\Lambda$) of $y$ in $d \circ d (x) \in
\mathcal{C}(f, \rho, J)$. By definition:
\begin{equation} \label{eq:d2-y}
   \begin{gathered}
      \langle d \circ d (x), y \rangle \; \; = \sum_{\substack{z \in
          \textnormal{Crit}(f), B, C
          \\ \delta_{\textnormal{prl}}(x,z,B)=0 \\
          \delta_{\textnormal{prl}}(z,y,C)=0}} \#_{\mathbb{Z}_2}
      \mathcal{P}_{\textnormal{prl}}(x,z;B;\mathcal{F},J)
      \#_{\mathbb{Z}_2}
      \mathcal{P}_{\textnormal{prl}}(z,y;C;\mathcal{F},J)
      t^{\bar{\mu}(B)+\bar{\mu}(C)}  =  \\
      \Biggl( \sum_{\substack{z \in \textnormal{Crit}(f), B, C \\
          \mu(B+C) = 2 - |x| + |y| \\
          \delta_{\textnormal{prl}}(x,z,B)=0}} \#_{\mathbb{Z}_2}
      \mathcal{P}_{\textnormal{prl}}(x,z;B;\mathcal{F},J)
      \#_{\mathbb{Z}_2}
      \mathcal{P}_{\textnormal{prl}}(z,y;C;\mathcal{F},J) \Biggr)
      t^{(2-|x|+|y|)/N_L}.
   \end{gathered}
\end{equation}
The last equality here follows from the fact that
$\delta_{\textnormal{prl}}(x,z,B)=\delta_{\textnormal{prl}}(z,y,C)=0$
iff $\mu(B) + \mu(C) = 2 - |x| + |y|$ and
$\delta_{\textnormal{prl}}(x,z;B)=0$. Thus the factor
$t^{\bar{\mu}(B)+\bar{\mu}(C)}$ is constant and always equals
$t^{(2-|x|+|y|)/N_L}$.

To prove that the sum in~\eqref{eq:d2-y} is $0$ we use the
description~\eqref{eq:bndry-prl} of the boundary of
$\overline{\mathcal{P}}_{\textnormal{prl}}$. Fix $A \in H_2^D(M,L)$
with $\mu(A) = 2 - |x| + |y|$. Since
$\overline{\mathcal{P}}_{\textnormal{prl}}(x,y;A;\mathcal{F},J)$ is a compact
$1$-dimensional manifold with boundary, its boundary consists of an
even number of points. Thus, by~\eqref{eq:bndry-prl} we have:
\begin{equation} \label{eq:sum-bndry}
   \begin{aligned}
      0 \; = \; & \#_{\mathbb{Z}_2} \partial
      \overline{\mathcal{P}}_{\textnormal{prl}}(x,y;A;\mathcal{F},J) = \\
      & \sum_{\substack{z \in \textnormal{Crit}(f), B+C = A \\
          \delta_{\textnormal{prl}}(x,z;B)=0 \\
          \delta_{\textnormal{prl}}(z,y;C)=0}} \#_{\mathbb{Z}_2}
      \mathcal{P}_{\textnormal{prl}}(x,z;B;\mathcal{F},J)
      \#_{\mathbb{Z}_2}
      \mathcal{P}_{\textnormal{prl}}(z,y;C;\mathcal{F},J) \; \; + \\
      & \sum_{B+C=A} \#_{\mathbb{Z}_2}
      \mathcal{P}_{\textnormal{prl,C-\ref{i:comp-2}}}(x,y;B,C;\mathcal{F},J)
      \; \; + \; \; \sum_{B+C=A} \#_{\mathbb{Z}_2}
      \mathcal{P}_{\textnormal{prl,C-\ref{i:comp-3}'}}(x,y;B,C;\mathcal{F},J).
   \end{aligned}
\end{equation}
However, as noted in~\S\ref{sb:comp-glue} above (see case
(C-\ref{i:comp-3}) there) the spaces
$\mathcal{P}_{\textnormal{prl,C-\ref{i:comp-2}}}(x,y;B,C;\mathcal{F},J)$
and
$\mathcal{P}_{\textnormal{prl,C-\ref{i:comp-3}'}}(x,y;B,C;\mathcal{F},J)$
are actually two identical copies of the same space. Thus the sum on
the last line of~\eqref{eq:sum-bndry} vanishes (in $\mathbb{Z}_2$).
Summing now equality~\eqref{eq:sum-bndry} over all possible classes
$A$ with $\mu(A) = 2 - |x| + |y|$ we obtain the needed equality.  This
concludes the (outline of the) proof that $d \circ d = 0$. \Qed

\subsection{The identification maps} \label{sbsb:identif} As in Morse
theory, there are essentially two techniques to construct a comparison
chain morphism $$\psi_{\mathscr{D}', \mathscr{D}}:
\mathcal{C}_*(\mathscr{D}) \longrightarrow
\mathcal{C}_*(\mathscr{D}')$$ for every two generic triples
$\mathscr{D} = (f, \rho, J)$, $\mathscr{D}' = (f', \rho', J')$.
   
The first method is based on using Morse cobordisms. Such a cobordism
is a pair $(F,\bar{\rho})$ defined on the product: $ F:L\times
[0,1]\to \R$ and $\bar{\rho}$ a metric on $L\times [0,1]$ and so that
(up to the possible addition of an appropriate constant) we have
$(F,\bar{\rho})|_{L\times\{0\}}=(f,\rho)$ and
$(F,\bar{\rho})|_{L\times\{1\}}=(f',\rho')$; the pair $(F,\bar{\rho})$
is Morse-Smale and $\Crit_{i}(F)=
\Crit_{i-1}(f)\times\{0\}\cup\Crit_{i}(f)\times\{1\}$; $\frac{\partial
  F}{\partial t} (x,t)=0$ for $(x,t)\in L\times\{0,1\}$ and
$\frac{\partial F}{\partial t} (x,t)< 0$ if $t\in L\times (0,1)$. We
also consider a smooth one parametric family of $\omega$-compatible
almost complex structures $\bar{J}_{t}$ so that $\bar{J}_{0}=J$ and
$\bar{J}_{1}=J'$.  We then define pearl type moduli spaces as in
\S\ref{sb:prl-complex} but with a couple of modifications: the place
of the flow $\Phi$ is now taken by the flow $\bar{\Phi}$ on $L\times
[0,1]$ induced by $-\nabla_{\bar{\rho}} F$; the non-constant disks
$u_{i}$ are $\bar{J}_{\tau{i}}$-holomorphic where, as in the
definition of $\mathcal{P}_{\textnormal{prl}}$, $t_{i}$ is so that
$\bar{\Phi}_{t_{i}}(u_{i}(+1))=u_{i+1}(-1)$ and the parameter
$\tau_{i}$ is determined by
$\tau_{i}=\mathrm{pr}_{2}(\phi_{t_{i}}(u_{i}(+1)))$ with
$\mathrm{pr}_{2}:L\times [0,1]\to [0,1]$ the projection on the second
factor.  The transversality issues for these moduli spaces are
perfectly similar to those for the usual pearl moduli spaces. Under
generic choices for $F,\bar{\rho},\bar{J}$ counting (mod 2) the
elements in the $0$-dimensional such moduli spaces defines the chain
morphism $\psi_{\mathscr{D}', \mathscr{D}}$ as desired. The same
construction is then applied to cobordisms of Morse cobordisms and it
shows that the induced map in homology, $\Psi_{\mathscr{D}',
  \mathscr{D}}$, is canonical.
     
The second method is more direct. Given the two data sets $\mathscr{D}
= (f, \rho, J)$, $\mathscr{D}' = (f', \rho', J')$ we consider moduli
spaces consisting of triples $(\mathbf{u},\mathbf{v},p)$ with
$\mathbf{u}\in\mathcal{P}_{\textnormal{prl}}(x,p;A;f, \rho, J)$,
$\mathbf{v}\in \mathcal{P}_{\textnormal{prl}}(p,y';A';f', \rho', J')$,
$p\in L$.  It is easily seen that counting $0$-dimensional such
configurations gives a chain morphism:
$$\psi'_{\mathscr{D}', \mathscr{D}}: \mathcal{C}_*(\mathscr{D})
\longrightarrow \mathcal{C}_*(\mathscr{D}')~.~$$ The disadvantage of
this second method is that, in this case, it is harder to directly
check that the map induced in homology,
$\Psi'_{\mathscr{D}',\mathscr{D}}$, is canonical.  Moreover, for the
moduli spaces involved in the definition of the morphism $\psi'$ to be
regular, the two pairs $(f,\rho)$ and $(f',\rho')$ need to be generic
in the sense that the unstable manifolds of $f$ are required to be
transverse to the stable manifolds of $f'$.
   
However, it is not difficult to verify that
$\Psi'_{\mathscr{D}',\mathscr{D}}=\Psi_{\mathscr{D}',\mathscr{D}}$.
Thus, both methods produce the same (canonical) morphism in homology.

\subsection{Proving Theorems~\ref{t:qprod}-~\ref{t:qinc}}
\label{sbsb:prf-rest}
As mentioned earlier, the proofs of
Theorems~\ref{t:qprod}-~\ref{t:qinc} follow the same scheme as the
proof of $d^{2}=0$ for the differential of the pearl complex.  For
example, in order to show that each of the maps $\circ$, $\circledast$
and $i_L$ are chain maps we compactify $1$-dimensional spaces of the
type $\mathcal{P}_{S}(I, \mathcal{F}, J)$ into compact $1$-dimensional
manifolds with boundary. To prove the other statements such as the
associativity of $\circ$ in homology, the fact that $\circledast$ is
indeed a module operation etc. we follow the same scheme, but now we
have to work with other types of moduli spaces that haven't been
introduced explicitly above. The main difference is that some of the
pseudo-holomorphic disks will have now more marked points and,
possibly, there will be more than a single disk with several marked
points. We refer the reader to~\cite{Bi-Co:qrel-long, Bi-Co:rigidity}
for the precise details.  Transversality in these cases as well as in
the exceptional case $\mathcal{S}=\textnormal{mod}, N_{L}=2,
\delta_{\mathcal{S}}(I)=1$ can be achieved by the scheme described
before only after allowing that some of the curves in the chain of
pearls configurations carry Hamiltonian perturbations of the type
described in \cite{Ak-Sa:Loops}.

There is a
unified approach to all these issues which is based on trees.
More precisely we consider planar oriented trees whose
edges and vertices are labeled as follows. The edges are labeled by
Morse functions (some on $L$, some on $M$) and the inner vertices by
elements of $H_2^D(M,L)$. The entries and exit vertices are labeled
by critical points of the function corresponding to the adjacent edge.

Each such tree (or collection of trees) determines in a natural way a
moduli space involving gradient trajectories attached to
pseudo-holomorphic disks. An appropriate count of the number of
elements in the $0$-dimensional components of these spaces gives rise
to a quantum operation on the chain level. For example, the pearly
differential is defined by looking at linear trees (i.e. one entry and
one exit). The quantum product is defined by considering trees with
two entries and one exit all having valence $1$.

The advantage in modeling all the moduli spaces on trees is that most
of the arguments involving compactness, gluing and transversality can
be proved for large classes of trees and there is no need to repeat
small variations of each argument over and over again for each quantum
operation separately. This is particularly useful in dealing with the
moduli spaces that appear in the proof of the various associativity
relations involving the quantum product and the module structure as
well as to keep track of the Hamiltonian perturbations which are
required.  In ~\cite{Bi-Co:rigidity} this approach is described in
full, for all the moduli spaces needed for these operations as well as
for the relations among them.

The idea to model homological operations in Morse and Floer theory on
graphs is not new and has been implemented in various settings, see
e.g.~\cite{Be-Co:graph-morse} for the Morse case, ~\cite{FO3} for
Lagrangian Floer theory as well as  ~\cite{Cor-La:Cluster-1} (where 
the point of view is closest to that of the present paper).

\subsection{Identification with Floer homology} \label{sb:pss}

The version of Floer homology that we need is defined in the presence
of a Hamiltonian $H:M\times [0,1]\to \R$.  Consider the path space
$\mathcal{P}_{0}(L)=\{\gamma\in C^{\infty}( [0,1], M)\ | \
\gamma(0)\in L\ , \ \gamma(1)\in L \ , \ [\gamma]=0\in\pi_{2}(M,L)\}$
and inside it the set of contractible orbits $\mathcal{O}_{H}\subset
\mathcal{P}_{0}(L)$ of the Hamiltonian flow $X_{H}$. Assuming $H$ to
be generic we have that $\mathcal{O}_{H}$ is a finite set.  Fix some
almost complex structure $J$.  The Maslov index induces a morphism
$\mu:\pi_{1}\mathcal{P}_{0}(L)\to \Z$ and we let
$\tilde{\mathcal{P}}_{0}(L)$ be the regular, abelian cover associated
to $\ker (\mu)$, the group of deck transformations being
$\pi_{1}(\mathcal{P}_{0}(L))/\ker (\mu)$.  Consider all the lifts
$\tilde{x}\in\tilde{\mathcal{P}}_{0}(L)$ of the orbits $x\in
\mathcal{O}_{H}$ and let $\tilde{\mathcal{O}}_{H}$ be the set of these
lifts.  Fix a basepoint $\eta_{0}$ in $\tilde{\mathcal{P}}_{0}(L)$ and
define the degree of each element $\tilde{x}$ by
$|\tilde{x}|=\mu(\tilde{x},\eta_{0})$ with $\mu$ being here the
Viterbo-Maslov index.  The Floer complex is the $\La$-module:
$$CF_{\ast}(H,J)=\Z_{2}\langle\tilde{\mathcal{O}}_{H}\rangle$$
where $t^{r}\in\La$ acts on $\tilde{x}$ by
$t^{r}\tilde{x}=rN_{L}\cdot\tilde{x}$.  The differential is given by
$d x=\sum \#\mathcal{M}(\tilde{x},\tilde{y}) \tilde{y}$ where
$\mathcal{M}(\tilde{x},\tilde{y})$ is the moduli space of solutions
$u:\R\times [0,1]\to M$ of Floer's equation $\partial u/\partial s+J\
\partial u/\partial t+\nabla H(u,t)=0$ which verify $u(\R\times
\{0\})\subset L,\ u(\R\times\{1\})\subset L$ and they lift in
$\tilde{\mathcal{P}}_{0}(L)$ to paths relating $\tilde{x}$ and
$\tilde{y}$. Moreover, the sum is subject to the condition
$\mu(\tilde{x},\tilde{y}) -1=0$.

The comparison map from the pearl complex
$$\Phi_{f,H}:\mathcal{C}( f,\rho, J)\to CF(L;H,J)$$
is defined by the PSS method (see \cite{PSS} and, in the Lagrangian
case, \cite{Bar-Cor:NATO}, \cite{Cor-La:Cluster-1},\cite{Alb:PSS}) as
well as the map in the opposite direction
$$\Theta_{H,f}: CF(H,J) \to \mathcal{C}(f,\rho, J)~.~$$

For example, the value of the map $\Phi_{f,H}$ on the generator $x\in
\Crit (f)$ is defined by counting elements in ($0$-dimensional) moduli
spaces consisting of triples $(\mathbf{u},p,v)$ so that $p\in L$,
$\mathbf{u}\in\mathcal{P}_{\textnormal{prl}}(x,p;f,\rho,J)$ and $v$ is
a solution of the equation
\begin{equation}\label{eq:perturbed}
   \partial v/\partial s+ J\partial v/\partial t+\beta(s)\nabla H(v,t)=0
\end{equation}
so that $\beta:\R\to [0,1]$ is an appropriate increasing smooth
function supported in the interval $[-1,+\infty)$ and which is
constant equal to $1$ on $[1,+\infty)$. This solution $v$ has also to
verify $v(\R\times\{0\})\subset L$, $v(\R\times\{1\})\subset L$,
$\lim_{s\to\infty}v(s,-)=\gamma(-)$ and $\lim_{s\to
  -\infty}v(s,-)=p\in L$.  Transversality issues can be dealt with by
methods similar to those described in the case of the pearl complex.

The value of the map $\Theta$ on some element $\gamma\in
\mathcal{O}_{H}$ is given by using similar moduli spaces which now
consist of triples $(v,p,\mathbf{u})$, with $a\in L$, $\mathbf{u}\in
\mathcal{P}_{\textnormal{prl}}(p,y;f,\rho,J)$ and $v$ verifying an
equation like \eqref{eq:perturbed} but with the function $\beta$
replaced by $\beta'=1-\beta$ and $\lim_{s\to-\infty}v(s,-)=\gamma(-)$
and $\lim_{s\to \infty}v(s,-)=p$. Proving that these maps are chain
morphisms and that their compositions induce inverse maps in homology
depends, in the first instance, on using one-dimensional moduli spaces
as above and, in the second, on yet some other moduli spaces which
will produce the needed chain homotopies (see again
\cite{Bi-Co:qrel-long} and \cite{Bi-Co:rigidity} as well as
\cite{Alb:PSS} for details). It is easy to see that these morphisms
identify the module and quantum product as defined in ``pearl'' terms
with the analogue structures defined in Floer homology.

\section{Further structures} \label{s:further}

\subsection{Augmentation, duality and spectral sequences}
\label{sb:aug-etc}
There are a number of additional algebraic structures associated to
the quantum homology of a monotone Lagrangian $L$ and we review here
the most significant of them.

\subsubsection{Augmentation.} Given a pearl complex
$\mathcal{C}(f,\rho,J)=\Z_{2}\langle\Crit(f)\rangle\otimes \La$, define a map:
$$\epsilon_{L}:\mathcal{C}(f,\rho,J)\to \La$$
by $\epsilon_{L}(x)=1$ for all $x\in \Crit_{0}(f)$ and
$\epsilon_{L}(x)=0$ for all critical points of $f$ of strictly
positive index. It is easy to see that this is a chain map (where the
differential on $\La$ is trivial) and that the map induced in homology
- which is called the {\em augmentation} is canonical.

By using the augmentation it is easy to see that the quantum inclusion 
is actually determined by the module action. The following formula
is true
\begin{equation}\label{eq:incl_mod}
\langle PD(h),i_{L}(x)\rangle=\epsilon_{L}(h\circledast x)
\end{equation}
for all $h\in H_{\ast}(M;\Z_{2}),\ x\in QH(L)$ with $PD(-)$ Poincar\'e
duality and $\langle -.-\rangle$ the Kronecker pairing.
\subsubsection{Duality.} \label{sbsb:duality} Assuming defined the
chain complex $\mathcal{C}(f,\rho,J)$ the dual co-chain complex
associated to it is given by $$\mathcal{C}^{\ast}(f,\rho,J)
=(\hom_{\Z_{2}}(\Z_{2}\langle \Crit(f)\rangle, \Z_{2})\otimes \Lambda,
d^{\ast})$$ where if $x\in \Crit_{i}(f)$, then the degree of
$x^{\ast}\in \hom_{\Z_{2}}(\Z_{2}\langle \Crit(f)\rangle, \Z_{2})$,
the dual of $x$, is $i$; the differential $d^{\ast}$ is the dual of
$d$.  The co-homology of this complex is again canonical and it
computes, by definition, the {\em quantum co-homology} of $L$,
$QH^{\ast}(L)$.  Clearly, we have an evaluation $QH^{\ast}\otimes
QH_{\ast}\to \La$ which we write as $\sigma\otimes \alpha \to
\sigma(\alpha)$.

\begin{thm}\label{thm:duality}
   There is a canonical isomorphism $$\eta: QH_{k}(L)\to QH^{n-k}(L)$$
   which corresponds to the bilinear map: $\bar{\eta}:QH_{k}(L)\otimes
   QH_{k'}(L)\stackrel{\circ}{\longrightarrow}
   QH_{k+k'-n}(L)\stackrel{\epsilon_{L}}{\longrightarrow} \La$ via the
   relation $\eta(x)(y)=\bar{\eta}(x\otimes y)$.
\end{thm}
The isomorphism $\eta$ is obtained by composing the standard
comparison map $\psi_{-f,f}:\mathcal{C}(f,\rho,J)\to
\mathcal{C}(-f,\rho,J)$ (which is defined for generic choices of data
$f,\rho,J$ as in \S\ref{sbsb:identif}) with the identification of
$\mathcal{C}(-f,\rho,J)$ and $\mathcal{C}^{\ast}(f,\rho,J)$ induced by
$x\to x^{\ast},\ \forall x\in\Crit(-f)=\Crit(f)$. We refer to
\cite{Bi-Co:qrel-long} \cite{Bi-Co:rigidity} for full details.

The quantum inclusion, $i_{L}$, the duality map, $\eta$, and the
Lagrangian quantum product determine the module structure by the
following formula which extends \eqref{eq:incl_mod}:
\begin{equation}\label{eq:mod_inclusion}
   \langle PD(h), i_{L}(x\circ y) \rangle = 
   \langle \eta(y), h\circledast x \rangle
\end{equation} 
where $h\in H_{\ast}(M;\Z_{2})$, $x,y\in QH(L)$.

\subsubsection{Degree filtration and the associated spectral sequence.}
All the structures discussed in this paper are based on moduli spaces
which consist of configurations consisting of pseudo-holomorphic
objects joined together by Morse trajectories. In particular, all
these objects have a positive symplectic area and, if this area is
null, then they reduce to the classical Morse moduli spaces associated
to the structure in question.

As a consequence, all these structures respect the filtration of $\La$
by the degrees of $t$: $$\La^{k}=t^{k}\Z_{2}[t]~.~$$

In particular, the pearl complex $\mathcal{C}(f,\rho,J)$ is filtered by
$$F^{k}\mathcal{C}(f,\rho,J)=\Z_{2}<\Crit(f)>\otimes \La^{k}$$
and the pearl differential respects this filtration.  Thus there is a
spectral sequence associated to this filtration which converges to the
quantum homology of $L$ and whose term $(E^{0},d^{0})$ is just the
Morse complex of $f$ (tensored with $\La$).  This spectral sequence is
a variant of the spectral sequence introduced by Oh in
\cite{Oh:spectral}.  The quantum product as well as the module action
also respect this filtration.

\subsection{Other coefficient rings} \label{sb:coefficients} Here we
extend the quantum homology $QH(L)$ to larger coefficient rings which
also take into account the actual homology classes of the pearly
trajectories, not only their total Maslov index. As mentioned above,
since we are in the monotone case we can actually work with rings
taking into account the positivity of the Maslov index of
pseudo-holomorphic curves. Indeed, all our operations, differentials
and comparison maps - with the notable exception of the identification
map with Floer homology - only involve holomorphic objects and so only
involve classes for which the Maslov class is positive. The resulting
quantum homology of $L$ carries more information than $QH(L)$ as
defined in~\S\ref{sb:prl-complex}.

Let $H_2^S(M,L) \subset H_2(M;\mathbb{Z})$ be the image of the
Hurewicz homomorphism $\pi_2(M) \to H_2(M;\mathbb{Z})$, and
$H_2^S(M)^+ \subset H_2^S(M)$ the semi-group consisting of classes $A$
with $c_1(A) > 0$. Similarly, denote by $H_2^{D}(M,L)^+ \subset
H_2^D(M,L)$ the semi-group of elements $A$ with $\mu(A) > 0$.  Let
$\widetilde{\Gamma}^+ = \mathbb{Z}_2[H_2^S(M)^+] \cup \{ 1 \}$ be the
unitary ring obtained by adjoining a unit to the non-unitary group
ring $\mathbb{Z}_2[H_2^S(M)^+]$. Similarly we put
$\widetilde{\Lambda}^{+} = \mathbb{Z}_2[H_2^D(M,L)^+] \cup \{1\}$. We
write elements of $Q \in \widetilde{\Gamma}^{+}$ and $P \in
\widetilde{\Lambda}^{+}$ as ``polynomials'' in the formal variable $S$
and $T$:
$$Q(S) = a_0 + \sum_{c_1(A)>0} a_{A} S^{A}, \qquad P(T) = b_0 +
\sum_{\mu(B)>0} b_B T^B \qquad a_0, a_A, b_0, b_B \in \mathbb{Z}_2.$$
We endow these rings with the following grading: $$deg S^A = -2
c_1(A), \quad \deg T^B = -\mu(B).$$ Note that these rings are smaller
than the rings $\hat{\Gamma}^{\geq 0} = \mathbb{Z}_2[\{A | c_1(A) \geq
0\}]$ and $\hat{\Lambda}^{\geq 0} = \mathbb{Z}_2[\{B | \mu(B) \geq
0\}]$. For example, $\hat{\Lambda}^{\geq 0}$ and $\hat{\Gamma}^{\geq
  0}$ might have many non-trivial elements in degree $0$, whereas in
$\widetilde{\Gamma}^+$ and $\widetilde{\Lambda}^{+}$ the only such
element is $1$.

Let $QH(M;\widetilde{\Gamma}^{+})$ be the quantum homology of $M$ with
coefficients in $\widetilde{\Gamma}^{+}$ endowed with the quantum
product $\ast$.  We have a natural map $H_2^{S}(M)^+ \to
H_2^{D}(M,L)^+$ which induces on $\widetilde{\Lambda}^+$ a structure
of a $\widetilde{\Gamma}^{+}$-module. Put $QH(M;\widetilde{\Lambda}^+)
= QH(M;\widetilde{\Gamma}^{+}) \otimes_{\widetilde{\Gamma}^+}
\widetilde{\Lambda}^{+}$ and endow it with the quantum intersection
product, still denoted $*$ (defined e.g. as in~\cite{McD-Sa:Jhol-2}).
Note that the quantum product is well defined with this choice of
coefficients, since by monotonicity the only possible
pseudo-holomorphic sphere with Chern number $0$ is constant. We grade
this ring with the obvious grading coming from the two factors.

Given a triple $\mathscr{D} = (f, \rho, J)$ put
$\mathcal{C}(\mathscr{D}; \widetilde{\Lambda}^{+}) = \mathbb{Z}_2
\langle \textnormal{Crit}(f) \rangle \otimes \widetilde{\Lambda}^{+}$
endowed with the grading coming form both factors.  We define a map
$\widetilde{d}^+: \mathcal{C}_*(\mathscr{D}; \widetilde{\Lambda}^{+})
\longrightarrow \mathcal{C}_{*-1}(\mathscr{D};
\widetilde{\Lambda}^{+})$ by changing the differential $d$ in
formula~\eqref{eq:d-prl} as follows: instead of the coefficient
$t^{\bar{\mu}(A)}$ put $T^{A}$ for $\widetilde{d}^+$. Note that
$\widetilde{d}^+$ is well defined due to monotonicity. Indeed, if
$\mathbf{u}$ is a pearly trajectory with total homology class $A$ then
either $A=0$ or $\mu(A)>0$. Therefore $T^A \in \widetilde{\Lambda}^+$.

We alter all the other operations, $\circ$, $\circledast$ and $i_L$
described in~\S\ref{s:alg-struct} by rewriting all formulas with
the coefficient ring $\widetilde{\Lambda}^{+}$.

\begin{thm} \label{t:lambda+} The map $\widetilde{d}^+$ is a
   differential and the homology of $\mathcal{C}_*(\mathscr{D};
   \widetilde{\Lambda}^+, \widetilde{d}^+)$ denoted $QH_*(L;
   \widetilde{\Lambda}^+)$ is independent of the choice of the generic
   triple $\mathscr{D} = (f, \rho, J)$. Furthermore, all the
   statements in Theorems~\ref{t:prl}-~\ref{t:qinc}, except of the
   comparison $\Theta$ with $HF(L,L)$, continue to hold when replacing
   $QH(L)$ by $QH(L; \widetilde{\Lambda}^+)$ and $QH(M)$ by $QH(M;
   \widetilde{\Lambda}^+)$.
\end{thm}
The proof of this theorem is essentially the same as the proofs of
Theorems~\ref{t:prl}-~\ref{t:qinc}. The main point is that, not only
the total Maslov index, but also the total homology class is preserved
under bubbling as well as under gluing.

\medskip
Let $\mathcal{R}$ be a commutative $\widetilde{\Lambda}^+$-algebra.
Consider the complex $$\mathcal{C}(\mathscr{D}; \mathcal{R}) =
\mathcal{C}(\mathscr{D}; \widetilde{\Lambda}^+)
\otimes_{\widetilde{\Lambda}^+} \mathcal{R}$$ endowed with the
differential $d^{\mathcal{R}}$ induced from $\widetilde{d}^+$. We
denote the homology of this complex by $QH_*(L; \mathcal{R})$.
Finally we extend the coefficients of the quantum homology of the
ambient manifold by $QH(M; \mathcal{R}) = QH(M; \widetilde{\Lambda}^+)
\otimes_{\widetilde{\Lambda}^+} \mathcal{R}$. Clearly, the statement
of Theorem~\ref{t:lambda+} continues to hold when replacing
$\widetilde{\Lambda}^+$ by $\mathcal{R}$. Moreover, there exists a
canonical map $QH_*(L; \widetilde{\Lambda}^+) \to QH_*(L;
\mathcal{R})$ induced by the obvious ring homomorphism
$\widetilde{\Lambda}^+ \to \mathcal{R}$.

Here are a few examples of rings $\mathcal{R}$ which are useful in
applications. We endow a commutative ring $\mathcal{R}$ with the
structure of $\widetilde{\Lambda}^+$-algebra by specifying a ring
homomorphism $q: \widetilde{\Lambda}^+ \to \mathcal{R}$.
\begin{enumerate}
  \item Take $\mathcal{R} = \Lambda = \mathbb{Z}_2[t^{-1}, t]$, and
   define $q$ by $q(T^A) = t^{\bar{\mu}(A)}$. It is easy to see that
   $QH(L; \Lambda)$ coincides with our original homology $QH(L)$.
   \label{i:r=lambda}
  \item Take $\mathcal{R} = \mathbb{Z}_2[t]$, and define $q$ as
   in~\ref{i:r=lambda}. We denote this ring by $\Lambda^+$ and the
   resulting homology $QH(L;\Lambda^+)$ by $Q^{+}H(L)$.
  \item Take $\mathcal{R} = \mathbb{Z}_2[H_2^D(M,L)]$ with the obvious
   $\widetilde{\Lambda}^{+}$-algebra structure.
\end{enumerate}

\begin{rem} \label{r:non-vanish}
   \begin{enumerate}[i.]
     \item While $QH(L)$ is isomorphic to $HF(L,L)$ the relation of
      the homology $QH(L; \widetilde{\Lambda}^+)$ to $HF(L,L)$ is not
      straightforward. For example, while $HF(L,L)$ might vanish (e.g.
      when $L$ is displaceable) this is {\em never} the case for
      $QH(L; \widetilde{\Lambda}^{+})$. To see this recall that if $f:
      L \to \mathbb{R}$ is a Morse function with a single maximum $x$
      then $x \in \mathcal{C}_{n}(f, \rho, J; \widetilde{\Lambda}^+)$
      is a cycle and its homology class $[x]$ is the unity of
      $QH_*(L;\widetilde{\Lambda}^+)$. Thus
      $QH_*(L;\widetilde{\Lambda}^+)$ vanishes iff $x$ is a boundary.
      However, it is easy to see that for degree reasons
      $\mathcal{C}_{n+1}(f, \rho, J; \widetilde{\Lambda}^+) = 0$,
      hence $x$ cannot have a $\widetilde{d}^+$-primitive. The same
      remark applies to $QH(L; \mathcal{R})$ where $\mathcal{R}$ is a
      $\widetilde{\Lambda}^+$-algebra with no elements of negative
      degree e.g. $\mathcal{R} = \Lambda^+$.

      Let us mention that working with rings such as
      $\widetilde{\Lambda}^+$ in the context of Floer homology has
      been considered before in~\cite{FO3} where the relation between
      displacement energy of a Lagrangian and algebraic properties of
      the torsion of the corresponding Floer homology is studied.
     \item Let $\widehat{\Lambda} = \mathbb{Z}_2[H_2^D(M,L)]$ and let
      $\mathcal{R}$ be a $\widehat{\Lambda}$-algebra. Under these
      assumptions it is possible to define a version of Floer homology
      $HF(L,L;\mathcal{R})$ over $\mathcal{R}$ in an analogous way to
      the usual definition (see~\cite{FO3} as well
      as~\cite{Bi-Co:qrel-long}). On the other hand since
      $\widehat{\Lambda}$ is a $\widetilde{\Lambda}^+$-algebra so is
      $\mathcal{R}$ hence we can define also $QH(L; \mathcal{R})$. It
      turns out that the comparison map $\Theta$ (see
      Theorem~\ref{t:prl}) can be extended to this case. It gives a
      canonical isomorphism (upto a shift in grading)
      $HF_*(L,L;\mathcal{R}) \cong QH_*(L; \mathcal{R})$.
      \label{i:rem-theta}
     \item All the coefficient rings discussed above are based on
      group rings of commutative groups or semi-groups (in particular,
      $H_{2}^{D}(M,L)^{+}$). We could also have used directly the semi
      group $\pi_{2}(M,L)^{+}$ which consists of the elements of
      $\alpha\in\pi_{2}(M,L)$ so that $\omega(\alpha)>0$ or, when a
      group is required, $\pi_{2}(M,L)$. Both are, in general,
      non-commutative. For now, we have not used this non-commutative
      ring in applications and so we have only treated above the more
      familiar, commutative case.
      \end{enumerate}
\end{rem}

\subsubsection{An example} \label{sbsb:two-examples} Here is a
simple example which illustrate the various types of homologies
considered here.

Let $L \subset \mathbb{R}^2$ be an embedded circle. This is a monotone
Lagrangian with $N_L = 2$. We have $QH(L) \cong HF(L,L) = 0$ since $L$
is displaceable. Let us compute now $Q^+H(L)$ (i.e. $QH(L;
\Lambda^+)$, where $\Lambda^+ = \mathbb{Z}_2[t]$). Let $f:L
\longrightarrow \mathbb{R}^2$ be a Morse function with two critical
points: the maximum $x_1$ and the minimum $x_0$. Let $\rho$ be any
Riemannian metric on $L$ and $J$ any almost complex structure
compatible with the symplectic structure of $\mathbb{R}^2$. We have:
\begin{equation} \label{eq:C+-1}
   \mathcal{C}_i(f, \rho, J) =
   \begin{cases}
      0, & i \geq 2 \\
      \mathbb{Z}_2 x_1 t^{k}, & i = 1-2k, \, k \geq 0 \\
      \mathbb{Z}_2 x_0 t^{k}, & i = -2k, \, k \geq 0
   \end{cases}
\end{equation}
As for the differential $d^+$, we have: $d^+(x_1) = 0$, $d^+(x_0) =
x_1 t$. The first equality is because there are exactly two ($= 0 \in
\mathbb{Z}_2$) negative gradient trajectories going from $x_1$ to
$x_0$ and no other pearly trajectories form $x_1$ to $x_0$. As for
$d^+(x_0)$, it is easy to see that there is precisely one pearly
trajectory from $x_0$ to $x_1$. This trajectory starts at $x_0$, then
involves the (single) holomorphic disk spanning $L$ and then stops at
$x_1$. This disk has Maslov index $2$. It is easy to see that this is
the single pearly trajectory from $x_0$ to $x_1$. This proves that
$d^+(x_0) = x_1 t$.

Passing to homology, we see that $Q^+H_i(L) = 0$ for every $i \neq 1$,
and $Q^+H_1(L) \cong \mathbb{Z}_2[x_1]$. Clearly $[x_1]$ is a torsion
element in the sense that $t [x_1] = 0$.

\subsubsection{Other ground fields and rings} \label{sbsb:fields} All
the constructions and results in~\S\ref{s:alg-struct} are very much
likely to hold true if we replace the ground field by $\mathbb{Q}$ or
even $\mathbb{Z}$ provided that the Lagrangians $L$ are assumed to be
orientable and relative spin. Indeed due to~\cite{FO3} under these
conditions it is possible to orient in a coherent way the moduli
spaces of pseudo-holomorphic disks, hence also the pearly moduli
spaces $\mathcal{P}_{\textnormal{prl}}$,
$\mathcal{P}_{\textnormal{prod}}$, $\mathcal{P}_{\textnormal{mod}}$
and $\mathcal{P}_{\textnormal{inc}}$. The only thing that remains to
be rigorously verified is that these orientations are compatible with
the algebraic structures introduced in~\S\ref{s:alg-struct}.

\subsection{Relation to classical operations revisited}
\label{sb:specialization}
The relation of the quantum operations $\circ$, $\circledast$ and
$i_L$ to their classical counterparts which was discussed
in~\S\ref{sb:classical} can be further clarified using the ring
$\Lambda^+ = \mathbb{Z}_2[t]$. We view this ring as a
$\widetilde{\Lambda}^{+}$-algebra as explained in the preceding
section.

Fix a generic triple $(f, \rho, J)$. Note that the pearl complex with
coefficients in $\Lambda^+$ can be simply  written as:
\begin{equation} \label{eq:c+}
   \mathcal{C}(f, \rho, J; \Lambda^+) = \mathbb{Z}_2
   \langle \textnormal{Crit}(f) \rangle \otimes \mathbb{Z}_2[t], \quad d
   = \partial_0 + \partial_1 t + \cdots + \partial_{\nu} t^{\nu},
\end{equation}
where the operators $\partial_i$ are as described
in~\S\ref{sb:classical}.  Recall also that $\partial_0$ coincides with
the Morse-homology differential. Denote by $C(f,\rho) = \mathbb{Z}_2
\langle \textnormal{Crit}(f) \rangle$ the Morse complex (endowed with
the differential $\partial_0$).

Consider the specialization homomorphism $\mathbb{Z}_2[t] \to
\mathbb{Z}_2$ defined by $t \mapsto 0$. It induces a map
$\tilde{\sigma}: \mathcal{C}_*(f, \rho, J; \Lambda^+) \to C_*(f,
\rho)$. A simple computation based on~\eqref{eq:c+} shows that
$\tilde{\sigma}$ is a chain map hence induces a map in homology
$\sigma: Q^{+}H_*(L) \longrightarrow H_*(L;\mathbb{Z}_2)$. This map
can be viewed as a comparison map between the quantum structures and
the classical ones. For example, $\forall$ \; $\alpha, \beta \in
Q^+H(L), a \in QH(M; \Lambda^+)$ we have:
\begin{equation} \label{eq:sigma}
   \begin{aligned}
      \sigma(\alpha \circ \beta) = \sigma(\alpha) \cap \sigma(\beta),
      \quad \sigma(a \circledast \alpha) = a \widetilde{\cap}
      \sigma(\alpha), \quad \textnormal{inc}_*(\sigma(\alpha)) =
      \pi(i_L (\alpha)).
   \end{aligned}
\end{equation}
Here $\cap$ is the classical intersection product on
$H(L;\mathbb{Z}_2)$, $\widetilde{\cap}: H(M;\mathbb{Z}_2) \otimes
H(L;\mathbb{Z}_2) \longrightarrow H(L;\mathbb{Z}_2)$ is the exterior
intersection product defined by intersecting cycles in $M$ with cycles
in $L$. The map $\textnormal{inc}_* : H_*(L;\mathbb{Z}_2)
\longrightarrow H_*(M;\mathbb{Z}_2)$ is the canonical map induced by
the inclusion $L \subset M$. Finally, $\pi: QH(M;\Lambda^+) \to
H(M;\mathbb{Z}_2)$ is the projection onto the $H(M;\mathbb{Z}_2)$
summand of $QH*M;\Lambda^+)$ corresponding to $t=0$.

The proof of the identities in~\eqref{eq:sigma} follows immediately
from the discussion in~\S\ref{sb:classical}.

\subsection{Action estimates.}\label{subsec:action}
Given a monotone Lagrangian $L\subset (M,\omega)$ we have described in
\S\ref{sb:qmod} the quantum module structure:
\begin{equation}\label{eq:mod_prod_bis}QH(M;\La)\circledast QH(L)\to
   QH(L)~.~
\end{equation}
For a Hamiltonian $H:M\times S^{1}\to \R$ there is a PSS isomorphism
$\psi:QH(M)\to HF(H,J)$ defined when the pair $(H,J)$) is generic.
This suggests the definition of another product:
$$CF(M;H,J)\circledast \mathcal{C}(f,\rho,J)\to \mathcal{C}(f,\rho,J)$$  
which, in homology, should be identified with \eqref{eq:mod_prod_bis}
via the PSS map (here $f,\rho,J$ are so that the respective pearl
complex is well defined; $CF(M;H,J)$ is the periodic orbit Floer
complex associated to $(H,J)$ whose homology will be denoted by
$HF(M;H,J)$).  It is easy to see - as in \cite{Bi-Co:rigidity} - that
such a product can be defined by using pearls in which one of the
disks is replaced by a Floer half tube parametrized by $(-\infty,
0]\times S^{1}$ with the $-\infty$ end on a periodic orbit of $X_{H}$
and the $\{0\}\times S^{1}$ end on $L$.

 If a Floer half tube $u$ as above exists, then:
$$\int_{S^{1}\times\{0\}} H(x,t)dt\leq \mathcal{A}_{H}(\bar{\gamma})~.~$$
Here $lim_{s\to -\infty}u(s,-)=\gamma$, $\bar{\gamma}$ is the periodic
orbit $\gamma$ together with an appropriate capping and
$\mathcal{A}_{H}$ is the Floer action functional.

The interest in this construction comes from noticing that, as a
consequence of the remark above, if a class $a\in QH(M;\La)$ acts
non-trivially on $QH(L)$, then the spectral invariants associated to
$a$ can be bounded in terms of the behavior of the respective
Hamiltonians on $L$.  In turn, this has interesting geometric
applications.  We will not further discuss these issues here - the
whole topic is described in detail in \cite{Bi-Co:rigidity}.

\section{Applications I: topological rigidity of Lagrangian
  submanifolds}
\label{s:top-rig} 

The purpose of this section is to show that, in a manifold with
sufficiently rich quantum homology, even mild algebraic topological
conditions imposed to monotone Lagrangians restrict considerably their
homological or homotopical types.

\subsection{Homological $\R P^{n}$'s}
 Consider the complex projective space ${\mathbb{C}}P^n$ endowed with
its standard K\"{a}hler symplectic structure
$\omega_{\textnormal{FS}}$. Our first application deals with
Lagrangians $L \subset {\mathbb{C}}P^n$ whose first integral homology
$H_1(L;\mathbb{Z})$ satisfies $2 H_1(L;\mathbb{Z})=0$, i.e. $\forall\,
\alpha \in H_1(L;\mathbb{Z})$, $2 \alpha = 0$. A familiar example of
such a Lagrangian is $\mathbb{R}P^n \subset {\mathbb{C}}P^n$, $n \geq
2$. In fact, this is the only known example of a Lagrangian $L \subset
{\mathbb{C}}P^n$ with this property. The following theorem shows that at least
from the homological point of view this example is unique.

\begin{thm} \label{t:2H1=0} Let $L \subset {\mathbb{C}}P^n$ be a
   Lagrangian submanifold with $2H_{1}(L;\Z)=0$.
   \begin{enumerate}[i.]
     \item There exists a map $\phi:L\to \R P^{n}$ which induces an
      isomorphism of rings on $\mathbb{Z}_2$-homology: $\phi_*:
      H_*(L;\mathbb{Z}_2) \stackrel{\cong}{\longrightarrow}
      H_*(\mathbb{R}P^n;\mathbb{Z}_2)$, the ring structures being
      defined by the intersection product. In particular we have
      $H_i(L;\mathbb{Z}_2)=\mathbb{Z}_2$ for every $0\leq i \leq n$,
      and $H_*(L;\mathbb{Z}_2)$ is generated as a ring by
      $H_{n-1}(L;\mathbb{Z}_2)$. \label{i:cor:RP-H_*}
     \item Denote by $h = [{\mathbb{C}}P^{n-1}] \in
      H_{2n-2}({\mathbb{C}}P^n; \mathbb{Z}_2)$ the generator. Then $h
      \widetilde{\cap} [L]$ is the generator of
      $H_{n-2}(L;\mathbb{Z}_2)$. Here $\widetilde{\cap}$ stands for
      the exterior intersection product between elements of
      $H_*({\mathbb{C}}P^n;\mathbb{Z}_2)$ and $H_*(L;\mathbb{Z}_2)$.
      \label{i:h-cap}
     \item Denote by $\textnormal{inc}_*:H_i(L;\mathbb{Z}_2) \to
      H_i({\mathbb{C}}P^n;\mathbb{Z}_2)$ the homomorphism induced by
      the inclusion $L \subset {\mathbb{C}}P^n$. Then
      $\textnormal{inc}_*$ is an isomorphism for every $0 \leq
      i=$\,even $\leq n$. \label{i:inc-rpn}
   \end{enumerate}
\end{thm}
In view of this theorem it is tempting to conjecture that the only
Lagrangians $L \subset {\mathbb{C}}P^n$ with $2 H_1(L; \mathbb{Z})=0$
are homeomorphic (or diffeomorphic) to $\mathbb{R}P^n$, or more
daringly symplectically isotopic to the standard embedding of
$\mathbb{R}P^n \hookrightarrow {\mathbb{C}}P^n$.

Parts of Theorem~\ref{t:2H1=0} have been proved before by a variety
of methods by Seidel~\cite{Se:graded} and later on by
Biran~\cite{Bi:Nonintersections}. We will now outline a different
proof based on our theory. We refer the reader
to~\cite{Bi-Co:qrel-long, Bi-Co:rigidity} for the full details of the
proof.

We start with the following general observation. Recall that $QH(L)$
is a module over $QH(M; \Lambda)$. We have: {\sl suppose
  that $a \in QH_q(M;\Lambda)$ is an invertible element (of pure
  degree $q$).  Then the map $a \circledast (-)$ gives rise to {\em
    isomorphisms} $QH_i(L) \longrightarrow QH_{i+q-2n}(L)$ for every
  $i \in \mathbb{Z}$.}  This clearly follows from the general algebraic
notion of a ``module over a ring with unit''.

\begin{proof}[Outline of the proof of Theorem~\ref{t:2H1=0} 
   (see~\cite{Bi-Co:rigidity} for the complete proof)]
   
   A simple computation shows that $L$ is monotone with $N_L = k(n+1)$
   where $k$ is either $1$ or $2$.

   Denote by $h = [{\mathbb{C}}P^{n-1}] \in
   H_{2n-2}({\mathbb{C}}P^n;\mathbb{Z}_2)$ the class of an hyperplane.
   Recall from Example~\ref{ex:cpn} in~\S\ref{sb:qmod} that $h \in
   QH_{2n-2}({\mathbb{C}}P^n; \Lambda)$ is an invertible element.
   Therefore external multiplication by $h$ gives isomorphisms: $h
   \circledast (-) : QH_i(L) \longrightarrow QH_{i-2}(L)$. In other
   words, $QH_*(L)$ is $2$-periodic.
   
   Choose a generic triple $(f, \rho, J)$ with $f$ having exactly one
   minimum $x_0$ and one maximum $x_n$. Denote by $(C(f, \rho),
   \partial_0)$ the Morse complex and by $\mathcal{C}(f, \rho, J) =
   C(f, \rho) \otimes \Lambda$ the pearl complex endowed with the
   pearly differential $d$.  Recall from~\S\ref{sb:classical} that we
   can write $\partial = \partial_0 + \partial_1 t + \cdots +
   \partial_{\nu} t^{\nu}$, where each $\partial_j$ is an operator
   that sends $C_*(f, \rho)$ to $C_{*-1+j N_L}(f, \rho)$.

   We claim that $N_L = n+1$, i.e. $k = N_L / (n+1)$ must be $1$.
   Indeed, if $k=2$ then $\partial_j$, $j \geq 1$, must vanish since
   $-1+j N_L = -1 + 2(n+1)j > n$. This implies that $d = \partial_0$,
   hence $QH_*(L) = (H(L;\mathbb{Z}_2) \otimes \Lambda)_*$. In
   particular, we have
   $$QH_i(L) \cong H_i(L;\mathbb{Z}_2) \;\;\;\; 
   \forall \, 0 \leq i \leq n, \quad QH_j(L) = 0 \;\;\;\; \forall \,
   n+1 \leq j \leq 2n+1.$$ However, this contradicts the
   $2$-periodicity of $QH_*(L)$. This proves that $k = 1$.

   As $N_L = n+1$, the differential $d$ can be written as $d =
   \partial_0 + \partial_1 t$. For degree reasons $\partial_1(x) = 0$
   for every critical point $x$ with $|x| \geq 1$.  As for
   $\partial_1(x_0)$, it can be either $0$ or $x_n$. It follows that
   $QH_i(L) \cong H_i(L;\mathbb{Z}_2)$ for every $1 \leq i \leq n-1$.
   We claim that $\partial_1(x_0) = 0$ too. Indeed, if
   $\partial_1(x_0)=x_n$ then, as $\partial_0(x_0)=0$, we have $d(x_0)
   = x_n t$ hence $[x_n] = 0 \in QH(L)$. But $[x_n]$ is the unity of
   $QH(L)$, so $QH_*(L) = 0$. In particular we have
   $H_1(L;\mathbb{Z}_2) \cong QH_1(L) = 0$. But this cannot happen
   since this would imply that $H_1(L;\mathbb{Z})=0$ (recall that
   $2H_1(L;\mathbb{Z})=0$) which would in turn imply that $N_L =
   2C_{{\mathbb{C}}P^n} = 2(n+1)$, a contradiction. This proves that
   $\partial_1(x_0)=0$. Therefore, we have $d \equiv \partial_0$,
   hence $QH(L) = H(L;\mathbb{Z}_2) \otimes \Lambda$.

   Since $QH_*(L)$ is $2$-periodic we obtain: $H_{2i}(L;\mathbb{Z}_2)
   \cong H_0(L;\mathbb{Z}_2) = \mathbb{Z}_2$, whenever $0 \leq 2i \leq
   n$. Similarly, $H_{2i+1}(L;\mathbb{Z}_2) \cong QH_{-1}(L) \cong
   QH_{n}(L) \cong H_n(L;\mathbb{Z}_2) = \mathbb{Z}_2$, whenever $1
   \leq 2i+1 \leq n$. The isomorphism $QH_{-1} \cong QH_n$ here holds
   since $QH_*$ is, by definition, also $N_L = n+1$ periodic. Summing
   up, we have $H_i(L;\mathbb{Z}_2) \cong \mathbb{Z}_2$ for every $0
   \leq i \leq n$. Since $H_i(\mathbb{R}P^n;\mathbb{Z}_2) =
   \mathbb{Z}_2$ for every $i$, $0 \leq i \leq n$, this shows that
   $H_*(L;\mathbb{Z}_2) \cong H_*(\mathbb{R}P^n;\mathbb{Z}_2)$.
   
   Notice that we actually have $QH_{i}(L)\cong \Z_{2}$ for each $i\in\Z$. 
   Denote by $\alpha_i \in QH_i(L)$, $i \in \mathbb{Z}$, the
   corresponding generators. As $h \in QH_{2n-2}({\mathbb{C}}P^n;
   \Lambda)$ is invertible we have $h \circledast \alpha_i =
   \alpha_{i-2}$, for every $i \in \mathbb{Z}$. For degree reasons it
   follows that $h \widetilde{\cap} \alpha_j = \alpha_{j-2}$ for every
   $2 \leq j \leq n$. (See the discussion in~\S\ref{sb:classical}.)  A
   similar argument shows that $\alpha_{n-2}^{\cap j} = \alpha_{n-2j}$
   for every $0 \leq j \leq n/2$ and that $\alpha_{n-1} \cap
   \alpha_{n-2l} = \alpha_{n-2l-1}$ for every $0 \leq l \leq (n-1)/2$.
   Denote by $\alpha^i \in H^i(L;\mathbb{Z}_2)$ the generator. What we
   have just proved is equivalent to saying that $\alpha^2$ generates
   $H^{\textnormal{even}}(L;\mathbb{Z}_2)$ (with respect to the cup
   product) and that $\alpha^1 \cup
   H^{\textnormal{even}}(L;\mathbb{Z}_2) =
   H^{\textnormal{odd}}(L;\mathbb{Z}_2)$.
   
   By a purely topological argument (without any symplectic
   ingredients) one shows now that $\alpha^1 \cup \alpha^1 =
   \alpha^2$.  (this can be proved e.g. by a Bockstein exact sequence
   using the fact that $H_1(L;\mathbb{Z})$ is a non-trivial
   $2$-torsion group). Equivalently, this means that $\alpha_{n-1}
   \cap \alpha_{n-1} = \alpha_{n-2}$. Summing up the information up to
   now, we have that the $\mathbb{Z}_2$-homologies (resp.
   cohomologies) of $L$ and $\mathbb{R}P^n$ are isomorphic as rings
   with respect to the cup (resp. intersection) products.

   The statement on $\textnormal{inc}_*$ at point~\ref{i:inc-rpn} of
   the Theorem can be proved by similar arguments by looking at the
   quantum inclusion map $i_L : QH_*(L) \to
   QH_*({\mathbb{C}}P^n;\Lambda)$.

   Finally, the fact that the isomorphism $H_*(L;\mathbb{Z}_2) \cong
   H_*({\mathbb{R}}P^n; \mathbb{Z}_2)$ is induced by a map  
    $\phi: L \to \mathbb{R}P^n$ follows from general algebraic
   topology. The argument is as follows. Let $\bar{\phi}: L \to
   K(\mathbb{Z}_2, 1) = \mathbb{R}P^{\infty}$ be the classifying map
   associated to $\alpha^1$, so that $\bar{\phi}^* c = \alpha^1$,
   where $c \in H^1({\mathbb{R}}P^{\infty}; \mathbb{Z}_2)$ is a
   fundamental class. As $\dim L = n$, $\bar{\phi}$ factors through a
   map $\phi: L \to \mathbb{R}P^n$ which still satisfies $\phi^*
   \gamma^1 = \alpha^1$, where $\gamma^1 = c|_{{\mathbb{R}}P^n} \in
   H^1(\mathbb{R}P^n;\mathbb{Z}_2)$ is the generator. As both
   $\alpha^1$ and $\gamma^1$ generate their respective cohomology
   rings it immediately follows that $\phi^*:
   H^*({\mathbb{R}}P^n;\mathbb{Z}_2) \to H^*(L;\mathbb{Z}_2)$ is an
   isomorphism.
\end{proof}

\subsection{Homological spheres in the quadric.}
Another case which exemplifies topological rigidity is that of the
quadric. Consider the smooth complex quadric $Q \subset
{\mathbb{C}}P^{n+1}$ endowed with the induced symplectic structure
from ${\mathbb{C}}P^{n+1}$. Note that $Q$ contains Lagrangians with
$H_1(L;\mathbb{Z})=0$, e.g. Lagrangian spheres. (To see this, write
$Q$ as $\{ z_0^2 + \cdots + z_n^2 = z_{n+1}^2\}$ and take $L = Q \cap
{\mathbb{R}}P^{n+1}$.) The next theorem shows that homologically this
is the only example.
\begin{thm} \label{t:quadric} Let $L \subset Q$, $n \geq 2$, be a
   Lagrangian submanifold with $H_1(L;\mathbb{Z})=0$.  Assume that $n
   = \dim_{\mathbb{C}}Q$ is even. Then $H_*(L;\mathbb{Z}_2) \cong
   H_*(S^n; \mathbb{Z}_2)$.
\end{thm}
The proof is based on similar ideas to that of Theorem~\ref{t:2H1=0}.
The main point now is that the class of a point $[pt] \in QH_0(Q;
\Lambda)$ is invertible. See~\cite{Bi-Co:qrel-long, Bi-Co:rigidity}
for a detailed proof.
It is likely that a similar statement holds for $n$ odd but our methods do not 
yield information in that case.
% LocalWords:  Bockstein

\section{Applications II: existence of holomorphic disks and
  symplectic packing}
\label{s:disks-pack}

In this section we explain how to use the theory
of~\S\ref{s:alg-struct} in order to prove existence of holomorphic
disks satisfying various incidence constrains. These in turn have
applications to relative symplectic packing. Below we
give a sample of our results in this direction. More complete and
general results, as well as detailed proofs, can be found
in~\cite{Bi-Co:qrel-long,Bi-Co:rigidity}.

\subsection{Existence of disks with pointwise constrains}
\label{sb:disks}

Our first result deals with Lagrangians as in Theorem~\ref{t:2H1=0}.
We recall again the familiar example of ${\mathbb{R}}P^n \subset
{\mathbb{C}}P^n$ for which we know for example that through every two
points passes a real algebraic line, i.e. a holomorphic disk which is
``half'' of a projective line (hence has Maslov index $n+1$). The
following theorem states, among other things, that this continues to
be so for generic almost complex structures and moreover that it is
actually true for all Lagrangians $L \subset {\mathbb{C}}P^n$ with
$2$-torsion $H_1(L;\mathbb{Z})$.

\begin{thm} \label{t:disks-2H1=0} Let $L \subset {\mathbb{C}}P^n$ be a
   Lagrangian with $2 H_1(L;\mathbb{Z})=0$. Then there exists a second
   category subset $\mathcal{J}_{\textnormal{reg}} \subset
   \mathcal{J}$ such that for every $J \in
   \mathcal{J}_{\textnormal{reg}}$ the following holds:
   \begin{enumerate}[i.]
     \item For every $p \in {\mathbb{C}}P^n \setminus L$ there exists
      a $J$-holomorphic disk $u:(D, \partial D) \to ({\mathbb{C}}P^n,
      L)$ with $u(\textnormal{Int\,} D) \ni p$ and $\mu([u]) = n+1$.
      \label{i:disks-2H1=0-1}
     \item For every two distinct points $x, y \in L$ there exists a
      $J$-holomorphic disks $u$ with $u(\partial D) \ni x, y$ and
      $\mu([u]) = n+1$. The number of such disks $u$ with
      $u(-1)=x$, $u(1)=y$, up to reparametrization is even.
      \label{i:disks-2H1=0-2}
     \item If $n=2$ then for every $p \in {\mathbb{C}}P^2 \setminus L$
      and $x, y \in L$ there exists a $J$-holomorphic disk $u$ with
      $u(\textnormal{Int\,} D) \ni p$, $u(\partial D) \ni x, y$ and
      $\mu([u]) \leq 6$. \label{i:disks-2H1=0-3}
   \end{enumerate}
\end{thm}
This theorem adds more evidence to the tempting conjecture, motivated
by Theorem~\ref{t:2H1=0}, that $L = {\mathbb{R}}P^n$ is in some sense
the unique example of a Lagrangian in ${\mathbb{C}}P^n$ with
$2$-torsion first homology.
 
\medskip

The next result is about the Clifford torus
$$\mathbb{T}_{\textnormal{clif}} = \bigl\{ [z_0: \cdots: z_n] \in
{\mathbb{C}}P^n \mid |z_0| = \cdots = |z_n| \bigr \} \subset
{\mathbb{C}}P^n.$$ This is a monotone Lagrangian torus with minimal
Maslov number $N = 2$. 

\begin{thm} \label{t:disk-clif} There exists a second category subset $\mathcal{J}_{\textnormal{reg}}\subset\mathcal{J}$ such that for every $J \in \mathcal{J}_{\textnormal{reg}}$ the
   following holds:
   \begin{enumerate}[i.]
     \item For every $p \in {\mathbb{C}}P^n \setminus
      \mathbb{T}_{\textnormal{clif}}$ there exists a $J$-holomorphic
      disk $u$ with $u(\textnormal{Int\,} D) \ni p$ and $\mu([u]) \leq
      2n$.
     \item For every $x \in \mathbb{T}_{\textnormal{clif}}$ there
      exists a $J$-holomorphic disk $u$ with $u(\partial D) \ni x$ and
      $\mu([u])=2$.
     \item If $n=2$ then for every $p \in {\mathbb{C}}P^2 \setminus
      \mathbb{T}_{\textnormal{clif}}$ and $x \in
      \mathbb{T}_{\textnormal{clif}}$ there exists a $J$-holomorphic
      disk $u$ with $u(\textnormal{Int\,}D) \ni p$, $u(\partial D) \ni
      x$ and $\mu([u]) \leq 4$.
   \end{enumerate}
\end{thm}

\medskip Finally, consider the smooth complex quadric $Q \subset
{\mathbb{C}}P^{n+1}$ endowed with the induced symplectic structure
from ${\mathbb{C}}P^{n+1}$.
\begin{thm} \label{t:disk-quad} Let $L \subset Q$ be a Lagrangian with
   $H_1(L;\mathbb{Z})=0$. Assume that $n = \dim L \geq 2$. Then there
   exists a second category subset $\mathcal{J}_{\textnormal{reg}}
   \subset \mathcal{J}$ such for every $J \in
   \mathcal{J}_{\textnormal{reg}}$ the following holds:
   \begin{enumerate}[i.]
     \item For every $p \in Q \setminus L$ and $x \in L$, there exists
      a $J$-holomorphic disk $u$ with $u(\textnormal{Int\,} D) \ni p$,
      $u(\partial D) \ni x$ and $\mu([u]) =2n$.
     \item If $n=$ even then for every three distinct points $x, y, z
      \in L$ there exists a $J$-holomorphic disk $u$ with $u(\partial
      D) \ni x, y, z$ and $\mu([u]) = 2n$.
   \end{enumerate}
\end{thm}

We will outline the proofs of some of the theorems above
in~\S\ref{sb:prf-disks} below. Before doing this we present some
immediate applications to symplectic packing.

\subsection{Relative symplectic packing} \label{sb:rel-pack} Let
$(M^{2n}, \omega)$ be a $2n$-dimensional symplectic manifold and $L
\subset M$ a Lagrangian submanifold. Denote by $B(r) \subset
\mathbb{R}^{2n}$ the closed $2n$-dimensional Euclidean ball of radius
$r$ endowed with the standard symplectic structure
$\omega_{\textnormal{std}}$ of $\mathbb{R}^{2n}$. Denote by
$B_{\mathbb{R}}(r) \subset B(r)$ the ``real'' part of $B(r)$, i.e.
$B_{\mathbb{R}}(r) = B(r) \cap (\mathbb{R}^n \times 0)$. Note that
$B_{\mathbb{R}}(r)$ is Lagrangian in $B(r)$. By a {\em relative
  symplectic embedding} $\varphi: (B(r), B_{\mathbb{R}}(r)) \to (M,L)$
of a ball in $(M,L)$ we mean a symplectic embedding $\varphi: B(r) \to
(M, \omega)$ which satisfies $\varphi^{-1}(L) = B_{\mathbb{R}}(r)$.
By analogy with the (absolute) Gromov width, we define here the Gromov
width of $L \subset M$ to be
$$w(L) = \sup \{ \pi r^2 \mid \exists \textnormal{ a relative symplectic
  embedding } (B(r), B_{\mathbb{R}}(r)) \to (M,L)\}.$$ We will
consider also symplectic embeddings of balls in the complement of $L$,
i.e. symplectic embeddings $\psi:B(r) \to M\setminus L$.  The
 Gromov width of $M \setminus L$ is: 
 $$w(M\setminus L) = \sup \{ \pi r^2 \mid \exists \textnormal{ a symplectic
  embedding } B(r) \to (M \setminus L) \}.$$

A natural generalization is to consider embeddings of several balls
with pairwise disjoint images i.e. symplectic packing. Let $l, m \geq
0$ and $r_1, \ldots, r_l>0$, $\rho_1, \ldots, \rho_m >0$.  A {\em
  mixed symplectic packing} of $(M,L)$ by balls of radii $(r_1,
\ldots, r_l; \rho_1, \ldots, \rho_m)$ is given by $l$ relative
symplectic embeddings $\varphi_i: (B(r_i), B_{\mathbb{R}}(r_{i})) \to (M,L)$,
$i=1, \ldots, l$, and $m$ symplectic embeddings $\varphi_j:B(r_j) \to
M \setminus L$, $j=l+1, \ldots, l+m$, such that the images
of all the $\varphi_k$'s are mutually disjoint, i.e.: $\varphi_{k'}
(B(r_{k'})) \cap \varphi_{k''} (B(r_{k''})) = \emptyset$ for every $k'
\neq k''$.

The following proposition provides a link between symplectic packing
and existence of holomorphic disks passing through given points.  It
is a straightforward generalization of Gromov's original approach to
symplectic packing~\cite{Gr:phol}.
\begin{prop}[See~\cite{Bi-Co:qrel-long, Bi-Co:rigidity}]
   \label{p:disk-pack-2}
   Let $L \subset (M, \omega)$ be a Lagrangian submanifold and $E>0$.
   Suppose that there exists a dense subset $\mathcal{J}_* \subset
   \mathcal{J}(M, \omega)$, a dense subset of $m$-tuples $\mathcal{U}'
   \subset (M \setminus L)^{\times m}$, and a dense subset of
   $l$-tuples $\mathcal{U}'' \subset L^{\times l}$ such that for every
   $J \in \mathcal{J}_*$, $(p_1, \ldots, p_m) \in \mathcal{U}'$,
   $(q_1, \ldots, q_l) \in \mathcal{U}''$ there exists a
   $J$-holomorphic disk $u:(D, \partial D) \to (M, L)$ with
   $u(\textnormal{Int\,}D) \ni p_1, \ldots, p_m$, $u(\partial D) \ni
   q_1, \ldots, q_l$ and $\textnormal{Area}_{\omega}(u) \leq E$. Then
   for every mixed symplectic packing of $(M,L)$ by balls of radii
   $(r_1, \ldots, r_l; \rho_1, \ldots, \rho_m)$ we have:
   $$\sum_{i=1}^l \frac{\pi r_i^2}{2} + \sum_{j=1}^m \pi \rho_j^2 \leq
   E.$$
\end{prop}

Combining Theorems~\ref{t:disks-2H1=0}--~\ref{t:disk-quad} with
Proposition~\ref{p:disk-pack-2} we obtain the following packing
inequalities. Below we normalize the symplectic structure
$\omega_{\textnormal{FS}}$ of ${\mathbb{C}}P^n$ so that
$\int_{\mathbb{C}P^1} \omega_{\textnormal{FS}} = \pi$. With this
normalization we have $({\mathbb{C}}P^n \setminus {\mathbb{C}}P^{n-1},
\omega_{\textnormal{FS}}) \approx (\textnormal{Int\,} B^{2n}(1),
\omega_{\textnormal{std}})$, hence $w({\mathbb{C}}P^n) = 1$.

\begin{cor} \label{c:pack}
   \begin{enumerate}[i.]
     \item If $L \subset {\mathbb{C}}P^n$ is a Lagrangian with
      $2H_1(L;\mathbb{Z})=0$ then we have $w({\mathbb{C}}P^n \setminus
      L) \leq \frac{1}{2}$.
     \item For $\mathbb{T}_{\textnormal{clif}}\subset \C P^{2}$ we have
     $w(\mathbb{T}_{\textnormal{clif}}) = \frac{2}{n+1}$,
      $w({\mathbb{C}}P^n \setminus \mathbb{T}_{\textnormal{clif}}) =
      \frac{n}{n+1}$.
     \item For every mixed symplectic packing of $({\mathbb{C}}P^2,
      \mathbb{T}_{\textnormal{clif}})$ by two balls of radii $(r;
      \rho)$ we have $\pi r^2 + \frac{1}{2}\pi\rho^2 \leq \frac{2}{3}$.
     \item Let $L \subset Q$ be a Lagrangian with
      $H_1(L;\mathbb{Z})=0$, and assume that $n = \dim L =$ even. Then
      for every relative symplectic packing of $(Q,L)$ by $3$ balls of
      radii $(\rho_1, \rho_2, \rho_3)$ we have $\pi(\rho_1^2 + \rho_2^2 +
      \rho_3^2) \leq 2$.
   \end{enumerate}
\end{cor}

The phenomenon that the Gromov width may decrease after removing a
Lagrangian submanifold was discovered in \cite{Bi:Barriers} where it
was proved for example that $w(\C P^{n}\setminus \R
P^{n})=\frac{1}{2}$.

\subsection{How to prove existence of disks satisfying pointwise
  constrains}
\label{sb:prf-disks}

We will outline here the proof of points~\ref{i:disks-2H1=0-1}
and~\ref{i:disks-2H1=0-2} of Theorem~\ref{t:disks-2H1=0}.  We refer
the reader to~\cite{Bi-Co:rigidity, Bi-Co:qrel-long} for the detailed
proofs.

Let $L \subset {\mathbb{C}}P^n$ be a Lagrangian with $2
H_1(L;\mathbb{Z})=0$. Recall from the proof of Theorem~\ref{t:2H1=0}
that $L$ is monotone with $N_L = n+1$ and that $QH_i(L) \cong
\mathbb{Z}_2$ for every $i \in \mathbb{Z}$. Denote by $\alpha_i \in
QH_i(L)$ the generator. Note that $t \in \Lambda$ has $\deg t =
-(n+1)$ so that $QH_j(L) \cong QH_{j+n+1}(L)t$. In particular
$\alpha_j = \alpha_{j+n+1}t$ for every $j \in \mathbb{Z}$.

Denote by $[pt] \in QH_0({\mathbb{C}}P^n; \Lambda)$ the class of a
point. Recall that $[pt]$ is an invertible element, hence we have
$[pt] \circledast \alpha_i = \alpha_{i-2n}$ for every $i \in
\mathbb{Z}$. In particular
\begin{equation} \label{eq:p*alpha_n} [pt] \circledast \alpha_n =
   \alpha_{-n} = \alpha_1 t.
\end{equation}

Let $f: L \to \mathbb{R}$ be a Morse function with one maximum $x_n$
and $h: {\mathbb{C}}P^n \to \mathbb{R}$ a Morse function with one
minimum at the point $p$. Choose two Riemannian metrics $\rho_L$ and
$\rho_M$. We make these choices so that $(h, \rho_M, f, \rho_L)$
satisfy the Assumption~\ref{a:generic} in~\S\ref{sb:trans}. Choose a
generic $J \in \mathcal{J}$. With these choices we have $[pt]=[p]$,
$\alpha_n = [x_n]$. From~\eqref{eq:p*alpha_n} it follows that there
exists a critical point $x_1 \in \textnormal{Crit}(f)$ of index
$|x_1|=1$ and $A \in H_2^D(M,L)$ with $\mu(A)=n+1$ such that the
moduli space $\mathcal{P}_{\textnormal{prl}}(p, x_n, x_1; A; h,
\rho_M, f, \rho_L, J)$, introduced in~\S\ref{sb:qmod}, is non-empty.
Let $(u_1, \ldots, u_l;k) \in \mathcal{P}_{\textnormal{prl}}(p, x_n,
x_1; A; h, \rho_M, f, \rho_L, J)$. By definition the $J$-holomorphic
disk $u_k$ satisfies $u_{k}(0) \in W_{p}^u(h)$. Since $p$ is the minimum of
$h$, we have $W_{p}^u(h) = \{ p \}$, hence $u_k(0) = p$. Clearly $\mu([u_k]) \leq
\mu(A)=n+1$.  But as $u_{k}$ can not be constant we actually have 
$\mu([u_{k}])=n+1$. The disk $u=u_k$ satisfies the statement in
point~\ref{i:disks-2H1=0-1} of Theorem~\ref{t:disks-2H1=0}.

We turn to the proof of the statement at point~\ref{i:disks-2H1=0-2}
of the theorem. Recall from the proof of Theorem~\ref{t:2H1=0} that
$QH_*(L) \cong (H(L;\mathbb{Z}_2) \otimes \Lambda)_*$. Recall also
that $\alpha_{n-1} \cap \alpha_{n-1} = \alpha_{n-2}$, where $\cap$ is
the classical intersection product on $H_*(L;\mathbb{Z}_2)$.

We now claim that $\alpha_0 \circ \alpha_0 = \alpha_1 t$. To see this,
first note that for degree reasons we have $\alpha_{n-1} \circ
\alpha_{n-1} = \alpha_{n-1} \cap \alpha_{n-1} = \alpha_{n-2}$.
Therefore:
\begin{align*}
   \alpha_0 \circ \alpha_0 & = ([pt] \circledast \alpha_{n-1} t^{-1})
   \circ ([pt] \circledast \alpha_{n-1} t^{-1}) = [pt]\circledast
   ([pt]\circledast (\alpha_{n-1} \circ \alpha_{n-1}))t^{-2} \\
   & = [pt]\circledast ([pt]\circledast \alpha_{n-2}) t^{-2} = [pt]
   \circledast \alpha_{-1}t^{-1} = \alpha_1 t.
\end{align*}

Pick a generic triple of Morse functions $f, f', f''$ on $L$ such that
$f$ and $f'$ each have a single minimum, $f$ at $x$ and $f'$ at $y$. Then, we
have that $[x]  \in QH_0(L;f, \rho_L, J)$ and $[y] \in QH_0(L; f',
\rho_L, J)$ both represent $\alpha_0 \in QH_0(L)$. As $\alpha_0 \circ
\alpha_0 = \alpha_1 t$, it follows that there exists $z \in
\textnormal{Crit}(f'')$ with $|z|=1$ and $A \in H_2^D(M,L)$ with
$\mu(A) = n+1$ such that the moduli space
$\mathcal{P}_{\textnormal{prod}}(x,y,z;A; f, f', f'', \rho_L, J)$,
introduced in~\S\ref{sb:qprod}, is non-empty. As $x$ and $y$ are both
minima of their functions it easily follows that for every element
$(\mathbf{u}, \mathbf{u}', \mathbf{u}'', v) \in
\mathcal{P}_{\textnormal{prod}}(x,y,z;A; f, f', f'', \rho_L, J)$ we
have $\mathbf{u}, \mathbf{u}', \mathbf{u}'' \equiv \textnormal{const}$
hence the $J$-holomorphic disk $v$ satisfies $x,y \in v(\partial D)$
and $\mu([v]) = n+1$. The point is again that as $x$ is a minimum we
have $W_{x}^u = \{ x \}$ and similarly for $y$.  The disk $v$ satisfies the properties stated at
point~\ref{i:disks-2H1=0-2} of the theorem (where the claimed disk was called
$u$).

It remains to show that the number of such disks it even. To prove
this pick a Morse function $g: L \to \mathbb{R}$ with a single minimum at $x$,
and a single maximum at $y$. Write the pearl differential $d$ as $d = \partial_0 +
\partial_1 t$ as in the proof of Theorem~\ref{t:2H1=0}. Clearly
$\partial_1 (x)$ counts the number of $J$-holomorphic disks (up to
reparametrization) $u: (D, \partial D) \to ({\mathbb{C}}P^n, L)$ with
$u(-1) = x$ and $u(1)=y$. However, as we saw in the proof of
Theorem~\ref{t:2H1=0} we have $\partial_1 = 0$, hence the number of
these disks is even.  \Qed

A proof of the statement at point~\ref{i:disks-2H1=0-3} of
Theorem~\ref{t:disks-2H1=0} as well as proofs of
Theorems~\ref{t:disk-clif} and~\ref{t:disk-quad} can be found
in~\cite{Bi-Co:qrel-long, Bi-Co:rigidity}.

% LocalWords:  clif const

\section{Applications III: relative enumerative invariants for
  Lagrangian tori} \label{s:enum-tori}

The purpose of this section is to present a general scheme that can be
used to construct numerical invariants associated to monotone
Lagrangians which is based on our machinery. We apply this scheme to
the case of $2$-dimensional tori.

\subsection{How to produce relative numerical invariants for wide
  Lagrangians}\label{subsec:strategy}
We will assume that $L$ is a monotone Lagrangian such that $QH(L)\cong
H_{\ast}(L;\Z_{2})\otimes \La$.  Such Lagrangians are called {\em
  wide} and it has been shown in \cite{Bi-Co:qrel-long} (see also
\cite{Bi-Co:rigidity}) that a large class of Lagrangians, tori in
particular, can only be {\em narrow} - in the sense that $QH(L)=0$ -
or wide.

With this assumption, and supposing that the quantum product in
$QH(L)$ is known, the naive way to produce numerical invariants would
be to replicate the construction in the closed case: pick first a
basis $\{a_{i}\}$ for $ H_{\ast}(L;Z_{2})$ and express the quantum
product as $a_{i}\circ a_{j}=\sum s(i,j;h:k) a_{h}t^{k}$ with
$s(i,j;h:k)\in \Z_{2}$; secondly, interpret $s(i,j;h:k)$ as the
(algebraic) number of $J$-holomorphic disks of Maslov class $kN_{L}$
through any cycles representing the classes $a_{i}, a_{j},
a_{h}^{\ast}$ (where $a_{h}^{\ast}$ is the dual of $a_{h}$). This
strategy fails for two reasons and it is instructive to understand
them in detail.

The first reason is quite obvious: the pearl moduli spaces consist of
configurations involving not only a single $J$-holomorphic curve but
also chains of curves joined together by Morse flow lines.  As a
consequence, the structural constants $s(i,j;h:k)$ can not be
interpreted directly as counts of disks with pointwise constraints. 
Recall also that the reason these configurations of chains of curves
are needed is that moduli spaces of disks have co-dimension one boundaries.

The second reason is much less obvious: the identification between
$QH(L)$ and $H_{\ast}(L;\Z_{2})\otimes\La$ is not canonical and so the
constants $s(i,j;h:k)$ as defined above are, in fact, not invariant !
This is a more subtle phenomenon and to describe it more precisely we
will now assume additionally that $L$ admits a perfect Morse function
$f:L\to \R$.  In this case, the isomorphism $QH(L)\cong
H_{\ast}(L;\Z_{2})\otimes \La$ translates to the fact that the pearl
complex $\mathcal{C}(f,\rho,J)$ (when defined) has a vanishing
differential. In particular, each critical point $x\in \Crit(f)$
represents not only a singular homology class (because $f$ is perfect)
but also a quantum homology class. Assume now that $f'$ is another
perfect Morse function so that the pearl complex
$\mathcal{C}(f',\rho,J)$ is also defined (obviously, it also has a
vanishing differential).  We already know from \S\ref{sbsb:identif}
that there is a chain morphism $\psi_{f',f}:\mathcal{C}(f,\rho,J)\to
\mathcal{C}(f',\rho,J)$ which induces a canonical isomorphism in
homology. In our case, as the pearl differentials vanish,
$\psi_{f',f}$ is itself a canonical isomorphism. In general, this
isomorphism has the form
$\psi_{f',f}=\psi^{M}_{f',f}+t\psi^{Q}_{f',f}$ where $\psi^{M}_{f',f}$
is the Morse comparison morphism.  Now, the key point here - and this
is specific to the ``open'' case - is that the quantum contribution
$\psi^{Q}_{f',f}$ is in general not zero ! Thus, while the structural
constants of the quantum product $QH(L)\otimes QH(L)\to QH(L)$ are
obviously invariant they can not be seen directly as invariants
counting pearly configurations through singular cycles because, even
if two singular cycles represent the same singular homology class,
they might represent different quantum classes.

\

We now describe a strategy which bypasses the two difficulties
described above and leads to numerical invariants.  We emphasize that,
for the moment, this is a strategy and not an algorithm.

We will continue to assume that $L$ is wide and admits a perfect Morse
function if this last property is not satisfied, there is a purely algebraic
minimal model technique which can be used instead  \cite{Bi-Co:rigidity}\cite{Bi-Co:qrel-long}.

Our approach consists of two steps which we describe below. Both
depend only on the minimal Maslov number $N=N_{L}$ and of the singular
homology $\mathcal{H}=H_{\ast}(L;\Z_{2})$ of $L$.  We fix some
algebraic notation. We put $\La= \Z_{2}[t,t^{-1}]$ and let
$\La^{+}=\Z_{2}[t]$ with $deg(t)=-N$. For a free $\La^{+}$-module $V$,
let $Aut^{+}_{0}(V)$ be the $\La^{+}$-automorphisms $\xi$ of $V$,
$\xi:V\to V$, which preserve degree and verify $\xi|_{V/tV}=id$. Let now
$V'=V\otimes_{\La^{+}} \La$.  This is clearly a free $\La$-module. We
denote by $Aut_{0}(V')$ the $\La$-module automorphisms of $V'$ which
are in the image of $Aut^{+}_{0}(V)$.

\begin{itemize}
  \item[i.] Pick a basis $\{a_{i}\}$ for $\mathcal{H}\otimes \La$ (as $\La$ module)
   and write the general form of the quantum product 
   $(\mathcal{H}\otimes \La)\otimes (\mathcal{H}\otimes \La)\to
   (\mathcal{H}\otimes \La)$ in this basis. The structural constants
   $s(i,j;h:k)$ appear as discussed above.  Find expressions
   $E(\ldots, s(i,j;h:k),\ldots)$ written in the constants
   $s(i,j;h:k)$ which are invariant by all the automorphisms $\xi\in
   Aut_{0}(\mathcal{H}\otimes\La)$ in the sense that, if the
   structural constants in the basis $\xi(a_{i})$ are $s'(i,j;h:k)$,
   then $E(\ldots, s(i,j;h:k),\ldots)=E(\ldots, s'(i,j;h:k),\ldots)$
   (even if not all $s'(i,j;h:k)=s(i,j;h:k)$).
  \item[ii.]  Let $f,f',f'':L\to \R$ be perfect Morse functions and
   let $\rho$, $J$ be generic.  Use the pearly description of the
   product
   $$\circ : \mathcal{C}(f,\rho,J)\otimes 
   \mathcal{C}(f',\rho,J)\to \mathcal{C}(f'',\rho,J)$$ to provide a
   geometric interpretation of the invariant expressions detected at
   point i.  in terms of counts of disks with various incidence
   conditions (in general, disks of different Maslov classes will
   appear in the same count).
\end{itemize}
As the counts given at point ii. are left invariant by the
automorphisms of $QH(L)$ which are induced by the comparisons which
appear at changes of the data $(f,\rho,J)$, it follows that each of
them provides a numerical invariant for all Lagrangians of Maslov
class $N$ and singular homology $\mathcal{H}$ (i.e. a number
independent of $f,f',f'', J, \rho$). It is important to emphasize that
due to the associativity of the quantum product the structural
constants $s(i,j;h:k)$ are not independent and this is a source of
relations among the various invariants constructed as above.

Once the two steps above are achieved, computing the invariants for a
specific Lagrangian with the fixed homology $\mathcal{H}$ and Maslov
class $N$ reduces to the computation of the quantum product
$QH(L)\otimes QH(L)\to QH(L)$.

This approach will be pursued systematically elsewhere. We will
describe it here only in the case of $2$-tori.

\subsection{Numerical invariants for tori.}
It turns out (see again \cite{Bi-Co:qrel-long},\cite{Bi-Co:rigidity})
that for a monotone Lagrangian torus $T$ to be wide, the Maslov number
$N_{T}$ has to be equal to $2$ so we assume this here.  Thus, to
implement the step i. in our strategy we fix a basis $m,a,b,w$ for
$H_{\ast}(T;\Z_{2})\otimes\La$ so that $a,b$ form a basis for
$H_{1}(T;Z_{2})$, $w\in H_{2}(T;\Z_{2})$ is the generator and $m\in
(H_{\ast}(T;\Z_{2})\otimes\La)_{0}$ together with $wt$ form a basis
for $H_{0}(T;\Z_{2})\oplus H_{2}(T;\Z_{2})t$. Notice that, in this case,
for degree reasons, the isomorphism $QH_{1}(L)\cong H_{1}(L;\Z_{2})
\otimes \La$ is canonical.
 
We now fix the notation for the structural constants involved in the
quantum product.  To do so we recall that this product is a
deformation of the usual intersection product at the chain level (hence, in this case
also at the homology level) and that $w$ is the unit.

We now write: $a\circ a = \alpha wt$, $b\circ b=\beta w t$, $a\circ b=
m+\gamma' wt$, $b\circ a=m+\gamma'' w t$ and we use the associativity
of the quantum product to deduce:
\begin{equation} \label{Eq:quant-prod-formulae-2}
   \begin{aligned}
      & m\circ a = \alpha b t+\gamma''a t, \quad a\circ m = \alpha
      b t + \gamma' a t \\
      & m\circ b = \beta a t + \gamma' b t, \quad b\circ m = \beta
      a t + \gamma'' b t \\
      & m\circ m = (\gamma'+\gamma'') m t + (\alpha\beta +
      \gamma'\gamma'')w t^2.
   \end{aligned}
\end{equation}
For further use, we fix the notation $s_{1}=\gamma'+\gamma''$ and
$s_{2}=\alpha\beta+\gamma'\gamma''$.  Let $\xi \in
Aut_{0}(H_{\ast}(T;\Z_{2})\otimes \La)$. There are only two
possibilities for such an automorphism as, for degree reasons, the
only quantum contribution in $\xi$ can appear in $\xi(m)=m+\epsilon w
t$, $\epsilon \in \Z_{2}$.  Let $\xi_{1}$ be the automorphism for
which $\epsilon=1$ (when $\epsilon=0$ the corresponding automorphism
is the identity).  It is immediate to see that, for degree reasons
$\alpha$ and $\beta$ are invariant with respect to $\xi_{1}$ and thus,
they are invariant in the sense of the step i. of \S
\ref{subsec:strategy}.  Let us remark that $\gamma'+\gamma''$ is also
invariant in the same sense. To see this write
\begin{equation}\label{eq:change_basis_prod}
   \xi_{1}(m)\circ \xi_{1}(m)=m\circ m +w t^{2}=
   s_{1}mt + (s_{2}+1)w t^{2}=s_{1}\xi_{1}(m)t +(s_{1}+s_{2} +1)\xi_{1}(w)t^{2}
\end{equation} 
and so the structural constant $s_{1}=\gamma'+\gamma''$ is invariant.
At the same time, individually, the constants $\gamma', \gamma''$ are
not necessarily invariant: indeed, if $\gamma'=1$, we have
$\xi_{1}(a)\circ \xi_{1}(b)= \xi_{1}(m)$ (while, for invariance, 
we would need $\xi_{1}(a)\circ\xi_{1}(b)=\xi_{1}(m)+\gamma'\xi_{1}(w)t$).

We now can proceed to the second step described in
\S\ref{subsec:strategy} and provide a geometric description for each
of these three invariants $\alpha$, $\beta$ and $\gamma'+\gamma''$.

Fix a basis $a', b'$ of the {\em integral} homology $H_1(T;\Z)$
which correspond after mod 2 reduction to the $a,b$ above.
Fix a point $x\in T$ and for some almost complex structure $J$
compatible with $\omega$ let $\mathcal{E}_{2}(x)$ be the set of
$J$-holomorphic disks $u$ with boundary on $T$ passing through $x$ and with 
$\mu([u])=2$.
Define a function $\nu:\mathbb{Z} \oplus \mathbb{Z} \to \mathbb{Z}_2$
as follows:
\begin{equation} \label{Eq:nu-function} \nu(k,l) =
   \#_{\mathbb{Z}_2}\{u\in \mathcal{E}_{2}(x) \ | \ 
  [ u(\partial D)]=ka'+lb'\} .
\end{equation}
where $J \in \mathcal{J}_{\textnormal{reg}}$ is a generic almost
complex structure. As $2$ is the minimal Maslov class, $\nu(k,l)$ does
not depend on the choice of $J \in \mathcal{J}_{\textnormal{reg}}$ or
on the choice of the point $x$ (this follows by a standard cobordism argument).
 Moreover, $\nu(k,l)=0$ for all but a
finite number of pairs $(k,l)$.

\begin{thm}(see \cite{Bi-Co:qrel-long})\label{T:nu-function-product} The
   coefficients $\alpha, \beta$ are given by:
   $$\alpha = \sum_{k,l} \nu(k,l) \frac{l(l+1)}{2} \,(\bmod{2}), \quad
   \beta = \sum_{k,l} \nu(k,l) \frac{k(k+1)}{2} \,(\bmod{2}).$$
   The
   sum $\gamma'+\gamma''$ is given by:
      $$\gamma'+\gamma'' = \sum_{k,l} \nu(k,l)kl \,(\bmod{2}).$$
\end{thm}

Notice also that $\gamma'+\gamma''$ is precisely the obstruction to the
commutativity of the quantum product. Moreover,  when this
product is non-commutative (thus when $\gamma'+\gamma''=1$) we have
from formula \eqref{eq:change_basis_prod} that $s_{2}$ si also an
invariant and is equal to $\alpha\beta$.

Till now the geometric interpretation of both $s_{1}$ and $s_{2}$ has
been based only on the formulae \eqref{Eq:quant-prod-formulae-2}
which, in turn, are based on the associativity of the quantum product.
However, - as indicated at the step ii. in \S\ref{subsec:strategy} -
both $s_{1}$ and $s_{2}$ have also geometric interpretations based
directly on the definition of the quantum product $m\circ m$. We
describe these interpretations next.

Let $\Delta$ be a triangle embedded in the torus $T$ with vertices
$A,B,C$ and with edges $AB$, $BC$, $CA$. For a fixed, generic almost
complex structure $J$ let $n_{\Delta}$ be the number (mod 2) of disks
of Maslov class $4$ passing, {\em in order} through the three points
$A,B,C$. Let $n_{A}$ be the number mod 2 (up to reparametrization)
 of $J$-disks $u$ of Maslov
class class $2$ with boundary on $L$ and with $u(-1)= A$, $u(+1)\in
BC$ (for generic $J$ both numbers are finite and the intersections of
the disks going through $A$ with the opposite edge is transverse).
Similarly, let $n_{B},n_{C}$ be the same numbers associated to the
other vertices of $\Delta$.
 
\begin{thm}(see \cite{Bi-Co:qrel-long})
   We have the formulae:
   $$ s_{1}=n_{A}+n_{B}+n_{C}\ $$ 
   Thus, the sum $n_{A}+n_{B}+n_{C}$ is independent of $J$ and
   $\Delta$.  If $s_{1}=1$, then $s_{2}$ is invariant and it equals 
   $$s_{2}=n_{\Delta}+n_{B}n_{C}~.~$$
   Thus, in this case, $n_{\Delta}+n_{B}n_{C}$ is also
   independent of $J$ and $\Delta$ and equals the product
   $\alpha\beta$.
\end{thm}
 
\begin{rem}
 a.  An interesting consequence of the formulae above is that if the
   quantum multiplication in $QH(L)$ is non-commutative, then the
   number of $J$-holomorphic disks of Maslov index $4$ passing in order
    through any three distinct points $A,B,C$ in $L$ can be computed out of the
   numbers $n_{A}, n_{B}, \alpha,\beta$ which only involve Maslov $2$
   disks. Moreover, the term $n_{B}n_{C}$ is exactly the correction
   needed to be added to the number of Maslov $4$ disks to obtain an invariant.
   
   b. Another nice consequence is that, for the same type of monotone Lagrangian 
   torus  as at point a. (i.e. $s_{1}=1$) the number of disks of Maslov class $4$ 
   through any three points is always even.
   Indeed, for a triangle $\Delta=ABC$ as above let $n'_{\Delta}$ be the number of 
   such disks going {\em in order} through $A, C, B$. Clearly, we have $s_{2}=n'_{\Delta}+n_{C}n_{B}$.
   Thus the total number (mod 2) of disks of Maslov $4$ through the three points is 
   $n_{\Delta}+n'_{\Delta}= 2s_{2}+2n_{C}n_{B}=0\in \Z_{2}$.  
\end{rem}
 
{\it Sketch of proof of $s_{2}=n_{\Delta}+n_{B}n_{C}$.}  We refer to
\cite{Bi-Co:qrel-long} for the rest of the proof of the theorem and
for additional details.  Let $f,g:T\to \R$ be two perfect Morse
functions with pairwise distinct critical points.  Let $x_{0}$ be the
minimum of $f$, let $x_{2}$ be the maximum of $f$, let $y_{0}$ be the
minimum of $g$. We may assume that the choices of $f,g$ as well as
that of the Riemannian metric $\rho$ are such that $y_{0}=A$,
$x_{0}=B$, $x_{2}=C$ and the edge $CA$ is the the unique flow line of
$-\nabla f$ going from $x_{2}$ to $y_{0}$, and (after slightly
rounding the corner at $A$) the edge $AB$ is the unique flow line
going from $y_{0}$ to $x_{0}$.

Notice that the product $\circ$ defined in \S\ref{sb:qprod} is also
defined when $f''=f$. In our case, the product we are interested in
is:
$$\mathcal{C}(f,\rho,J)\otimes \mathcal{C}(g,\rho,J)\to\mathcal{C}(f,\rho,J)$$
and we want to list all the configurations which give the coefficient
of $x_{2}t^{2}$ in $x_{0}\circ y_{0}$.  It is not hard to see that
there are precisely two types of such configurations:
\begin{itemize}
  \item[i.] disks of Maslov class $4$ passing, in order through
   $x_{0},x_{2},y_{0}$.
  \item[ii.] configurations made out of a disk of Maslov class $2$
   going through $x_{0}$ followed by a negative gradient flow line of
   $f$ going through $y_{0}$ which continues till it reaches a second
   disk of Maslov class $2$ which goes through $x_{2}$.
\end{itemize}

Clearly, the number of configurations of type i. is precisely
$n_{\Delta}$. A little thought (and a look at Figure
\ref{f:coef-s2-2}) shows that the configurations of type ii. are
precisely those counted by $n_{B}n_{C}$.

\begin{figure}[htbp]
   \begin{center}
      \epsfig{file=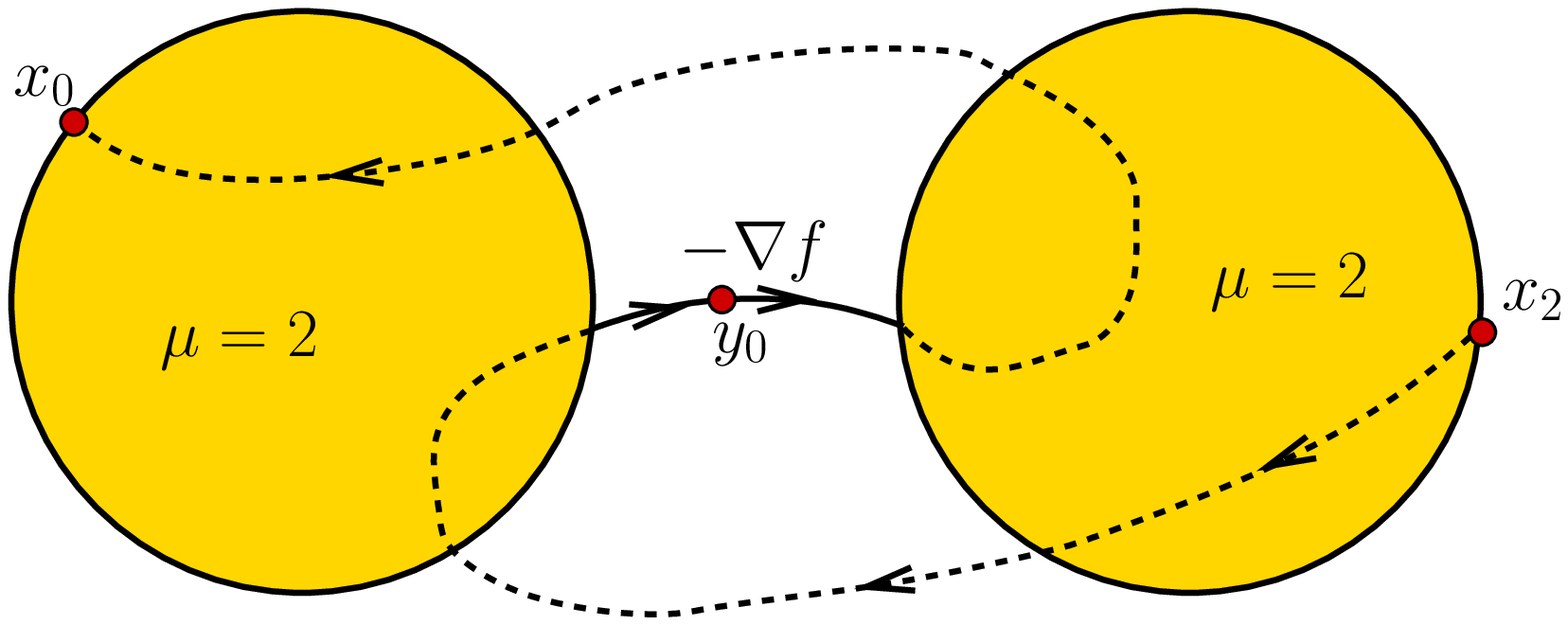, width=0.7\linewidth}
   \end{center}
   \caption{}
   \label{f:coef-s2-2}
\end{figure}

Finally, we can take as generators of $H_{\ast}(T;\Z_{2})\otimes \La$
the critical points of $f$. In this case, we know that $s_{2}$ is the
coefficient of $x_{2}t^{2}$ in the product $x_{0}\circ x_{0}$ (written
now in quantum homology and not at the chain level). In homology the
relation between $x_{0}$ and $y_{0}$ is $y_{0}=x_{0}+\epsilon' x_{2}t$
where $\epsilon'\in\Z_{2}$.  Therefore, $x_{0}\circ y_{0}=x_{0}\circ
x_{0}+\epsilon' x_{0}t$ so that $s_{2}$ coincides with the coefficient of
$x_{2}t^{2}$ in $x_{0}\circ y_{0}$ which is $n_{\Delta}+n_{B}n_{C}$.
\Qed

\begin{ex} We will give here a couple of examples for the invariants
   discussed in this section (see \cite{Bi-Co:qrel-long} for details
   on these calculations).

   a. The Clifford torus,
   $\mathbb{T}^{2}_{\textnormal{clif}}=\{[z_{0}:z_{1}:z_{2}] \in \C
   P^{2} : |z_{0}|=|z_{1}|=|z_{2}|\}$.  In this case we have
   $\alpha=\beta=\gamma'+\gamma''=1$. Therefore, $s_{1}=1$ and
   $s_{2}=1$.

   b. The split torus in $S^{2}\times S^{2}$. This is the split torus
   $Eq\times Eq\subset (S^{2}\times S^{2}, \omega_{S^{2}}\times
   \omega_{S^{2}})$ where $Eq$ is the equator in $S^{2}$. In this
   example, $\alpha=\beta=1$, $\gamma'+\gamma''=0$.  Thus in this case
   $s_{2}$ is not necessarily invariant.
\end{ex}

\subsubsection{Relation to previous works}
An explicit computation of the Floer homology of the Clifford torus
was first carried out by Cho~\cite{Cho:Clifford}. Computations related
to the quantum product for the Clifford torus have been done before by
Cho~\cite{Cho:products} and by Cho and Oh~\cite{Cho-Oh:Floer-toric}
using different methods (see also the recent work of Fukaya, Oh, Ohta
and Ono~\cite{FO3:toric-1}). These works consider Lagrangian tori that
appear as fibres of the moment map in a toric manifold, and the toric
picture plays there a crucial role. It seems likely that these
computations combined with our approach can give rise to more relative
numerical invariants. It would be interesting to see if this leads to
a better understanding of the structure and nature of these relative
invariants.

\section{Applications IV: from quantum structures to Lagrangian intersections}
\label{s:qh-inter} Here we explore the relations between the quantum
operations from~\S\ref{s:alg-struct} associated to two different
Lagrangians $L$ and $L'$. It turns out that a correct composition of
the operations involving $QH(L)$, $QH(L')$ and $QH(M)$ yields
information on intersection properties of $L$ and $L'$. The exposition
presented here is somewhat heuristic in the sense that we ignore quite
a few non-trivial technical difficulties and concentrate only on the
geometric and algebraic pictures. For this reason, some of results
below are marked with a $^*$ to indicate that their proofs are still
not 100\% rigorous. A rigorous treatment of the material of this
section will be pursued in~\cite{Bi-Co:in-prep}. See
also~\cite{Bi-Co:rigidity} for a different approach which is
completely rigorous.

\subsection{Detecting Lagrangian intersections.}
\label{subsec:composition}
Let $L, L' \subset (M, \omega)$ be two monotone Lagrangians with
minimum Maslov numbers $N_L$ and $N_{L'}$. Denote by
$\widetilde{\Lambda}^+_{L}$ and $\widetilde{\Lambda}^+_{L'}$ the
corresponding rings as defined in~\S\ref{sb:coefficients} and let
$\Lambda_{L} = \mathbb{Z}_2[t_0^{-1}, t_0]$, $\Lambda_{L'} =
\mathbb{Z}_2[t_1^{-1}, t_1]$ be the associated Laurent polynomial
rings so that $\deg t_0 = -N_L$ and $\deg t_1 = -N_{L'}$.  Denote by
$\Lambda_{L,L'}$ the ring $\Lambda_{L} \otimes_{\Gamma} \Lambda_{L'}$,
where $\Gamma = \mathbb{Z}_2[s^{-1}, s]$ and $\Lambda_{L}$,
$\Lambda_{L}$ are $\Gamma$-modules by the maps $s \to t_0^{2C_M/N_{L}}
\in \Lambda_{L}$ and $s \to t_1^{2C_M/N_{L'}} \in \Lambda_{L'}$.  The
ring $\Lambda_{L,L'}$ has a grading induced by both factors and it is
easy to see that it is well defined. Note that
\begin{equation} \label{eq:Lambda_L,L'} \Lambda_{L,L'} \cong
   \mathbb{Z}_2[t_0^{-1}, t_1^{-1}, t_0, t_1] / \{ t_0^{2C_M/N_{L}} =
   t_1^{2C_M/N_{L'}} \}.
\end{equation}
The map $q : \widetilde{\Lambda}^+_L \to \Lambda_{L,L'}$ defined by
$q(T^A) = t_0^{\mu(A)/N_{L}}$ turn $\Lambda_{L,L'}$ into a commutative
$\widetilde{\Lambda}^+_{L}$-algebra and similarly $\Lambda_{L,L'}$ is
also a commutative $\widetilde{\Lambda}^+_{L'}$-algebra. According to
the discussion in~\S\ref{sb:coefficients} we can define $QH(L;
\Lambda_{L,L'})$ and $QH(L'; \Lambda_{L,L'})$ as well as $QH(M;
\Lambda_{L,L'})$ and all the theory from~\S\ref{s:alg-struct}
continues to work in this setting.  Note that the identifications
$\Theta$ of $QH(L; \Lambda_{L,L'})$ and $QH(L';\Lambda_{L,L'})$ with
$HF(L,L; \Lambda_{L,L'})$ and $HF(L',L'; \Lambda_{L,L'})$ hold too
since $\Lambda_{L,L'}$ is also a commutative
$\mathbb{Z}_2[H_2^D(M,L)]$-algebra as well as a commutative
$\mathbb{Z}_2[H_2^D(M,L')]$-algebra, both structures being compatible
with the $\widetilde{\Lambda}^+_{L}$ and
$\widetilde{\Lambda}^{+}_{L'}$-algebras structures. See
point~\ref{i:rem-theta} of Remark~\ref{r:non-vanish}
in~\S\ref{sb:coefficients}.

Let $i_L: QH_*(L; \Lambda_{L,L'}) \longrightarrow
QH_*(M;\Lambda_{L,L'})$ be the quantum inclusion map
(see~\S\ref{sb:qinc}) and let $j_{L'}: QH_*(M; \Lambda_{L,L'})
\longrightarrow QH_{*-n}(L'; \Lambda_{L,L'})$ the map defined by
$j_{L'}(a) = a \circledast [L']$.

The following theorem gives information on the composition $$j_{L'}
\circ i_{L}: QH_*(L; \Lambda_{L,L'}) \longrightarrow QH_{*-n}(L';
\Lambda_{L,L'})~.~$$ We denote by $\textnormal{Symp}_H(M, \omega)$ the
group of symplectic diffeomorphisms of $(M, \omega)$ that act as the
identity on $H_*(M)$. Note that we have $\textnormal{Symp}_{H}\supset
\textnormal{Symp}_{0}\supset \textnormal{Ham}$ where
$\textnormal{Symp}_{0}$ is the identity component of the
symplectomorphism group.

\begin{thm} \label{t:j-circ-i} Suppose that there exists $\varphi
   \in \textnormal{Symp}_H(M, \omega)$ such that $L \cap \varphi (L')
   = \emptyset$. Then $j_{L'} \circ i_{L} = 0$.
\end{thm}
A proof of this Theorem appears in~\cite{Bi-Co:rigidity}, based on the
relation between quantum structures and spectral invariants.
In~\S\ref{sb:chain-homotopy} below we will explain a completely
different way to prove this theorem which yields more information on
Lagrangian intersections. Before that, let us present two quick
applications to Lagrangian intersections.
\begin{cor} \label{c:intersection-cpn} Let $L, L' \subset
   {\mathbb{C}}P^n$ be two monotone Lagrangians. If $QH(L) \neq 0$ and
   $QH(L') \neq 0$, then  $L \cap L' \neq \emptyset$.
\end{cor}
This corollary has recently been obtained by Entov and
Polterovich~\cite{En-Po:rigid-subsets, En-Po:private-com}, as well as
by the authors of this paper in~\cite{Bi-Co:rigidity} by completely
different methods based on the the tools developed
in~\S\ref{s:alg-struct} and the theory of spectral numbers for
Hamiltonian diffeomorphisms along the lines mentioned in
\S\ref{subsec:action}, see also \cite{Alb:extrinisic} for earlier
results in this direction.

\begin{proof}[Proof of Corollary~\ref{c:intersection-cpn}]
   As $QH(L), QH(L') \neq 0$ it is easy to see that we also have
   $QH(L;\Lambda_{L,L'}), QH(L'; \Lambda_{L,L'}) \neq 0$.

   Let $f: L \longrightarrow \mathbb{R}$ be a Morse function with one
   minimum $x_0$. Let $\rho$ be a Riemannian metric on $L$ and $J \in
   \mathcal{J}$ an almost complex structure. Denote by $d^{L}$ the
   pearl differential of the complex $\mathcal{C}(L; f, \rho, J)$.

   Notice that although $x_0 \in \mathcal{C}_0(f, \rho, J)$ is a
   ``Morse homology''-cycle it might not be a $d^{L}$-cycle. However,
   an argument based on duality (see~\S\ref{sbsb:duality}) and the
   fact that $QH(L;\Lambda_{L,L'}) \neq 0$ implies that there exist
   $x_j \in \textnormal{Crit}_{jN_{L}}(f)$, $r_j \in \mathbb{Z}_2$,
   for $j \geq 1$ such that
   $$x_0 + \sum_{j \geq 1} r_j x_j t_0^j \in
   \mathcal{C}_{0}(L;f, \rho, J),$$ is a $d^{L}$-cycle.
   (See~\cite{Bi-Co:rigidity} for more details.)

   Denote by $\alpha_0$ the homology class of this element. Consider
   the image of $\alpha_0$ by the canonical map $QH(L) \to
   QH(L;\Lambda_{L,L'})$ induced by the obvious map $\Lambda_{L} \to
   \Lambda_{L,L'}$. We continue to denote this class by $\alpha_0$.
   We will prove below that $j_{L'} \circ i_{L}(\alpha_0) \neq 0$.

   First notice that
   \begin{equation} \label{eq:i_L_0}
      i_{L}(\alpha_0) = [pt] + \sum_{i \geq 1} a_i
      t_0^i,
   \end{equation}
   where $a_i \in H_{i N_{L}}(M;\mathbb{Z}_2)$ and the sum is taken
   over all $0<i$ with $iN_{L} \leq 2n$. The reason for the term
   $[pt]$ comes from the fact the the quantum inclusion extends (on
   the chain level) the classical map induced by the inclusion $L \to
   M$. The fact that there are no $t_1$'s on the righthand side
   of~\eqref{eq:i_L_0} is because $\alpha_0$ is the image of an
   element in $QH(L) = QH(L; \Lambda_L)$.

   Applying the map $j_{L'}$ to~\eqref{eq:i_L_0} we obtain:
   \begin{equation} \label{eq:j-L_1-i_L-0} j_{L'}\circ i_{L}(\alpha_0)
      = [pt] \circledast [L'] + \sum_{i\geq 1} a_i \circledast
      [L']t_0^i.
   \end{equation}
   Now assume by contradiction that $j_{L'}\circ i_{L}(\alpha_0) = 0$.
   Since $[pt] \in QH_0(M; \Lambda_{L,L'})$ is invertible and $[L']
   \neq 0$ (as $QH(L') \neq 0$) it follows that $[pt] \circledast [L']
   \neq 0 \in QH_{-n}(L'; \Lambda_{L,L'})$. Next note that the
   products $[pt] \circledast [L']$ and $a_i \circledast [L']$ on the
   righthand side of~\eqref{eq:j-L_1-i_L-0} both belong to the image
   of $QH(L'; \Lambda_{L'}) \to QH(L'; \Lambda_{L,L'})$. As the sum on
   the righthand side of~\eqref{eq:j-L_1-i_L-0} vanishes it follows
   that there exists an index $i$ that contributes to this sum such
   that $t_0^i = t_1^r$ for some $r \geq 1$. This can happen only if
   $\frac{2C_{{\mathbb{C}}P^n}}{N_{L}} = \frac{2n+2}{N_{L}}$ divides
   $i$. This implies that $2n+2 \mid i N_{L}$, in particular $i N_{L}
   \geq 2n+2$. On the other hand the $i$'s that contribute to the sum
   in~\eqref{eq:i_L_0} (hence in~\eqref{eq:j-L_1-i_L-0}) all satisfy
   $i N_{L} \leq 2n$, a contradiction. This proves that $j_{L'} \circ
   i_{L} \neq 0$. The fact that $L \cap L' \neq \emptyset$ follows now
   from Theorem~\ref{t:j-circ-i}.
\end{proof}

\begin{corspec} \label{c:intersection-quad} Let $L, L' \subset Q$ be two
   Lagrangians in the quadric with $H_1(L;\mathbb{Z})=0$,
   $H_1(L';\mathbb{Z})=0$ and assume that both $L$ and $L'$ are
   relative spin (see~\cite{FO3} for the definition); e.g. both $L$
   and $L'$ are Lagrangian spheres. Then $L \cap L' \neq \emptyset$.
\end{corspec}
The statement of this corollary has been conjectured by Biran
in~\cite{Bi:ECM2004, Bi:ICM2002, Bi:Nonintersections}.

\begin{proof}[Proof of Corollary~\ref{c:intersection-quad}]
   The proof below uses $\mathbb{Z}$ as the ground ring of
   coefficients. As already mentioned in~\S\ref{sbsb:fields} we expect
   our theory to work over $\mathbb{Z}$ however we have not rigorously
   checked that. Still, it is instructive to see how the proof works
   in this framework. Note that under the assumptions of the
   corollary, both $L$ and $L'$ are orientable and relative spin.

   Put $2n = \dim Q$. As the minimal Chern number $C_Q$ of $Q$ is $n$
   we have $N = N_L = N_{L'} = 2n$. It follows that the ring
   $\Lambda_{L,L'} = \Lambda_{L} \otimes_{\Gamma} \Lambda_{L'}$
   coincides with both of $\Lambda_{L}$ and $\Lambda_{L'}$, i.e. it is
   $\mathbb{Z}_2[t^{-1}, t]$, where $\deg t = -2n$. We therefore
   denote all these rings by $\Lambda$ and omit it from the notation.

   As $N = 2n > n+1$ there exists a canonical isomorphism $QH_*(L)
   \cong (H(L; \mathbb{Z}_2) \otimes \Lambda)_*$ and similarly for
   $L'$. Denote by $\alpha_0 \in H_0(L;\mathbb{Z})$ and by $[pt] \in
   H_0(Q;\mathbb{Z}_2)$ the classes of a point. By the results
   of~\cite{Bi-Co:qrel-long, Bi-Co:rigidity} we have $i_L(\alpha_0) =
   [pt] - [Q]t$, and $[pt] \circledast [L'] = -[L']t$. It follows that
   : $$j_{L'} \circ i_L (\alpha_0) = ([pt] - [Q] t)\circledast [L']
   =-2[L'] t \not=0.$$ The result now follows from
   Theorem~\ref{t:j-circ-i}.
\end{proof}

\subsection{A chain homotopy} \label{sb:chain-homotopy} Let $f:L
\longrightarrow \mathbb{R}$, $f': L' \longrightarrow \mathbb{R}$ be
Morse functions, and $\rho_{L}$, $\rho_{L'}$ Riemannian metrics on $L$
and $L'$. Assume that $f'$ has a single maximum, denoted by $x'_n$.
Let $h: M \longrightarrow \R$ be a Morse function and $\rho_M$ a
Riemannian metric on $M$. Finally, let $J \in \mathcal{J}$ be an
almost complex structure. Assume that all these structures are generic
so that the constructions in~\S\ref{s:alg-struct} work.  Put
$\mathcal{F} = (f, \rho_L)$, $\mathcal{F}' = (f', \rho_{L'})$, 
$\mathcal{H} = (h, \rho_M)$.

Given $x \in \textnormal{Crit}(f)$, $y' \in \textnormal{Crit}(f')$ and
$k \in \mathbb{Z}$ consider the space of all tuples $(\mathbf{u}, v,
R, \mathbf{u}')$ such that (see Figure~\ref{f:prl-cyl}):
\begin{enumerate}
  \item There exists $z \in L$, $A \in H_2^D(M,L)$, such that
   $\mathbf{u} \in \mathcal{P}_{\textnormal{prl}}(x,z; A; \mathcal{F},
   J)$.
  \item There exists $z' \in L'$, $A' \in H_2^D(M,L')$, such that
   $\mathbf{u}' \in \mathcal{P}_{\textnormal{prl}}(z',y'; A';
   \mathcal{F}', J)$.
  \item $1 < R < \infty$.
  \item $v: S^1 \times [1,R] \to M$ is a $J$-holomorphic map which
   satisfies $v(S^1 \times 1) \subset L'$, $v(S^1 \times R) \subset L$
   and $v(-1, R) = z$, $v(1, 1) = z'$. Here we view $S^1$ as the unit
   circle in $\mathbb{C}$.
  \item The loop $v(S^1 \times 1)$ is contractible in $M$.
  \item $\mu(A) + \mu(A') + \mu([v]) = k$. Here the Maslov index
   $\mu([v])$ of the cylinder $v$ is defined in an obvious way by
   trivializing $v^*T(M)$ over the cylinder $S^1 \times [1, R]$ and
   computing the difference of Maslov indices of the respective
   Lagrangian loops along the boundaries $S^1 \times 1$ and $S^1
   \times R$. \label{I:cyl-v}
\end{enumerate}
We denote the space of such tuples $(\mathbf{u}, v, R, \mathbf{u}')$
by $\mathcal{P}_{\textnormal{prl-cyl}}(x, y'; k; \mathcal{F},
\mathcal{F}', J)$.

For every cylinder $v$ participating in an element $(\mathbf{u}, v, R,
\mathbf{u}')$ as above we will now associate an element $\tau(v) \in
\Lambda_{L,L'}$ as follows.  Pick $1 < r_0 < R$ and choose a disk $Q$
(in $M$) spanning the loop $v(S^1 \times r_0)$ (recall that the this
loop is assumed to be contractible in $M$).  By ``dissecting'' the
cylinder $v$ along the loop $v(S^1 \times r_0)$ we obtain two tubes,
$T = v|_{S^1 \times [1,r_0]}$ and $T' = v_{S^1 \times [r_0, R]}$, one
with a boundary component on $L$ and the other with a boundary
component on $L'$.  By gluing $Q$ to $T$ and $\overline{Q}$ to $T'$
($\overline{Q}$ is $Q$ with reversed orientation) we now obtain two
disks $w$ and $w'$ with boundaries on $L$ and $L'$ respectively.
($\overline{Q}$ stands for $Q$ with reversed orientation.) Obviously
we have $\mu([w])+\mu([w'])=\mu([v])$. We define $$\tau(v) =
t_0^{\mu([w])/N_L} t_1^{\mu([w'])/N_{L'}} \in \Lambda_{L,L'}.$$ It
follows from the definition of the ring $\Lambda_{L,L'}$ (see
also~\eqref{eq:Lambda_L,L'} that the element $\tau(v)$ does not depend
on the choice of $r_0$ and the spanning disk $Q$.

\begin{figure}[htbp]
      \psfig{file=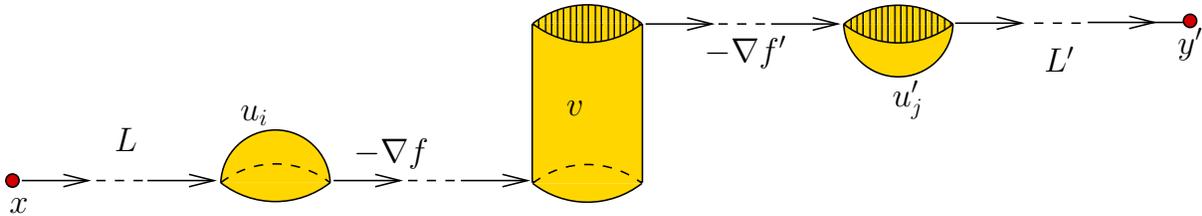, width=1 \linewidth}
      \caption{An element of the space 
        $\mathcal{P}_{\textnormal{prl-cyl}}
        (x, y'; k; \mathcal{F},\mathcal{F}', J)$}
      \label{f:prl-cyl}
\end{figure}

We will need yet another moduli space which is defined as follows.
Consider the space of all tuples $(\mathbf{u}, z, z', \mathbf{u'})$
such that (see Figure~\ref{f:prl-grad}) :
\begin{enumerate}
  \item There exists $A \in H_2^D(M,L)$ such that $\mathbf{u} \in
   \mathcal{P}_{\textnormal{prl}}(x, z; A; \mathcal{F}, J)$.
  \item There exists $A' \in H_2^D(M,L)$ such that $\mathbf{u}' \in
   \mathcal{P}_{\textnormal{prl}}(z', y'; A'; \mathcal{F}', J)$.
  \item There exists $t>0$ such that $\Phi_t(z) = z'$, where $\Phi_t$
   is the negative gradient flow of the Morse function $h$ with
   respect to $\rho_M$.
  \item $\mu(A) + \mu(A') = k$.
\end{enumerate}
We denote the space of such tuples by
$\mathcal{P}_{\textnormal{prl-grad}}(x,y'; k; \mathcal{F},
\mathcal{F}', \mathcal{H}, J)$. The virtual dimension of both moduli
spaces $\mathcal{P}_{\textnormal{prl-cyl}}$ and
$\mathcal{P}_{\textnormal{prl-grad}}$ is: $$\delta(x, y'; k) =
|x|-|y'|-n+1 + k.$$
\begin{figure}[htbp]
   \psfig{file=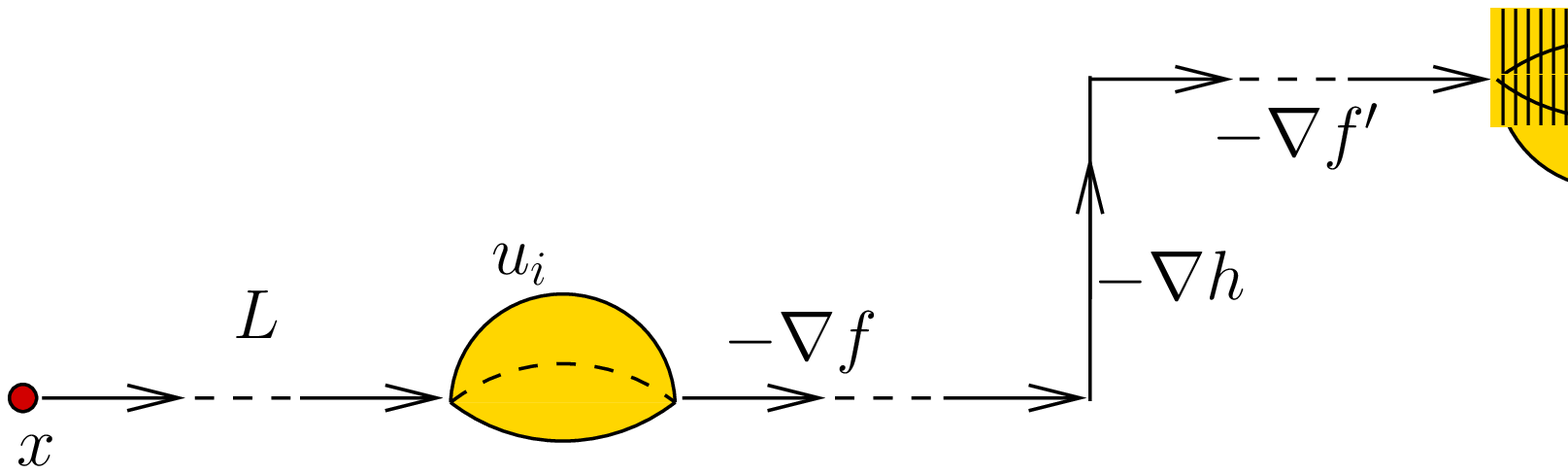, width=1 \linewidth}
   \caption{An element of the space
     $\mathcal{P}_{\textnormal{prl-grad}} (x, y'; k;
     \mathcal{F},\mathcal{F}', \mathcal{H}, J)$}
   \label{f:prl-grad}
\end{figure}

Define a morphism $\Phi_{L, L'}: \mathcal{C}_{*}(\mathcal{F}, J;
\Lambda_{L,L'}) \longrightarrow \mathcal{C}_{*-n+1}(\mathcal{F}', J;
\Lambda_{L,L'})$ by:
$$\Phi_{L,L'}(x) = \sum_{k, y'} 
\Bigl( \sum_{(\mathbf{u}, v, R, \mathbf{u}')} y' \,
t_0^{\mu(\mathbf{u})/N_L} \tau(v) t_1^{\mu(\mathbf{u}')/N_{L'}} +
\sum_{(\mathbf{u}, z,z', \mathbf{u}')} y' \, t_0^{\mu(\mathbf{u})/N_L}
t_1^{\mu(\mathbf{u}')/N_{L'}}\Bigr),$$ where the first sum is taken
over all $k \in \mathbb{Z}$ and $y' \in \textnormal{Crit}(f')$ with
$\delta(x,y',k) = 0$; the second sum is taken over all $(\mathbf{u},
v, R, \mathbf{u}') \in \mathcal{P}_{\textnormal{prl-cyl}}(x, y'; k;
\mathcal{F}, \mathcal{F}', J)$; the third sum is taken over all
$(\mathbf{u}, z,z', \mathbf{u}') \in
\mathcal{P}_{\textnormal{prl-grad}}(x,y'; k; \mathcal{F},
\mathcal{F}', \mathcal{H}, J)$.

Denote by $\widetilde{i}_L: \mathcal{C}_*(\mathcal{F}, J; \Lambda_{L,L'})
\longrightarrow C_*(\mathcal{H}; \Lambda_{L,L'})$ the quantum inclusion
map (on the chain level) as defined by~\eqref{eq:inc}
in~\S\ref{sb:qinc}. The induced map in homology is $i_L$.  Denote by
$\widetilde{j}_{L'} : C_*(\mathcal{H}; \Lambda_{L,L'}) \longrightarrow
\mathcal{C}_{*-n}(\mathcal{F}', J; \Lambda_{L,L'})$ the chain map defined
by $\widetilde{j}_{L'}(a) = a \circledast x'_n$ (recall that $x'_n$ is
the single maximum of $f'$). Again, the induced map in homology is
$j_{L'}$. For simplicity we denote the differentials of the complexes
$\mathcal{C}(\mathcal{F}, J;\Lambda_{L,L'})$ and
$\mathcal{C}(\mathcal{F}', J;\Lambda_{L,L'})$ by $d$ and $d\,'$
respectively.

Theorem~\ref{t:j-circ-i} follows from the following.
\begin{thmspec} \label{t:chain-homotopy} Suppose that $L \cap L' =
   \emptyset$. Then the following identity holds:
   $$\widetilde{j}_{L'} \circ \widetilde{i}_L = 
   \Phi_{L,L'} \circ d + d\,' \circ \Phi_{L,L'}.$$ In other words, the
   chain map $\widetilde{j}_{L'} \circ \widetilde{i}_L$ is null
   homotopic. In particular the induced map in homology $j_{L'} \circ
   i_L$ vanishes.
\end{thmspec}

\begin{rem} A map similar to $\Phi_{L,L'}$ has been discussed before
   in the context of the cluster complex in \cite{Cor-La:Cluster-1}
   but it was used there to define a chain morphism (under certain
   assumptions) and not a chain homotopy.  For some related earlier
   constructions see \cite{Ga-La:holcyl}.
\end{rem}

Note that for $\varphi \in \textnormal{Symp}_H(M, \omega)$ the map
$j_{L'} \circ i_L$ vanishes iff $j_{\varphi(L')} \circ i_L$ vanishes.
Therefore in proving Theorem~\ref{t:j-circ-i} there is no loss of
generality in assuming that $L \cap L' = \emptyset$ rather than $L
\cap \varphi(L') = \emptyset$.

\subsubsection{Main ideas of the proof of Theorem~\ref{t:chain-homotopy}}
\label{sbsb:prf-chain-homotopy}
In essence the proof follows the same standard scheme in Morse-Floer
theory, as described in~\S\ref{s:main-ideas-proof}, i.e.
compactifying certain $1$-dimensional moduli spaces and deriving
identities by counting the number of points in their boundaries.

Here is a more detailed account of the arguments. We need to introduce
another type of moduli space.  Denote by $\Phi^f_t$, $\Phi^{f'}_t$ and
$\Phi^h_t$ the negative gradient flows of $(f, \rho_L)$, $(f',
\rho_{L'})$ and $(h, \rho_{M})$ respectively. Let $x \in
\textnormal{Crit}(f)$, $y' \in \textnormal{Crit}(f')$, and $k \in
\mathbb{Z}$. Consider the space of all pairs $(\mathbf{u},
\mathbf{u}')$ where (see figure~\ref{f:prl-disk-grad-disk}):
\begin{enumerate}
  \item $\mathbf{u} = (u_1, \ldots, u_l)$, $\mathbf{u}' = (u'_1,
   \ldots, u'_{l'})$ are two sequences of $J$-holomorphic disks $u_i:
   (D, \partial D) \longrightarrow (M, L)$, $u'_j: (D, \partial D)
   \longrightarrow (M, L')$. The disks $u_1, \ldots, u_{l-1}$ and
   $u'_2, \ldots, u'_{l'}$ are non-constant.
  \item $u_1(-1) \in W_{x}^{u}(f)$, $u'_{l'}(1) \in W_{s}^{y'}(f')$.
  \item For every $1 \leq i \leq l-1$ there exists $0< t_i < \infty$
   such that $\Phi^f_{t_i}(u_i(1)) = u_{i+1}(-1)$. For every $2 \leq j
   \leq l' $ there exists $0< \tau_j < \infty$ such that
   $\Phi^{f'}_{\tau_j}(u'_{j-1}(1)) = u'_j(-1)$.
  \item There exists $0<t<\infty$ such that $\Phi^h_t(u_l(0)) = u'_1(0)$.
  \item $\mu([\mathbf{u}]) + \mu([\mathbf{u}']) = k$.
\end{enumerate}
We quotient the space of such elements by the obvious
reparametrization groups. The resulting space is denoted by
$\mathcal{P}_{\textnormal{prl-prl}}(x, y'; k; \mathcal{F},
\mathcal{F}', \mathcal{H}, J)$. Its virtual dimension is $\delta(x,
y'; k) = |x|-|y'|-n+1+k$.
\begin{figure}[htbp]
      \psfig{file=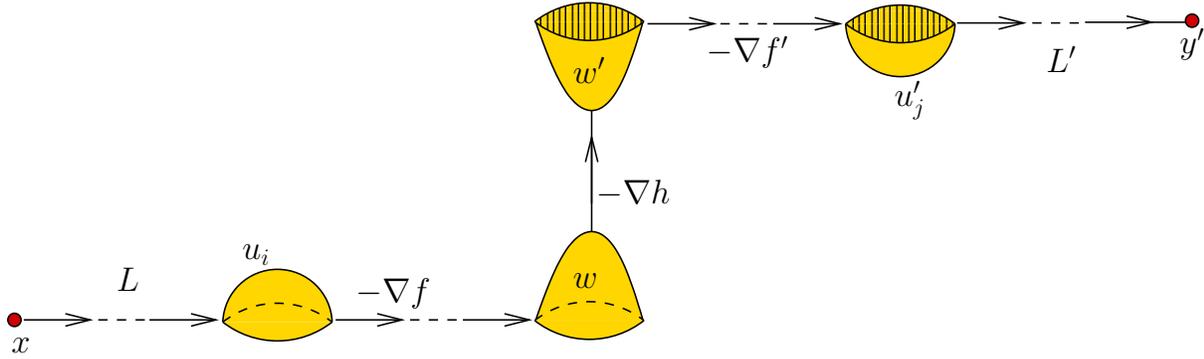, width=1 \linewidth}
      \caption{An element of the space
        $\mathcal{P}_{\textnormal{prl-prl}}(x, y'; k; \mathcal{F},
        \mathcal{F}', \mathcal{H}, J)$}
      \label{f:prl-disk-grad-disk}
\end{figure}

Let $x \in \textnormal{Crit}(f)$, $y' \in \textnormal{Crit}(f')$, $k_0
\in N_{L}\mathbb{Z}$ and $k_1 \in N_{L'}\mathbb{Z}$ with $|x|-|y'|-n +
k_0 + k_1=0$. In order to prove the chain homotopy formula in
Theorem~\ref{t:chain-homotopy} we have to show that the coefficient of
$y' t_0^{k_0/N_{L}}t_1^{k_1/N_{L'}}$ in $\widetilde{j}_{L'} \circ
\widetilde{i}_L(x) - (\Phi_{L,L'}\circ d (x) + d\,' \circ \Phi_{L,L'}
(x))$ vanishes. For this end, put $k = k_0 + k_1$ and consider the
$1$-dimensional moduli spaces
$\mathcal{P}_{\textnormal{prl-cyl}}(x,y';k; \mathcal{F}, \mathcal{F},
J)$, $\mathcal{P}_{\textnormal{prl-grad}}(x,y';k; \mathcal{F},
\mathcal{F}, \mathcal{H}, J)$ and
$\mathcal{P}_{\textnormal{prl-prl}}(x,y';k; \mathcal{F}, \mathcal{F},
\mathcal{H}, J)$.

The compactifications of these moduli spaces goes along the same lines
as in~\S\ref{sb:comp-glue} with the following additional types of
boundary points:
\begin{enumerate}[a.]
  \item The gradient trajectory of $h$ involved in
   $\mathcal{P}_{\textnormal{prl-prl}}(x,y';k;\mathcal{F},
   \mathcal{F}', \mathcal{H}, J)$ or in \\
   $\mathcal{P}_{\textnormal{prl-grad}}(x,y';k; \mathcal{F},
   \mathcal{F}', \mathcal{H}, J)$ may break at a critical point of
   $h$.
  \item A gradient trajectory of $h$ involved in
   $\mathcal{P}_{\textnormal{prl-prl}}(x,y';k;\mathcal{F},
   \mathcal{F}', \mathcal{H}, J)$ may shrink to a point. Note that
   this cannot happen for $\mathcal{P}_{\textnormal{prl-grad}}(x,y';k;
   \mathcal{F}, \mathcal{F}', \mathcal{H}, J)$ since $L$ and $L'$ are
   assumed to be disjoint.
  \item The parameter $R$ in elements $(\mathbf{u}, v, R, \mathbf{u}')
   \in \mathcal{P}_{\textnormal{prl-cyl}}(x,y';k;\mathcal{F},
   \mathcal{F}',J)$ goes to $\infty$. The limit of the cylinder $v$ in
   this case is two $J$-holomorphic disks $w$ and $w'$, one with
   boundary on $L$ and one with boundary on $L'$, attached to each
   other at an interior point. See figure~\ref{f:prl-disk-disk}. Note
   that the other type of degeneration $R \to 1$ is impossible here
   because $L \cap L' = \emptyset$.
  \item Bubbling of a $J$-holomorphic disk coming from the cylinder
   $v$ either with boundary on $L$ or with boundary on $L'$. Note that
   bubbling of a $J$-holomorphic sphere from $v$ may occur in general,
   but not in our case since we consider only $1$-dimensional moduli
   spaces and such a bubbling would decrease the dimension to a
   negative one.
\end{enumerate}
\begin{figure}[htbp]
   \psfig{file=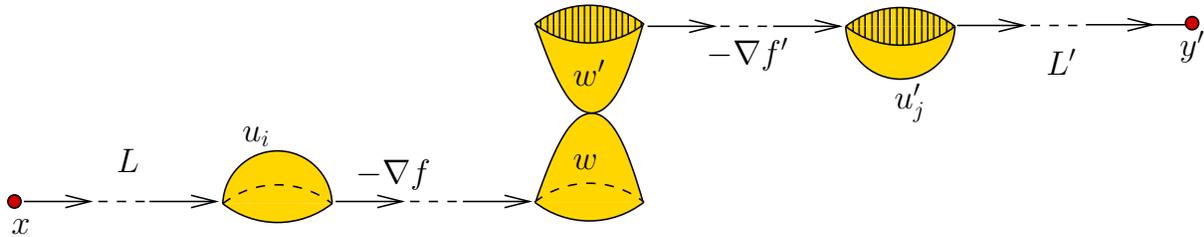, width=1 \linewidth}
   \caption{A $J$-holomorphic cylinder that converged to two disks}
   \label{f:prl-disk-disk}
\end{figure}

The above together with gluing arguments would then lead to a
compactification of the $1$-dimensional spaces
$\mathcal{P}_{\textnormal{prl-prl}}$,
$\mathcal{P}_{\textnormal{prl-grad}}$,
$\mathcal{P}_{\textnormal{prl-cyl}}$ into compact $1$-dimensional
manifolds with boundary.  The identities needed to prove the homotopy
formula in Theorem~\ref{t:chain-homotopy} would then follow by
counting the number of points in the boundaries of these spaces. Note
that since we want to show the vanishing of the coefficient of the
monomial $y'\,t_0^{k_0/N_L}t_1^{k_1/N_{L'}}$ we actually have to
restrict here only to those components of the spaces
$\mathcal{P}_{\textnormal{prl-cyl}}(x,y';k; \mathcal{F}, \mathcal{F},
J)$, $\mathcal{P}_{\textnormal{prl-grad}}(x,y';k; \mathcal{F},
\mathcal{F}, \mathcal{H}, J)$,
$\mathcal{P}_{\textnormal{prl-prl}}(x,y';k; \mathcal{F}, \mathcal{F},
\mathcal{H}, J)$ that contribute to this monomial. (In general, these
spaces might contribute to other monomials of the type
$y'\,t_0^{k'_0}t_1^{k'_1}$ with $k'_0 + k'_1 = k$, that are different
than $y'\,t_0^{k_0}t_1^{k_1}$ in the ring $\Lambda_{L,L'}$.)

Another point which should be kept in mind within these arguments is
that when $\delta_{\textnormal{mod}}(a,x'_n,y';A')=0$ and $\mu(A') >
0$, every element $(u_1, \ldots, u_l; r) \in
\mathcal{P}_{\textnormal{mod}}(a,x'_n, y';A; \mathcal{H},
\mathcal{F}',J)$ (see~\S\ref{sb:qmod}) must have $r=1$, i.e. the disk
with three marked points must be the first one. This follows from a
straightforward transversality argument.

\medskip

On the technical side, most of the steps of the proof indicated above
can be carried out by essentially standard analytic techniques
(versions of well known compactness theorems and gluing procedures).
The only issue which remains to be rigorously clarified is how to
achieve transversality for the spaces
$\mathcal{P}_{\textnormal{prl-cyl}}$, in particular for holomorphic
cylinders. The difficulty occurs in the presence of holomorphic
cylinders that are not somewhere injective.

Another approach which overcomes the transversality difficulties is to
perturb the Cauchy-Riemann equation for the cylinders via Hamiltonian
perturbations in the spirit of~\cite{Ak-Sa:Loops} (see also Chapter~8
of~\cite{McD-Sa:Jhol-2}). An even more natural type of perturbation
is to replace the Morse
function $h: M \to \mathbb{R}$ by a generic Hamiltonian function $H:M
\times S^1 \to \mathbb{R}$. One can also replace the almost complex
structure $J$ by a time dependent one (though this is not really
necessary here). The module action of $QH(M)$ on $QH(L')$ would now be
replaced by the equivalent action of $HF(H)$ on $QH(L') \cong
HF(L',L')$ (Here $HF(H)$ stands for the periodic-orbit Floer homology
of $H$). The cylinder $v$ would now satisfy a Floer-type equation
(involving the Hamiltonian vector field of $H$), with Lagrangian
boundary conditions etc.  With these replacements the scheme of the
proof presented above goes through with minor modifications.  This
approach will be further explored and developed
in~\cite{Bi-Co:in-prep}.

\subsection{Further generalizations} \label{sb:further-homotopy} When
$L \cap L' \neq \emptyset$ the homotopy formula in
Theorem~\ref{t:chain-homotopy} does not hold in general. The reason is
that there are more types of boundary points for the spaces
$\mathcal{P}_{prl-cyl}$ etc. than described in the preceding
subsection. For example, when $L \cap L' \neq \emptyset$ it is
possible to have a sequence $(\mathbf{u}_{\nu}, v_{\nu}, R_{\nu},
\mathbf{u}'_{\nu}) \in
\mathcal{P}_{\textnormal{prl-cyl}}(x,y';k;\mathcal{F},
\mathcal{F}',J)$ with $R_{\nu} \longrightarrow 1$. A compactness
argument shows that (generically) the limit of the cylinders $v_{\nu}$
would look like a cylinder in which an arc connecting its two boundary
components degenerated to a point $p$ which lies in $L \cap L'$.
Analytically, this limit object can also be viewed as a
$J$-holomorphic strip with one boundary on $L$ and one on $L'$
connecting the point $p$ with itself, i.e. a Floer connecting
trajectory going from $p$ to $p$.

By analyzing all the other possible boundary points for the relevant
moduli spaces we obtain a correction term in the homotopy formula
which takes into account information related to $L \cap L'$. More
precisely, put $\mathcal{R} = \Lambda_{L,L'}$. Then we have:
$$\widetilde{j}_{L'} \circ \widetilde{i}_L - \widetilde{\chi}_{L,L'}=
\Phi_{L,L'} \circ d + d\,' \circ \Phi_{L,L'},$$ where
$\widetilde{\chi}_{L,L'}: \mathcal{C}_*(\mathcal{F}; \mathcal{R})
\longrightarrow \mathcal{C}_{*-n}(\mathcal{F}'; \mathcal{R})$ is a
chain map which is the composition of two chain morphisms:
$$\widetilde{\chi}_{L,L'}: \mathcal{C}(\mathcal{F};\mathcal{R})
\to CF(L,L'; \mathcal{R})^{\ast}\otimes_{\mathcal{R}}
CF(L,L';\mathcal{R}) \otimes_{\mathcal{R}} \mathcal{C}(\mathcal{F}';
\mathcal{R}) \xrightarrow[]{\langle - , - \rangle \otimes
  id}\mathcal{C}(\mathcal{F}'; \mathcal{R})$$ where $\langle - , -
\rangle : A^{\ast}\otimes_{\mathcal{R}} A\to \mathcal{R}$ is the usual
pairing of $A^{\ast}=Hom_{\mathcal{R}}(A,\mathcal{R})$ and $A$.  (Note
that we have to assume here that $N_L, N_{L'} \geq 3$ and that
$\pi_1(L), \pi_1(L')$ both have torsion images in $\pi_1(M)$ in order
for $HF(L,L')$ to be well defined and invariant.)

An immediate corollary of this is that if $j_{L'} \circ i_{L} \neq 0$
then $HF(L,L') \neq 0$. Similarly we can strengthen
Corollaries~\ref{c:intersection-cpn} and~\ref{c:intersection-quad} to
conclude that not only $L \cap L' \neq \emptyset$ but also that
$HF(L,L') \neq 0$. This direction will be further pursued
in~\cite{Bi-Co:in-prep}.

%\bibliography{/home/biran/latex/general/bibliography}
%\bibliography{bibliography}

\end{document}